\documentclass[10 pt]{amsart}

\usepackage{amsmath}
\usepackage{amsfonts}
\usepackage{amssymb}
\usepackage{amstext}
\usepackage{amsbsy}
\usepackage{amsopn}
\usepackage{amsthm}
\usepackage{amsxtra}
\usepackage{graphicx}
\usepackage{caption}
\usepackage{subcaption}
\usepackage{color}
\usepackage{hyperref}
\usepackage{enumerate}
\usepackage{amscd}
\usepackage{tikz}
\usepackage[super]{nth}
\usepackage{xy} \xyoption{all}
\usepackage{verbatim}
\usepackage{enumitem}
\usepackage{stackengine}
\usepackage{mathtools}
\usepackage{adjustbox}
\usepackage{float}
\usepackage{bm}
\usepackage{tensor}

\usetikzlibrary{arrows}
\usetikzlibrary{arrows.meta,decorations.pathreplacing,patterns}

\newtheorem{theorem}{Theorem}[section]
\newtheorem{lemma}[theorem]{Lemma}
\newtheorem{corollary}[theorem]{Corollary}
\newtheorem{proposition}[theorem]{Proposition}

\newtheorem{question}[theorem]{Question}

\newtheorem*{conjecture*}{Conjecture}
\newtheorem*{corollary*}{Corollary}
\newtheorem*{claim*}{Claim}
\newtheorem*{theorem*}{Theorem}

\theoremstyle{remark}
\newtheorem{remark}[theorem]{Remark}
\theoremstyle{definition}
\newtheorem{definition}[theorem]{Definition}
\newtheorem*{definition*}{Definition}

\newtheorem*{notation*}{Notation}
\newtheorem{example}[theorem]{Example}

\newcommand{\CG}{\mathcal{G}}
\newcommand{\A}{\mathcal{A}}

\newcommand{\Z}{\mathbb{Z}}
\newcommand{\Q}{\mathbb{Q}}
\newcommand{\N}{\mathbb{N}}
\newcommand{\BL}{\mathbb{S}}

\newcommand{\CA}{\mathcal{A}}

\newcommand{\CL}{\mathcal{L}}
\newcommand{\CR}{\mathcal{R}}
\newcommand{\CW}{\mathcal{W}}

\newcommand{\GL}{\mathrm{GL}}
\newcommand{\SL}{\mathrm{SL}}
\newcommand{\ab}{\mathrm{ab}}

\newcommand{\id}{\mathrm{Id}}

\newcommand{\diag}{\mathrm{diag}}
\newcommand{\rts}{\mathfrak{R}}

\newcommand{\aut}{\textnormal{Aut}}
\newcommand{\inert}{\textnormal{Inert}}
\newcommand{\simp}{\textnormal{Simp}}

\newcommand{\homeo}{\textnormal{Homeo}}
\newcommand{\bij}{\textnormal{Bijection}}
\newcommand{\sym}{\textnormal{Sym}}
\newcommand{\alt}{\textnormal{Alt}}
\newcommand{\topo}{\textnormal{top}}
\newcommand{\ev}{\textnormal{ev}}
\newcommand{\autinf}[1]{\textnormal{Aut}^{(\infty)}(\sigma_{#1})}
\newcommand{\autk}[2]{\textnormal{Aut}^{(#1)}(\sigma_{#2})}
\newcommand{\inertinf}[1]{{\inert}^{(\infty)}(\sigma_{#1})}

\newcommand{\inv}[1]{\textnormal{Inv}(\sigma_{n})}
\newcommand{\infinv}[1]{\textnormal{Inv}^{\infty}(\sigma_{n})}
\newcommand{\Homeo}{{\rm Homeo}}
\newcommand{\im}{{\rm Im}}
\newcommand{\Id}{{\rm Id}}

\newcommand{\simpinf}[1]{\textnormal{Simp}^{(\infty)}(\Gamma_{#1})}
\newcommand{\simpinfev}[1]{\textnormal{Simp}_{\ev}^{(\infty)}(\Gamma_{#1})}
\newcommand{\etwo}[2]{{\scriptsize \begin{pmatrix} #1 \\ #2 \end{pmatrix}}}
\newcommand{\etwoid}[2]{{\scriptsize \textnormal{$\id$ on } \begin{pmatrix} #1 \\ #2 \end{pmatrix}}}
\newcommand{\smalldots}{\ldots}

\title{The stabilized automorphism group of a subshift}
\author{Yair Hartman}
\address{Ben-Gurion University of the Negev, Israel}
\email{hartmany@bgu.ac.il}
\author{Bryna Kra}
\address{Northwestern University, Evanston, IL 60208 USA}
\email{kra@math.northwestern.edu}
\author{Scott Schmieding}
\address{University of Denver, Denver, CO 80208 USA}
\email{scott.schmieding@du.edu}

\thanks{YH was partially supported by ISF grant 1175/18, BK by NSF grant 1800544, and SS by NSF grant 1502643.}

\begin{document}
\begin{abstract}
For a mixing shift of finite type, the associated automorphism group has a rich algebraic structure,
and yet we have few criteria to distinguish when two such groups are isomorphic. We introduce a stabilization of the automorphism group, study its
 algebraic properties, and use them to distinguish many of the stabilized automorphism groups. We also show that for a full shift, the subgroup of the stabilized automorphism group generated by elements of finite order is simple, and that the stabilized automorphism group is an extension of a free abelian group of finite rank by this simple group.
\end{abstract}

\maketitle
\setcounter{tocdepth}{1}
\section{Distinguishing automorphism groups}
\subsection{Automorphism groups and stabilized automorphism groups}
Let $(X, \sigma)$ be a shift over a finite alphabet $\A$,  that is, $X\subset\A^\Z$  is closed and invariant under the left shift $\sigma\colon \A^\Z\to \A^\Z$.   The automorphism group $\aut(X, \sigma)$ of the shift is the collection of homeomorphisms $\phi\colon X\to X$ such that $\phi\circ\sigma = \sigma\circ\phi$. For many shifts with complicated dynamical behavior, including any mixing shift of finite type, the associated automorphism group is known to have a rich algebraic structure, for example containing isomorphic copies of any finite group,  the countably infinite direct sum of copies of $\Z$, and the free group on two generators (see~\cite{H, BLR}). In contrast to shifts of finite type, numerous results show that for many zero entropy shifts, the automorphism group is more constrained (see for example~\cite{CK1,DDMP,CK2}).

In spite of much attention, several natural and simple to state questions remain open.  Boyle, Lind, and Rudolph~\cite{BLR} raised the question of distinguishing (up to isomorphism) the automorphism groups of full shifts $(X_n, \sigma_n)$ for various $n$ (meaning $X_n = \A^\Z$ and the alphabet $\A$ has $n$ symbols). They ask if the automorphism group of the full shift on $2$ symbols is isomorphic to the automorphism group of the full shift on $3$ symbols, and more generally, for which $p$ and $q$ the groups $\aut(X_p,\sigma_p)$ and $\aut(X_q,\sigma_q)$ are isomorphic as groups. For some choices of $p$ and $q$, such as when $q=p^2$ for a prime $p$, one can show that the associated automorphism groups are not isomorphic (this was explicitly pointed out for $2$ and $4$ in~\cite{BLR}, and we
make note in Theorem~\ref{th:rootdistinguish} of the natural generalization using their method).  But for general $p$ and $q$, this problem remains open.

While many groups are known to embed into the automorphism group of a shift of finite type,
the subgroup structure of the automorphism groups can not be used to distinguish them, as
shown by a result of Kim and Roush~\cite{KR90}.  Namely, they
showed the automorphism group of any full shift can be embedded into the automorphism group of any other full shift (in fact, it can be embedded into the automorphism group of any mixing shift of finite type). Thus any strategy for distinguishing two automorphism groups relying on finding some subgroup of one that does not lie in the other,
 must fail.

Taking a new approach to this problem, we define a certain stabilization of the automorphism group, and show that many of these stabilized groups can be distinguished (up to isomorphism) based only on the alphabet size. To simplify notation, we suppress the associated space in the notation for the automorphism group,
writing $\aut(\sigma)$ instead of $\aut(X, \sigma)$. For a subshift $(X, \sigma)$, we define the {\em stabilized automorphism group $\aut^{(\infty)}(\sigma)$} to be
$$
\aut^{(\infty)}(\sigma) = \bigcup_{k=1}^\infty\aut(\sigma^k).
$$
Passing from the non-stabilized automorphism group to the stabilized setting offers certain advantages, and some of our results are analogous to what happens in the realm of algebraic K-theory. Given a ring $\mathcal{R}$, one defines the stabilized general linear group $GL(\mathcal{R})$ by taking the union of the finite general linear groups $GL_{n}(\mathcal{R})$. An important subgroup of $GL_{n}(\mathcal{R})$ is $E_{n}(\mathcal{R})$, the subgroup generated by elementary matrices (matrices which differ from the identity in at most one coordinate), and in 1950, Whitehead proved that, upon stabilizing, the commutator of $GL(\mathcal{R})$ coincides with the stabilized subgroup of elementary matrices $E(\mathcal{R})$. One way to interpret this result is that, by stabilizing, a certain abstract subgroup which is defined group-theoretically (in this case the commutator) may be identified with a concrete naturally occurring subgroup: the group of stabilized elementary matrices. In our setting, stabilizing produces analogous results. While the commutator of $\aut(\sigma)$ is not very well understood, we prove in Theorem~\ref{thm:stablecommutator} that, at the stabilized level, the abelianization of $\aut^{(\infty)}(\sigma)$ coincides with the abelianization of a certain explicit quotient of $\aut^{(\infty)}(\sigma)$: the dimension representation (see Section~\ref{sec:stabdimrep} for definitions). Thus in many cases (e.g. when $(X,\sigma)$ is a full shift), the commutator subgroup of $\aut^{(\infty)}(\sigma)$ coincides with a certain naturally occurring subgroup (the subgroup of stabilized inert automorphisms).

Illustrating the stronger tools available in the stabilized setting, we are able to distinguish many stabilized automorphism groups for which there are currently no techniques to distinguish the (non-stabilized) counterparts. In particular, in Section~\ref{sec:commutator}, we show that the stabilized automorphism groups of full shifts on alphabets with different numbers of prime factors can not be isomorphic:
\begin{theorem}
\label{thm:main-distinguish}
Assume that $(X_m, \sigma_m)$ and $(X_n, \sigma_n)$ are the
full shifts on $m$ and $n$ symbols for some integers $m,n\geq 2$ and assume
that the stabilized automorphism group $\autinf m$ on $m$ symbols and
the stabilized automorphism group $\autinf n$ on $n$ symbols are isomorphic.
Then $m$ and $n$ have the same number of distinct prime divisors.
\end{theorem}
In particular, this means that the stabilized automorphism groups on $2$ symbols and $6$ symbols are not isomorphic; the analog of this result for the (non-stabilized) automorphism groups on $2$ and on $6$ symbols remains open.
However, our results do not distinguish the stabilized automorphism groups with $2$ and $3$ symbols, nor those with $6$ and $12$ symbols, and another method is needed to address this question (see Question~\ref{conj:stable}).

In Section~\ref{sec:stables}, we prove various properties of the stabilized automorphism group, and compare them to the (non-stabilized) automorphism group of the shift.
It is easy to check that, as for the automorphism group, the stabilized automorphism group is countable. We also prove that, like the automorphism group, the stabilized automorphism group is not finitely generated;
in contrast, though,
 the proof is quite different from the proof for the non-stabilized case.

However, differences between the (non-stabilized) automorphism group and the stabilized group appear quickly. For example, while Ryan's Theorem~\cite{R1, R2} states that the center of the automorphism is exactly the powers of the shift, in Proposition~\ref{prop:trivial-center} we show that the stabilized automorphism group has a trivial center.

A mixing shift of finite type $(X_{A},\sigma_{A})$ has
a dense set of periodic points, and as a result, the action of the
automorphism group on $X_A$ is far from minimal, and has many invariant measures.  However, it follows from a result of Boyle, Lind, and Rudolph~\cite{BLR} that the $\autinf{A}$-action  on the space $X_A$ is
minimal and uniquely ergodic. We discuss this in Section~\ref{sec:autinfaction}.

An important tool for studying $\aut(\sigma)$ when $(X,\sigma)$ is a shift of finite type is the dimension representation, a certain homomorphism from $\aut(\sigma)$ to the group of automorphisms of an ordered abelian group associated to $(X,\sigma)$. The kernel of this dimension representation, known as the subgroup of inert automorphisms, is a large, algebraically rich subgroup of $\aut(\sigma)$: for example, in the case of a full shift, the automorphism group is an extension of a finitely generated free abelian group by the inert subgroup.  However, in general the inert subgroup is not  well understood. In Section~\ref{sec:stables} we show the dimension representation extends naturally to a stabilized dimension representation, and that the abelianization of the group $\aut^{(\infty)}(\sigma)$ factors through this stabilized dimension representation. Similar to the non-stabilized group $\aut(\sigma)$, the kernel of the stabilized dimension representation, which we refer to as the group of stabilized inerts, constitutes the core combinatorial part of $\aut^{(\infty)}(\sigma)$. In the classical (non-stabilized) setting, the
inert subgroup $\inert(\sigma) \subset \aut(\sigma)$ is residually finite, and hence far from simple. In stark contrast to this, in Section~\ref{sec:simplicity},
we prove:
\begin{theorem}
\label{th:even-simple}
For any $n \ge 2$, the group of stabilized inert automorphisms of the full shift $(X_{n},\sigma_{n})$ is simple.
\end{theorem}

In some sense, the stabilized automorphism groups capture different information about the shift system than the non-stabilized automorphism groups.
For example, the stabilized automorphism groups for the full shift on $2$ symbols and on $4$ symbols are isomorphic, whereas for the automorphism groups this is essentially the only case in which these groups can be distinguished.  However,
there is often an advantage in working with a stabilized object
involving sufficiently high powers of the transformation, rather than the original object.  Examples of success in solving problems in the stabilized setting, but which are still open in the non-stabilized setting, are
Wagoner's Finite Order Generation Theorem~\cite{WagonerEFOG} for stabilized inert automorphisms,  the classification~\cite{williams, KR79} of shifts of finite type up to topological conjugacy, and the characterization~\cite{BMT}
of the existence of a closing factor map between equal entropy mixing shifts of finite.
Some of these results, in turn, have shed light on problems in the non-stabilized setting, such as the use of shift equivalence to address the problem of classification of shifts of finite type up to conjugacy.

In this direction, we use our results on the stabilized automorphism group to address a question about the (non-stabilized) automorphism group.
In~\cite{WagonerEFOG}, Wagoner, asked whether the group of inert automorphisms is
always generated by simple automorphisms. Kim and Roush~\cite{KR91b} answered Wagoner's question by constructing a particular shift of finite type which has an inert automorphism that is not a product of simple automorphisms. Our methods (together with the realization results in~\cite{KRWinertaction1, KRWinertaction2}) also show that the same result holds for a wide class of shifts of finite type; for example, any shift of finite type having at least three fixed points and no points of least period two (we note this can also be deduced using some results from~\cite{Boyle88proc}, though our methods are quite different).
However, we do not know if this phenomena is even more general, and it is possible that the same result holds
for any shift of finite type (including the full shift).
A related problem is posed in Question~\ref{question:indexcomm}.

In Section~\ref{sec:stabilizedembedding} we prove a stabilized version of the Kim-Roush Embedding Theorem; namely, we show the stabilized automorphism group of any full shift embeds into the stabilized automorphism group of any mixing shift of finite type. We use this to show that, unlike the classical automorphism group, the stabilized automorphism group of a mixing shift of finite type is never residually finite. We also prove along the way that the stabilized group contains divisible subgroups, highlighting another difference with the classical setting.

\subsection{Guide to the paper}
In Section~\ref{sec:background} we give an overview of the tools we need from the classical setting of (non-stabilized) automorphism groups.
Most of these results appear scattered throughout the literature, and we present them with the goal of generalizing and adapting these results for the setting of stabilized automorphisms. Along the way, in Theorem~\ref{th:rootdistinguish} we write down the natural generalization of the observation made by Boyle, Lind, and Rudolph~\cite{BLR} that Ryan's Theorem may be used to distinguish the automorphism groups of the full 2 shift and the full 4 shift.

In Section~\ref{sec:stables}, we introduce the stabilized automorphism group. The basic properties are small variations on the classical setting, allowing us to set up and study the stabilized versions of the center, the dimension representation, and the inert subgroup.
The innovations arise when we turn to studying the commutator subgroup of the stabilized automorphism group. The key ingredient used throughout this section that is not available in the classical setting is Wagoner's Theorem, which shows that the stabilized inert automorphisms are generated by simple automorphisms. Our analysis in particular leads to Theorem~\ref{th:wagoner-answer}, which, in conjunction with the constructions in~\cite{KRWinertaction1, KRWinertaction2}, gives a method to detect, in the classical non-stabilized setting, the difference between the subgroup of inerts and the subgroup generated by simple automorphisms.
In Section~\ref{sec:abelianization}, we study the abelianization of the stabilized automorphism group.  Using our characterization of the commutator, we show how the abelianization can be used to distinguish many automorphism groups in the stabilized setting.

Section~\ref{sec:stabilizedembedding} continues the extension of various properties from the classical setting to the stabilized automorphism group. In particular, we prove a stabilized version of the Kim-Roush Embedding Theorem. The proof adapts the original construction used by Kim and Roush, with some necessary modifications.

The most difficult arguments of the paper are in Section~\ref{sec:simplicity}, where we show that the group of stabilized inert automorphisms of a
full shift is simple.
For a given shift of finite type presented by a labeled graph $\Gamma$, the group of stabilized inerts contains a certain locally finite subgroup of stabilized simple graph automorphisms associated to the presenting graph $\Gamma$. In the case of a full shift, this locally finite subgroup turns out to be simple. By a result of Boyle, this locally finite subgroup, together with the shift, generates all of the stabilized inert subgroup. The key ingredient for us then is Lemma~\ref{lemma:lemma1}, which shows that any non-trivial normal subgroup of the stabilized inert automorphisms must have non-trivial intersection with the subgroup of stabilized simple graph automorphisms. The proof of Lemma~\ref{lemma:lemma1} occupies the majority of the section.

\section{Background and notation}
\label{sec:background}
\subsection{Symbolic Dynamics}
Assume that $\A$ is a finite set endowed with the discrete topology; we call $\A$ the {\em alphabet}.   The
space $\A^{\Z}$, endowed with the product topology, is a compact, metrizable space.
An element $x\in \A^{\Z}$ is a bi-infinite sequence over the alphabet $\A$,
and we write  $x=(x_i)_{i\in\Z}$ with each $x_i\in\A$.
It is easy to check that the left shift $\sigma\colon\A^{\Z}\to\A^{\Z}$ defined by $(\sigma x)_i:=x_{i+1}$ is a homeomorphism of $\A^\Z$ to itself, and the dynamical system $(\A^{\Z},\sigma)$ is called the {\em full $\A$-shift}.  While the choice of symbols in the alphabet is irrelevant, we often
want to distinguish different full shifts by the size of the alphabet $\A$, and
so to emphasize the size of the alphabet, we write the full shift as $(X_n,\sigma_n)$ when $|\CA| = n$.

A {\em subshift} $X\subset\A^{\Z}$ is a closed, $\sigma$-invariant set $X$, and we use the shorthand {\em shift} to refer to a subshift.  We write $(X, \sigma_X)$ for this system, though when the context is clear we simplify this and just write $(X, \sigma)$.

If $w=w_1\ldots w_n\in\CA^n$, then we call $w$ a {\em word of length $n$}.  If $w$ is a word of length $n$, then
the set $[w]$ defined by
$$[w] = \{x\in \CA^\Z\colon x_i = w_i \text{ for } i=1, \ldots n\}
$$
is the {\em cylinder set determined by $w$}.  If $(X, \sigma)$ is a subshift, then the {\em language $\CL(X)$ of $X$}
is defined by
$$
\CL(X) = \{w\in\bigcup_{n=1}^\infty \CA^n\colon [w]\cap X \neq\emptyset\}.
$$
The cylinder sets associated to words in $\CL(X)$ generate the topology of the space $X$.

If $x\in X$ and $k,m\in\Z$ with $m> k$, then  $x_{[k,m]}$ denotes the word $x_kx_{k+1}\ldots x_m$ of consecutive entries in $x$.  Analogously,  $x_{(-\infty, m]}$ denotes the infinite word $\ldots  x_{m-1}x_m$, and we similarly define $x_{[k,\infty)}$.

A shift $(X, \sigma)$ is {\em irreducible} if for all words $u,v\in\CL(X)$ there exists some $w\in\CL(X)$ such that $uwv\in\CL(X)$, and the shift is {\em mixing} if for all $u,v\in\CL(X)$, there exists $N\in\N$  such that for all $n\geq N$ there is a word $w\in\CL(X)$ of length $n$ such that $uwv\in\CL(X)$.
Irreducibility of the shift $(X, \sigma)$ is equivalent to the system $(X,\sigma)$ being {\em transitive}: there exists some $x\in X$ such that the orbit closure $\overline{\{\sigma^n x\}_{n\in \mathbb{N}}}$ is all of $X$.

Two systems $(X,\sigma_{X})$ and $(Y,\sigma_Y)$ are {\em (topologically) conjugate} if there exists a homeomorphism
$h\colon X\to Y$ such that $h\circ \sigma_{X} = \sigma_{Y} \circ h$ and we refer to the map $h$ as a {\em conjugacy}.

A {\em shift of finite type} is a subshift whose language consists of all words (over some finite alphabet) which do not contain some given finite list of words.  Alternatively, a shift of finite type can be defined by an $n\times n$ adjacency matrix $A = (a_{i,j})$ over $\Z_+$ as follows. Given $A$, we define $\Gamma_{A}$ to be a graph with $n$ vertices and $a_{i,j}$ edges between vertices $i$ and $j$. Labeling the set of edges, the associated shift of finite type, which we denote by $(X_{A},\sigma_{A})$, consists of bi-infinite walks through edges in $\Gamma_{A}$. Any shift of finite type $(X,\sigma)$ is conjugate to a shift of finite type $(X_{A},\sigma_{A})$ for some $\mathbb{Z}$-matrix $A$.

A shift of finite type $(X,\sigma)$ is mixing if and only if it is conjugate to a shift of finite type $(X_{A},\sigma_{A})$ for which the $\mathbb{Z}_{+}$-matrix $A$ is \emph{primitive}, meaning there exists $K$ such that every entry of $A^{K}$ is positive. A shift of finite type $(X,\sigma)$ is irreducible if and only if it is conjugate to some $(X_{A},\sigma_{A})$ for which $A$ is an irreducible matrix, i.e. for any entry $A_{i,j}$ in $A$ there exists $k$ such that $A_{i,j}^{k}$ is positive.

\textbf{Standing Assumption: } Unless otherwise noted, we always assume that any shift of finite type  $(X,\sigma)$ has positive entropy $h_{\topo}(X)$: in terms of the language, this means that
$$ h_{\topo}(X) =\lim_{n\to \infty} \frac{\log |\{w\in \CL(X)\colon  |w|=n\}|}{n} >0 . $$
In terms of a matrix presentation, if $A$ is an irreducible matrix and $(X,\sigma)$ is conjugate to $(X_{A},\sigma_{A})$, then $h(\sigma) = h(\sigma_{A}) = \log \lambda_{A}$ where $\lambda_{A}$ is the Perron-Frobenius eigenvalue of the matrix $A$.

It follows from the  Curtis-Hedlund-Lyndon Theorem~\cite{H} that any such conjugacy is given by a {\em sliding block code}, meaning there exists some radius $r\in\N$ such that
for all $x\in X$, the value $h(x)_i$ only depends on the entries $x_{i-r}\ldots x_i\ldots x_{i+r}$.  For example, the shift $\sigma$ is given by a sliding block code with $r=1$.

\subsection{Automorphism groups}
Given a compact space $X$, let $\Homeo(X)$ denote the group of all homeomorphisms from $X$ to itself (with group operation given by composition). It is obvious that for a shift system $(X, \sigma)$ one has $\sigma\in\Homeo(X)$, and the centralizer of $\sigma$ in $\Homeo(X)$ is
called the {\em automorphism group of the subshift $(X, \sigma)$}.
As we consider various shift spaces, we denote the group (under composition) of all automorphisms of a subshift
$(X, \sigma)$ by $\aut(X,\sigma)$, and when the shift is clear from the context, we write this as $\aut(\sigma)$.
So the automorphism group of the full shift on $n$ letters is denoted by $\aut(\sigma_n)$.

 A topological conjugacy $h \colon (X,\sigma_{X}) \to (Y,\sigma_{Y})$ between shift spaces $(X, \sigma_X)$ and $(Y, \sigma_Y)$  induces an isomorphism
 $h_{*} \colon \aut(X,\sigma_{X}) \to \aut(Y,\sigma_{Y})$ defined by
\begin{equation*}
h_{*}(\phi) = h \circ \phi \circ h^{-1}.
\end{equation*}

For any subshift $(X, \sigma)$, the subgroup $\langle\sigma\rangle$ generated by the shift always lies, by definition, in the center $Z(\aut(X))$ of the automorphism group $\aut(X)$; when $X$ is infinite, the subgroup generated by $\sigma$ is isomorphic to $\mathbb{Z}$. For an irreducible shift of finite type, this subgroup is the whole center:
\begin{theorem}[Ryan~\cite{R1, R2}]
\label{th:ryan}
If $(X, \sigma)$ is an infinite irreducible shift of finite type, then $Z(\aut(X)) = \langle\sigma\rangle$.
\end{theorem}
As observed in~\cite{BLR}, this has an immediate application to distinguishing automorphism groups of full shifts, using arithmetic properties of the size of the alphabet.
\begin{corollary}\label{cor:primepowerdistinguish}
For any prime $p$, $\aut(\sigma_p)$ is not isomorphic to $\aut(\sigma_{p^p})$.
\end{corollary}

\begin{proof}
Fix a prime $p$. It is easy to check that $\sigma_{p^p} \in \aut(\sigma_{p^{p}})$ has a $p$th root, meaning there exists $\phi \in \aut(\sigma_{p^{p}})$ such that $\phi^{p} = \sigma_{p^{p}}$ (for example, one can construct such an $\phi$ using the fact that $(X_{p^{p}}, \sigma_{p^{p}})$ and $(X_{p}, \sigma_{p}^{p})$ are topologically conjugate).

If $\aut(\sigma_p)$ and $\aut(\sigma_{p^{p}})$ are isomorphic, then any isomorphism maps the center isomorphically onto the center.  By Ryan's Theorem, this means that $\sigma_p \in \aut(\sigma_p)$ is mapped to $\sigma_{p^{p}}^{\pm 1} \in \aut(\sigma_{p^{p}})$. Since $\sigma_{p^p}$ has a $p$th root, this implies either $\sigma_{p}$ or $\sigma_{p}^{-1}$ has a $p$th root.
However, we claim that neither $\sigma_p$ nor $\sigma_p^{-1}$ does.
Indeed, suppose there exists $\psi\in\aut(\sigma_{p})$ such that  $\psi^{p} = \sigma_p$ or $\psi^{p} = \sigma_{p}^{-1}$; we'll suppose $\psi^{p} = \sigma_{p}$, as the other case is similar. The system $(X_{p},\sigma_{p})$ has $p^{p}-p$ points of least period $p$, and hence $p^{p-1}-1$ orbits of length $p$. Since $p$ does not divide $p^{p-1}-1$, there exists some $1 \le i < p$, $0 \le j < p$, such that $\psi^{i}(x) = \sigma_{p}^{j}(x)$ for some period $p$ point $x$. But this implies
$$\sigma_{p}^{i}(x) = \psi^{pi}(x) = \sigma_{p}^{pj}(x) = x$$
which, since $i < p$, is a contradiction.
\end{proof}

We prove a more general result along these lines in Theorem~\ref{th:rootdistinguish}.

\subsection{The dimension representation}
Krieger~\cite{krieger80a, krieger} defined a {\em dimension triple $(\CG_A, \CG_A^+, \delta_A)$} associated to a shift of finite type $(X_{A},\sigma_{A})$, where $\CG_{A}$ is an abelian group, $\CG_{A}^{+}$ is a positive cone in $\CG_{A}$ (i.e. a subsemigroup of $\CG_{A}$ containing 0 which generates $\CG_{A}$), and $\delta_{A}$ is a group automorphism of the pair $(\CG_{A},\CG_{A}^{+})$.
A conjugacy between shifts of finite type induces a corresponding isomorphism of their respective dimension triples; since each element of $\aut(\sigma_A)$ is a conjugacy from $(X_A,\sigma_A)$ to itself, this gives rise to the \emph{dimension representation}
$$\pi_{A} \colon \aut(\sigma_{A}) \to \aut(\mathcal{G}_{A}).$$
To define this representation precisely in the manner suitable for our purposes, we briefly outline two definitions of the dimension triple $(\CG_A, \CG_A^+, \delta_A)$; the first is an intrinsic definition given by Krieger, and the second is more algebraic. These two definitions produce isomorphic objects  and this is described in~\cite[Section 7.5]{LM}; our presentation closely follows
the one given there.

Assume that $A$ is an irreducible $k \times k$ matrix with entries in $\mathbb{Z}_{+}$ and let $(X_A, \sigma_{A})$ denote the associated shift of finite type.  We further assume that  $(X_{A},\sigma_{A})$ has positive topological entropy $h_{\topo}(\sigma_{A}) >0 $, and note that $h_{\topo}(\sigma_{A}) = \log \lambda_{A}$ where $\lambda_{A}$ denotes the Perron-Frobenius eigenvalue of $A$.
The {\em eventual range $\CR(A)$ of $A$} is the subspace of $\Q^k$ defined by
$$
\CR(A) = \bigcap_{j=1}^\infty \Q^{k}A^{j}
$$
(throughout we assume the matrices act on row vectors).
The \emph{dimension triple $(\mathcal{G}_{A},\mathcal{G}_{A}^{+},\delta_{A})$ associated to $A$} consists of the abelian group $\mathcal{G}_{A}$, the semigroup $\mathcal{G}_{A}^{+} \subset \mathcal{G}_{A}$, and the automorphism $\delta_{A}$ of $\mathcal{G}_{A}$, where
\begin{enumerate}
\item
$\mathcal{G}_{A} = \{ x \in \CR(A) \colon x A^{j} \in \mathbb{Z}^{k} \textnormal{ for some }j \ge 0\}$.
\item
$\mathcal{G}_{A}^{+} = \{x \in \CR(A) \colon x A^{j} \in (\mathbb{Z}_{+})^{k} \textnormal{ for some } j \ge 0\}$.
\item
$\delta_{A}(x) = x A$.
\end{enumerate}

When $A=(n)$, we usually simply write $(\mathcal{G}_{n},\mathcal{G}^{+}_{n},\delta_{n})$ instead of $(\mathcal{G}_{(n)},\mathcal{G}^{+}_{(n)},\delta_{(n)})$.

We now describe the intrinsic definition of the dimension triple.
An \emph{$m$-ray} is defined to be a subset of $X_{A}$ of the form
$$R(x,m) = \{y \in X_{A} \colon y_{(-\infty,m]}=x_{(-\infty,m]}\}$$
for some $x \in X_{A}$ and $m \in \mathbb{Z}$,
and an \emph{$m$-beam} is a finite union of $m$-rays.
A \emph{ray} is defined to be an $m$-ray for some $m \in \mathbb{Z}$,
and a {\em beam} is an $m$-beam for some $m\in\Z$.
Note that if $U$ is an $m$-beam for some $m\in\Z$, then $U$ is also an $n$-beam for any $n \ge m$. Given an $m$-beam
$$U = \bigcup_{i=1}^{j}R(x^{(i)},m),$$
 let $v_{U,m} \in \mathbb{Z}^{k}$ denote the vector whose $J$-th component is the cardinality of the set
$$\{x^{(i)} \in U \colon \textnormal{ the edge corresponding to }x_{m}^{(i)} \textnormal{ ends at state } J\}.$$
Beams $U$ and $V$ are said to be \emph{equivalent} if there exists some $m\in\Z$ such that $v_{U,m} = v_{V,m}$, and we use $[U]$ to  denote the equivalence class of a beam $U$. Since $A$ is irreducible and $0 < h_{\topo}(\sigma_{A}) = \log \lambda_{A}$, given beams $U,V$, there exists beams $U^{\prime}, V^{\prime}$ such that
$$[U]=[U^{\prime}],\ [V] = [V^{\prime}], \text{ and }  \ U^{\prime} \cap V^{\prime} = \emptyset.$$
Let $D_{A}^{+}$ denote the abelian semigroup whose elements are equivalence classes of beams endowed with the operation defined by
$$[U] + [V] = [U^{\prime} \cup V^{\prime}].$$
Letting $D_{A}$ denote the group completion of $D_{A}^{+}$ (thus elements of $D_{A}$ are formal differences $[U]-[V]$), the map $d_{A} \colon D_{A} \to D_{A}$ induced by
$$d_{A}([U]) = [\sigma_{A}(U)]$$
is a group automorphism of $D_{A}$.  This defines Krieger's dimension triple $(D_{A},D_{A}^{+},d_{A})$.

An automorphism $\phi \in \aut(X_A, \sigma_{A})$ induces an automorphism
$$\phi_{*} \colon (D_{A},D_{A}^{+},d_{A}) \to (D_{A},D_{A}^{+},d_{A})$$
by setting
$$\phi_{*}([U]) = [\phi(U)].$$
Here by a {\em morphism of a triple}, we mean a morphism preserving all the relevant data given by the group, the subsemigroup, and the group automorphism associated to $D_A$ or $\CG_{A}$. For example, an automorphism $\Phi \in \aut(\mathcal{G}_{A},\mathcal{G}_{A}^{+},\delta_{A})$ is a group automorphism $\Phi \colon \mathcal{G}_{A} \to \mathcal{G}_{A}$ taking $\mathcal{G}_{A}^{+}$ onto $\mathcal{G}_{A}^{+}$ such that $\Phi\circ \delta_{A} = \delta_{A}\circ  \Phi$.

The relation between these two definitions
is settled by the following.
\begin{proposition}[see {\cite[Theorem 7.5.13]{LM}}]\label{prop:fernus}  Assume $(X_A, \sigma_A)$ is a shift of finite type.
The map $\theta \colon D_{A}^{+} \to \mathcal{G}_{A}^{+}$ induced by the map
$$\theta([U]) = \delta_{A}^{-k-n}(v_{U,n}A^{k}), $$
where $U$ is an $n$-beam,
is a semigroup isomorphism, and its completion is  a group isomorphism $\theta \colon D_{A} \to \mathcal{G}_{A}$ such that
$$\theta \circ d_{A} = \delta_{A} \circ \theta.$$
\end{proposition}
In other words, this proposition means that $\theta$ induces an isomorphism of triples
$$\theta \colon (D_{A},D_{A}^{+},d_{A}) \to (\mathcal{G}_{A},\mathcal{G}_{A}^{+},\delta_{A}).$$

For $\phi \in \aut(\sigma_{A})$, let $S_{\phi} \colon (\mathcal{G}_{A},\mathcal{G}_{A}^{+},\delta_{A}) \to (\mathcal{G}_{A},\mathcal{G}_{A}^{+},\delta_{A})$ denote the automorphism of the dimension triple such that  the diagram
\begin{equation*}
\begin{aligned}\label{diagram}
\xymatrix{
D_{A} \ar^{\theta}[r] \ar_{\phi_{*}}[d] & \mathcal{G}_{A} \ar^{S_{\phi}}[d]\\
D_{A} \ar^{\theta}[r] & \mathcal{G}_{A} \\
}
\end{aligned}
\end{equation*}
commutes. We can now define the dimension representation
\begin{equation*}
\begin{gathered}
\pi_{A} \colon \aut(\sigma_{A}) \to \aut(\mathcal{G}_{A},\mathcal{G}_{A}^{+},\delta_{A}) \end{gathered}
\end{equation*}
by setting
$\pi_{A}(\phi)  = S_{\phi}$.

\subsection{An application of the dimension representation}\label{sec:application-dim}

As usual, $\omega(n)$ denotes the number of distinct prime divisors of $n$ (counted without multiplicity).

The following result appears implicitly in~\cite{BLR}:
\begin{proposition}\label{prop:dimrepnonstable}
For a full shift $(X_{n},\sigma_{n})$ we have
$$\aut(\mathcal{G}_{n},\mathcal{G}_{n}^{+},\delta_{n}) \cong \left(\mathbb{Z}^{\omega(n)},\mathbb{Z}_{+}^{\omega(n)},\bm{1}\right)$$
where $\bm{1}$ denotes the vector of all $1$'s. Moreover, the dimension representation $\pi_{n} \colon \aut(\sigma_{n}) \to \aut(\mathcal{G}_{n},\mathcal{G}_{n}^{+},\delta_{n})$ is surjective.
\end{proposition}
In the proof and in the sequel, if $H \subset \mathbb{R}$ is a subgroup and $n \ge 1$, we use the notation $\mathfrak{m}_{n}$ to refer to the map from $H$ to itself given by $a \mapsto n \cdot a$.

\begin{proof}
For a full shift $(X_{n},\sigma_{n})$, there is an isomorphism of triples
$$(\mathcal{G}_{n},\mathcal{G}_{n}^{+},\delta_{n}) \cong (\mathbb{Z}[\tfrac{1}{n}],\mathbb{Z}_{+}[\tfrac{1}{n}],\mathfrak{m}_{n}).$$  Then it is straightforward to check that
$$\aut(\mathbb{Z}[\tfrac{1}{n}],\mathbb{Z}_{+}[\tfrac{1}{n}],\mathfrak{m}_{n}) \cong \left(\mathbb{Z}^{\omega(n)},\mathbb{Z}_{+}^{\omega(n)},\bm{1}\right)$$
is generated by the maps $\{\mathfrak{m}_{p} \colon p \textnormal{ is a prime dividing } n\}$.

For the second part, we write the prime factorization of $n$ as $n = \prod_{i=1}^{\omega(n)}p_{i}^{a_{i}}$ with $p_{i}$ prime. There exists a conjugacy $h \colon (X_{n},\sigma_{n}) \to \left(\prod_{i=1}^{\omega(n)}X_{p_{i}},\prod_{i=1}^{\omega(n)}\sigma_{p_{i}}^{a_{i}}\right)$ and we let $h_{*} \colon \aut(\sigma_{n}) \to \aut(\prod_{i=1}^{\omega(n)}\sigma_{p_{i}}^{a_{i}})$ denote the induced isomorphism of automorphism groups. For each $i$, let $\phi_{i}$ denote the automorphism of $\left(\prod_{i=1}^{\omega(n)}X_{p_{i}},\prod_{i=1}^{\omega(n)}\sigma_{p_{i}}^{a_{i}}\right)$ which acts by $\sigma_{p_{i}}$ in the $i$th coordinate and the identity in the other coordinates. Then the image of the automorphisms $h_{*}^{-1}(\phi_{i})$ under $\pi_{n}^{}$ generate $\aut(\mathcal{G}_{n})$.
\end{proof}

For $a \in \mathbb{N}$, let $\rts(a) = \{k \in \mathbb{N} \colon a^{1/k} \in \mathbb{N}\}$ denote the non-negative integral roots of $a$. To the authors' knowledge, the only known method for distinguishing automorphism groups of full shifts relies on Ryan's Theorem~\cite{R1}, which characterizes the center of the group of $\aut(\sigma_{A})$. This technique was explicitly mentioned in~\cite{BLR} for the full shifts on $2$ and $4$ symbols. The following result, a natural generalization of this, is not altogether new; we include it since it could not be found explicitly in the literature. Our argument uses the dimension representation; an alternative proof may be given using~\cite[Theorem 8]{Lind}.
\begin{theorem}
\label{th:rootdistinguish}
Let $n,m \ge 2$ and suppose $\aut(\sigma_{m}) \cong \aut(\sigma_{n})$. Then $\rts(m) = \rts(n)$. In particular, for any prime $p$ and $k \ge 2$, $\aut(\sigma_{p})$ and $\aut(\sigma_{p^{k}})$ are not isomorphic.
\end{theorem}
\begin{proof}
Let $k \in \rts(m)$, so there exists $a \in \mathbb{N}$ such that $a^{k} = m$. Then $(X_{m},\sigma_{m})$ is topologically conjugate to $(X_{a},\sigma_{a}^{k})$, and in particular, there exists $\phi \in \aut(\sigma_{m})$ such that $\phi^{k} = \sigma_{m}$. Suppose $\Psi \colon \aut(\sigma_{m}) \to \aut(\sigma_{n})$ is an isomorphism and let $\phi^{\prime} = \Psi(\phi)$. By Ryan's Theorem (Theorem~\ref{th:ryan}), $\Psi(\sigma_{m}) = \sigma_{n}^{\pm 1}$, so $(\phi^{\prime})^{k} = \sigma_{n}^{\pm 1}$. Applying the dimension representation gives
$$k(\pi_{n}(\phi^{\prime})) = \pi_{n}((\phi^{\prime})^{k}) = \pi_{n}(\sigma_{n}^{\pm 1}) = \pm \begin{pmatrix}v_{1} \\ v_{2} \\ \vdots \\ v_{r} \end{pmatrix} \in \mathbb{Z}^{\omega(n)}.$$
Since $\pi_{n}(\phi^{\prime}) \in \mathbb{Z}^{\omega(n)}$, each $v_{i}$ must be divisible by $k$. Let $w_{i} = \frac{v_{i}}{k}$.
Writing $n = \prod_{i=1}^{\omega(n)}p_{i}^{v_{i}}$ for some primes $p_{i}$, it follows  from Proposition~\ref{prop:dimrepnonstable} that $n = \left( \prod_{i=1}^{\omega(n)}p_{i}^{w_{i}}\right)^{k}$ so $k \in \rts(n)$. Thus $\rts(m) \subset \rts(n)$, and the same argument shows $\rts(n) \subset \rts(m)$.  Thus  $\rts(m) = \rts(n)$.
\end{proof}

In particular, it follows that the group $\aut(\sigma_{9})$ is not isomorphic to the group $\aut(\sigma_{27})$, as $\rts(9) \ne \rts(27)$.

\subsection{Inert and Simple Automorphisms}

An automorphism $\phi \in \aut(\sigma_{A})$ is said to be \emph{inert} if it lies in the kernel of the dimension representation, and we denote the subgroup of inert automorphisms by $\inert(\sigma_{A})$.  A particularly important collection of inert automorphisms is the class of simple automorphisms, first introduced by Nasu~\cite{nasu}.
We recall the definition.

If $\Gamma$ is a directed graph, we call a graph automorphism of $\Gamma$ which fixes every vertex a \emph{simple graph symmetry}  of the graph $\Gamma$.
We use the term graph symmetry instead of graph automorphism to avoid confusion between automorphisms of a graph and automorphisms of a shift.

Let $(X_{A},\sigma_{A})$ be a shift of finite type presented by a matrix $A$ over $\mathbb{Z}_{+}$ with associated directed labeled graph $\Gamma_{A}$, and suppose $\tau$ is a simple graph symmetry of $\Gamma_{A}$. Then $\tau$ induces an automorphism $\tilde{\tau} \in \aut(\sigma_{A})$ given by a 1-block code, and any automorphism in $\aut(\sigma_{A})$ which is induced by such a graph symmetry is called a \emph{simple graph automorphism}. An automorphism $\phi \in \aut(\sigma_{A})$ is called \emph{simple} if there exists a shift of finite type $(X_{B},\sigma_{B}$), a conjugacy $h \colon (X_{A},\sigma_{A}) \to (X_{B},\sigma_{B})$, and a simple graph automorphism $\tilde{\tau} \in \aut(X_{B},\sigma_{B})$ such that
$$\phi = h_{*}^{-1}(\tilde{\tau}) = h^{-1} \circ \tilde{\tau} \circ h.$$
Note that, by construction, any simple automorphism is of finite order.
It is straightforward to check that the subgroup of $\aut(\sigma_{A})$ generated by simple automorphisms forms a normal subgroup contained in $\inert(\sigma_{A})$, and we denote this subgroup by $\simp(\sigma_{A})$.

There exist irreducible shifts of finite type $(X_{A},\sigma_{A})$ for which $\simp(\sigma_{A})$ is a proper subgroup of $\inert(\sigma_{A})$; see~\cite{KR91b}. In general, the difference between $\simp(\sigma_{A})$ and $\inert(\sigma_{A})$ for an irreducible shift of finite type is not well understood; for example, it is not known whether for a full shift $(X_{n},\sigma_{n})$ the groups $\simp(\sigma_{n})$ and $\inert(\sigma_{n})$ agree.

However, Wagoner in~\cite{WagonerEFOG} showed that, upon passing to sufficiently large powers of the shift, inert automorphisms can be written as products of simple automorphisms (an alternate proof was given by Boyle in~\cite{Boyle1988}).
\begin{theorem}[Wagoner~\cite{WagonerEFOG}]
\label{th:wagoner}
If $\phi$ is an inert automorphism of a mixing shift of finite type $(X_A, \sigma_A)$, then there exists $N$ such that for all $n \geq N$, $\phi$ can be written as a product of simple automorphisms lying in $\aut(X_A, \sigma_A^n)$.
\end{theorem}

\section{The stabilized automorphism group}
\label{sec:stables}
\subsection{First properties}
For a subshift $(X,\sigma_{X})$,  let $\aut^{(k)}(\sigma_{X})$ denote  the centralizer of $\sigma_{X}^{k}$ in the group $\Homeo(X)$. Thus $\aut^{(k)}(\sigma_{X})$ is precisely $\aut(X, \sigma_{X}^{k})$ and $\aut^{(k)}(\sigma_{X})$ is a subgroup of $\aut^{(km)}(\sigma_{X})$ for all $k,m \geq 1$.

\begin{definition}
If $(X, \sigma_{X})$ is a subshift, define the {\em stabilized automorphism group $\aut^{(\infty)}(\sigma_{X})$} to be
$$
\aut^{(\infty)}(\sigma_{X}) = \bigcup_{k=1}^\infty\aut^{(k)}(\sigma_{X}),
$$
where the union is taken in $\homeo(X)$.

\end{definition}

For the full shift $(X_n, \sigma_n)$ on $n$ symbols, we denote the stabilized automorphism group by $\autinf{n}$.

It is straightforward to verify the following:
\begin{lemma}[Stabilized Curtis-Lyndon-Hedlund Theorem]\label{lem:stabilized-CLH}
Let $(X,\sigma_X)$ be a shift with alphabet $\mathcal{A}$ and let $\phi \in \aut^{(k)}(\sigma_{X})$.
Then there exist natural numbers $k$ and $r$, and $k$ block maps $\beta_i\colon \mathcal{A}^{2r+1}\to \mathcal{A}$ for $i=0,1, \dots, k-1$ such that

$$\phi(x)_z = \beta_{z \textnormal{ mod } k}(x_{z-r} , \dots, x_z, \dots , x_{z+r}).$$
\end{lemma}

Note that, the case that all $\beta_i$ are identical  yields an element that commutes with $\sigma_X$.

One concludes, either from the definition or using Lemma~\ref{lem:stabilized-CLH} that $\aut^{(\infty)}(\sigma_{X})$ is a countable group that contains the automorphism group $\aut(\sigma_{X})$.

For some subshifts, nothing new arises in the stabilized automorphism group:
\begin{example}
Let $(X,\sigma_{X})$ be a minimal shift associated to an irrational rotation: for example, such a shift can be defined by fixing an irrational $\alpha\in(0,1)$, considering $$T(x) =x+\alpha\pmod 1,$$
and using the coding of the orbit of $0$ defined by setting the $n^{th}$ entry to be $0$ if $T^n(x)\in[0, \alpha)$ and $1$ if $T^n(x)\in[\alpha,1)$.  This gives rise to a Sturmian shift (see for example~\cite[Chapter 6]{PF} for background on Sturmian shifts), and $\aut(\sigma_{X}) \cong \mathbb{Z}$ is generated by the shift $\sigma_{X}$ (see~\cite{Olli2013}).

The system $(X,\sigma_{X})$ has a single pair of asymptotic orbits $\mathcal{O}_{1}, \mathcal{O}_{2}$, and for each $k \ge 1$ the system $(X,\sigma_{X}^{k})$ then has $k$ pairs of asymptotic orbits given by the collection $\{\sigma_{X}^{i}(\mathcal{O}_{1}),\sigma_{X}^{i}(\mathcal{O}_{2})\}_{i=0}^{k-1}$. Using~\cite[Lemma 2.3]{DDMP}, it follows that any automorphism in $\aut(\sigma_{X}^{k})$ is of the form $\sigma_{X}^{j}$ for some $j \in \mathbb{Z}$, and hence
$$\aut(\sigma_{X}^{k}) \cong \mathbb{Z} = \langle \sigma_{X} \rangle.$$
Thus, in this case, for any $k, m \ge 1$ we have
\begin{equation}
\xymatrix{
\aut^{(k)}(\sigma_{X}) \ar[r] \ar[d]_{\cong} & \aut^{(km)}(\sigma_{X}) \ar[d]^{\cong}\\
\mathbb{Z} \ar[r]_{\textnormal{id}} & \mathbb{Z} \\
}
\end{equation}
and $\aut^{(\infty)}(\sigma_{X}) = \aut(\sigma_{X}) \cong \mathbb{Z}$.
Moreover, these groups are not just abstractly both isomorphic to $\Z$ but are the same as subgroups of $\Homeo(X)$, as they are all equal to $\aut(\sigma_X)$.
\end{example}

However, for a shift of finite type, each inclusion in the definition of the stabilized automorphism group is strict:
\begin{lemma}\label{lemma:onthewayup}
If $(X_A, \sigma_A)$ is an infinite irreducible shift of finite type, then for any $k\in\N$ and any $m \geq 2$, the subgroup $\aut^{(k)}(\sigma_A)$ is a proper subgroup of $\aut^{(km)}(\sigma_A)$.
\end{lemma}
\begin{proof}
By Ryan's Theorem (Theorem~\ref{th:ryan}), the center of
$\aut^{(km)}(\sigma_A) = \aut(\sigma_A^{km})$ is exactly $\langle\sigma_A^{km}\rangle$.  Thus there exists some $\phi\in\aut^{(km)}(\sigma_A)$ such that $\phi$ does not commute
with $\sigma_A^k$.
\end{proof}

In Proposition~\ref{prop:trivial-center},
we make further use of Ryan's Theorem and prove a stronger result, showing that for an irreducible shift of finite type $(X, \sigma_A)$, we have that $\aut(\sigma_A)$ is not abstractly isomorphic to $\aut^{\infty}(\sigma_A)$.

The following proposition follows immediately from  the definition of the stabilized automorphism group:
\begin{proposition}
\label{prop:smallprop}
For any shift $(X, \sigma)$ and $k \ge 1$, $\aut^{(\infty)}(\sigma^{k}) = \aut^{(\infty)}(\sigma)$.
\end{proposition}

It is well known that if two shifts are conjugate, then their
automorphism groups are isomorphic, and the same holds true
for their stabilized automorphism groups.
In fact, a stronger result holds  in the stabilized setting, and to make this precise, we define a weaker notion that suffices for the associated groups to be isomorphic.

Recall that $(X,\sigma_{X})$ and $(Y,\sigma_{Y})$ are {\em eventually conjugate} if there exists some $K \in\N$ such that for all $k \ge K$, $(X,\sigma_{X}^{k})$ and $(Y,\sigma_{Y}^k)$ are conjugate.
We define a weaker notion: we say that the systems $(X,\sigma_{X})$ and $(Y,\sigma_{Y})$ are \emph{rationally conjugate} if there exist $j,k \ge 1$ such that the systems $(X,\sigma_{X}^{j})$ and $(Y,\sigma_{Y}^{k})$ are conjugate.
For example, the systems $(X_2,\sigma_2)$ and $(X_4,\sigma_4)$ are rationally conjugate but are not eventually conjugate.
\begin{proposition}
\label{prop:eventualconj}
If the systems $(X,\sigma_{X})$ and $(Y,\sigma_{Y})$
are rationally conjugate, then $\aut^{(\infty)}(\sigma_{X})$ and $\aut^{(\infty)}(\sigma_{Y})$ are isomorphic.
\end{proposition}

\begin{proof}
If $h \colon (X,\sigma_{X}^{j}) \to (Y,\sigma_{Y}^{k})$ is a conjugacy then $h_{*}$ gives rise to an isomorphism
$$h_{*} \colon \aut^{(\infty)}(\sigma_{X}^{j}) \to \aut^{(\infty)}(\sigma_{Y}^{k}).$$
By Proposition~\ref{prop:smallprop}, this implies $\autinf{X}$ and $\autinf{Y}$ are isomorphic.
\end{proof}

In particular, since $(X_{4},\sigma_{4})$ is conjugate to $(X_{2},\sigma_{2}^{2})$, it follows that $\aut^{(\infty)}(\sigma_2)$ and
$\aut^{(\infty)}(\sigma_4)$ are isomorphic, in constrast to the non-stabilized setting, where
$\aut(\sigma_2)$ and $\aut(\sigma_4)$ are not isomorphic (see Theorem~\ref{th:rootdistinguish}).

Recall that   two matrices  $A$ and $B$ with entries in $\Z_+$ are said to be {\em shift equivalent} (over $\mathbb{Z}_{+}$) if there exists an integer $m\geq 1$
and matrices $R$ and $S$ over $\Z_+$ such that
$$
AR = RB, \ SA = BS, \ A^m = RS, \text{ and } B^m= SR.
$$
If $A$ and $B$ are irreducible $\mathbb{Z}_{+}$-matrices which are shift equivalent then the systems $(X_{A},\sigma_{A}), (X_{B},\sigma_{B})$ are eventually conjugate, and Kim and Roush~\cite{KR79} showed the converse holds. We use this to show:
\begin{proposition}
Suppose $(X_A, \sigma_A)$ and $(X_B, \sigma_B)$ are irreducible shifts of finite type defined by $\mathbb{Z}_{+}$-matrices $A,B$. If $A$ and $B$ are shift equivalent, then $\autinf{A}$ and $\autinf{B}$ are isomorphic.
\end{proposition}

\begin{proof}
By Kim and Roush~\cite{KR79, KR99}, matrices $A$ and $B$ are shift equivalent if and only if the systems $(X_{A},\sigma_{A})$ and $(X_{B},\sigma_{B})$ are eventually conjugate. The result then follows from Proposition~\ref{prop:eventualconj}.
\end{proof}

\subsection{The center}
\label{sec:the-center}
Ryan's Theorem (Theorem~\ref{th:ryan}) shows that for any irreducible shift of finite type, the center is exactly the powers of the shift.
In contrast, the center is trivial in the stabilized automorphism group:
\begin{proposition}
\label{prop:trivial-center}
Suppose $(X_{A},\sigma_{A})$ is an infinite irreducible shift of finite type. Then the center $Z(\autinf A)$ of $\autinf A$ is trivial, and the group
$\autinf A$ is not finitely generated.
\end{proposition}
\begin{proof}
Suppose $\phi \in Z(\autinf A)$ and choose $k \ge 1$ such that $\phi \in \autk{k}{A}$. Then $\phi \in Z(\autk{k}{A})$,
so by Ryan's Theorem we have $\phi = \sigma_{A}^{k m}$ for some $m \in \mathbb{Z}$.
However if $\sigma_{A}^{k m} = \phi \in Z(\autinf A)$, then $\sigma_{A}^{k m} \in Z(\autk{2km}A) = \langle \sigma_{A}^{2k m} \rangle$, so $m=0$.

For any irreducible shift of finite type $(X_{A},\sigma_{A})$,
any finitely generated subgroup of $\autinf{A}$ has  nontrivial centralizer (as each finitely generated subgroup is included in $\autk k A$ for some $k$, for which $\sigma_A^k$ would be in the centralizer). By the previous part, it follows that for any infinite irreducible shift of finite type, the group $\autinf{A}$ is not finitely generated.
\end{proof}

\subsection{The $\autinf{A}$-action on $X_{A}$}\label{sec:autinfaction}
Let $(X_{A},\sigma_{A})$ be a mixing shift of finite type and let $P(X_{A})$ denote the set of $\sigma_{A}$-periodic points in $X_{A}$. Then both $\aut(\sigma_{A})$ and $\autinf{A}$ act on the set $P(X_{A})$. While the action of $\aut(\sigma_{A})$ on $P(X_{A})$ is far from transitive (since any $\phi \in \aut(\sigma_{A})$ must preserve the order of a $\sigma_{A}$-periodic point), it follows from~\cite[Theorem 3.6]{BK1987} that $\autinf{A}$ acts highly transitively on the $\sigma_A$-periodic points of $X_{A}$ (recall an action of a group $G$ on a countable set $X$ is said to be highly transitive if for all $k \ge 1$ it is transitive on the set of ordered $k$-tuples of distinct elements in $X$).

It is straightforward to check that the action of $\aut(\sigma_{A})$ on $X_{A}$ is not minimal, since there are periodic points. Similarly, there are many $\aut(\sigma_{A})$-invariant probability measures, including atomic measures supported on periodic points, and the measure of maximal entropy. However, the minimal components and $\aut(\sigma_{A})$-invariant measures are essentially classified in~\cite[Sections  9 and 10]{BLR}. Using this, we deduce:
\begin{proposition}
If $(X_{A},\sigma_{A})$ is a mixing shift of finite type, then $\autinf{A}$ acts highly transitively on the set of $\sigma_{A}$-periodic points in $X_{A}$, and the action of $\autinf{A}$ on $X_{A}$ is minimal and uniquely ergodic. Moreover, the unique $\autinf{A}$-invariant probability measure is given by the measure of maximal entropy for the system $(X_{A},\sigma_{A})$.
\end{proposition}
\begin{proof}
It follows from~\cite[Theorem 3.6]{BK1987} that $\autinf{A}$ acts highly transitively on the $\sigma_A$-periodic points of $X_{A}$. Given this, the minimality, unique ergodicity, and claim regarding the measure of maximal entropy then follow from~\cite[Theorem 9.2 and Corollary 10.2]{BLR}.
\end{proof}

\subsection{The stabilized dimension representation}\label{sec:stabdimrep}
Let $A$ be a $\mathbb{Z}_{+}$-matrix, and recall we have defined the dimension representation
\begin{equation*}
\pi_{A} \colon \aut(\sigma_{A}) \to \aut(\mathcal{G}_{A},\mathcal{G}_{A}^{+},\delta_{A}).
\end{equation*}
For any $k\geq 1$, we also have a homomorphism
$$\pi_{A}^{(k)} \colon \aut(\sigma_A^k) \to \aut(\mathcal{G}_{A^k},\mathcal{G}_{A^{k}}^{+},\delta_{A^{k}}).$$
Note that in general we have $(\mathcal{G}_{A},\mathcal{G}_{A}^{+}) = (\mathcal{G}_{A^{k}},\mathcal{G}_{A^{k}}^{+})$ for all $k \in \mathbb{N}$. However the dimension triples $(\mathcal{G}_{A},\mathcal{G}_{A}^{+},\delta_{A})$ and $(\mathcal{G}_{A^{k}},\mathcal{G}_{A^{k}}^{+},\delta_{A}^{k})$ are \emph{not} isomorphic, as there is no isomorphism that intertwines the maps $\delta_{A}$ and $\delta_{A}^{k}$. For each $k\geq 1$ the map $\pi_A^{(k)}:\aut^{(k)}(\sigma_A)\to \aut(\mathcal{G}_A,\mathcal{G}_{A^{k}}^{+},\delta_{A^{k}})$ sends $\sigma_{A}^k$ to $\delta_{A}^{k}$, and the image of $\pi_{A}^{(k)}$ lands in the centralizer of $\delta_{A}^{k}$, so in fact we have a homomorphism
$$\pi_{A}^{(k)} \colon \aut^{(k)}(\sigma_{A}) \to \aut(\mathcal{G}_{A},\mathcal{G}_{A}^{+},\delta_{A}^{k}).$$

It follows from the definitions that for all $k\geq 1$, $\aut(\mathcal{G}_{A},\mathcal{G}_{A}^{+},\delta_{A})$
can be viewed naturally as a subgroup of  $\aut(\mathcal{G}_{A},\mathcal{G}_{A}^{+},\delta_{A}^{k})$, and we can define the \emph{stabilized group of automorphisms of the dimension triple} by setting
$$\aut^{(\infty)}(\mathcal{G}_{A},\mathcal{G}_{A}^{+},\delta_A) = \bigcup_{k=1}^{\infty}\aut(\mathcal{G}_{A},\mathcal{G}_{A}^{+},\delta_{A}^{k}).$$
Equivalently, $\aut^{(\infty)}(\mathcal{G}_{A},\mathcal{G}_{A}^{+},\delta_A)$ is the union of the centralizers of $\delta_{A}^{k}$ in the group of automorphisms of the pair $(\mathcal{G}_{A},\mathcal{G}_{A}^{+})$, that is, all automorphisms of the group $\mathcal{G}_{A}$, which preserve $\mathcal{G}_{A}^{+}$.

Furthermore, as remarked in~\cite[p. 87]{BLR}, for $k \geq 1$, the restriction of the map
$$\pi_{A}^{(k)} \colon \aut^{(k)}(\sigma_{A}) \to \aut(\mathcal{G}_{A^{k}},\mathcal{G}_{A^{k}}^{+},\delta_{A}^{k})=\aut(\mathcal{G}_{A},\mathcal{G}_{A}^{+},\delta_{A}^{k})$$
to $\aut(\sigma_{A}) \subset \aut(\sigma_{A}^{k})$ coincides with the map $\pi_{A} \colon \aut(\sigma_{A}) \to \aut(\mathcal{G}_{A},\mathcal{G}_{A}^{+},\delta_{A})$. We can thus define the \emph{stabilized dimension representation}
$$
\pi_{A}^{(\infty)} \colon \autinf A \to \aut^{(\infty)}(\mathcal{G}_{A},\mathcal{G}_{A}^{+},\delta_A).
$$
In what follows, we use the shorthand notation $\aut^{(\infty)}(\mathcal{G}_{A})$ to refer to the group $\aut^{(\infty)}(\mathcal{G}_{A},\mathcal{G}_{A}^{+},\delta_A)$.
\begin{example}
Consider the case of the full 3-shift, presented via the matrix  $A = (3)$. For all $k\in\N$, we have $\mathcal{G}_{3} = \mathcal{G}_{3^{k}} = \mathbb{Z}[\frac{1}{3}]$. In this case $\aut(\mathcal{G}_{3^{k}}) = \aut(\mathcal{G}_{3}) \cong \left(\mathbb{Z},\mathbb{Z}_{+},1\right)$ for any $k$ (so $\delta_{3}$ corresponds to $1$ in $\mathbb{Z}$), and
$$\pi_{3}^{(k)} \colon \autk{k}{3} \to \aut(\mathcal{G}_{3}) \cong \left(\mathbb{Z},\mathbb{Z}_{+},1\right)$$
with $\pi_{3}^{(k)}(\sigma_{3}) = \delta_{3}.$
\end{example}
Recall $\omega(n)$ denotes the number of distinct prime factors of $n$, and
the maps $\mathfrak{m}_p$ are defined by $\mathfrak{m}_{p}(x) = p \cdot x$.
\begin{proposition}\label{prop:stableautodimgroup}
For the full shift $(X_n, \sigma_n)$, we have
$$\aut^{(\infty)}(\mathcal{G}_{n}) \cong \aut(\mathcal{G}_{n}) \cong \left(\mathbb{Z}^{\omega(n)},\mathbb{Z}_{+}^{\omega(n)},\bm{1}\right)$$
is generated by the maps  $\{\mathfrak{m}_{p} \colon p \textnormal{ is a prime dividing } n\}$.
\end{proposition}
\begin{proof}
The statement follows immediately from Proposition~\ref{prop:dimrepnonstable}, and the fact that the maps $\mathfrak{m}_{p}$ generate $\aut(\mathbb{Z}[\frac{1}{n}],\mathbb{Z}_{+}[\frac{1}{n}],\delta_{n}) \cong \left(\mathbb{Z}^{\omega(n)},\mathbb{Z}_{+}^{\omega(n)},\bm{1}\right)$.
\end{proof}

In the case of a full shift $(X_{n},\sigma_n)$, the classical dimension representation
$$\pi_{n} \colon \aut(\sigma_{n}) \to \aut(\mathcal{G}_{n})$$
is surjective (see Proposition~\ref{prop:dimrepnonstable}). However, in the general setting of mixing shifts of finite type, the dimension representation need not be surjective: Kim, Roush, and Wagoner~\cite{KRW} give an example of a mixing shift of finite type for which the dimension representation is not surjective, and in the general setting of mixing shifts of finite type, the question of when the dimension representation is surjective remains open. In the stabilized setting, however, the question has a satisfying answer, as shown in~\cite{BLR} (our terminology is different, but this is an immediate translation of their result):
\begin{theorem}[{Boyle, Lind and Rudolph~\cite[Theorem 6.8]{BLR}}]
\label{thm:stabledimrepsurj}
For any mixing shift of finite type $(X_{A},\sigma_{A})$, the stabilized dimension representation
$$\pi_{A}^{(\infty)} \colon \autinf{A} \to \aut^{(\infty)}(\mathcal{G}_{A})$$
is surjective.
\end{theorem}

As in the standard setting, we define the {\em group of stabilized inert automorphisms} to be the kernel of $\pi_{A}^{(\infty)}$, and we denote this group by
$$\inert^{(\infty)}(\sigma_{A}) = \ker \pi_{A}^{(\infty)}.$$
It follows immediately from the definitions that
$$\inert^{(\infty)}(\sigma_{A}) = \bigcup_{k=1}^{\infty}\inert(\sigma_{A^{k}}).$$

We show later that one of the many differences between stabilized and standard automorphism groups lies in the structure of their corresponding inert subgroups. In particular, in Section~\ref{sec:simplicity} we prove that, in the case of a full shift, $\inertinf{n}$ is always simple. This is in stark contrast to the classical inert subgroup $\inert(\sigma_{n})$, which is residually finite. Using the stabilized version of the Kim-Roush Embedding proved in Section~\ref{sec:stabilizedembedding}, it follows that for any mixing shift of finite type $(X_{A},\sigma_{A})$, $\inertinf{A}$ always contains an infinite simple group; in particular, $\inertinf{A}$ is never residually finite (see Section~\ref{sec:residual}). We note that, as a consequence, $\inert^{(\infty)}(\sigma_{A})$ and $ \inert(\sigma_{A})$ are not isomorphic as groups (in fact, it follows that $\inertinf{A}$ does not even embed into $\inert(\sigma_{A})$).

Rewriting Wagoner's Theorem (Theorem~\ref{th:wagoner}) in our terminology, we have:
\begin{theorem}[Wagoner (Theorem~\ref{th:wagoner} rephrased)]\label{thm:efog}  If $(X_A, \sigma_A)$ is a mixing shift of finite type, then
$\inert^{(\infty)}(\sigma_{A})$ is generated by simple automorphisms.
\end{theorem}

\subsection{The commutator subgroup}
\label{sec:commutator}
The goal of this section is to prove:
\begin{theorem}\label{thm:stablecommutator}
Let $(X_{A},\sigma_{A})$ be a mixing shift of finite type. Then
$$\inert^{(\infty)}(\sigma_{A}) \subseteq [\autinf{A}, \autinf{A}].$$
If $\aut^{(\infty)}(\mathcal{G}_{A})$ is abelian, then equality holds. In particular, for a full shift we have
$$\inert^{(\infty)}(\sigma_{n}) = [\autinf{n}, \autinf{n}].$$
\end{theorem}
Note that, in the case where $\aut^{(\infty)}(\mathcal{G}_{A})$ is torsion-free (e.g. a full shift), Wagoner's Theorem as phrased in Theorem~\ref{thm:efog} characterizes the dynamical object given by the group of stabilized inert automorphisms via an abstract property of the group: the subgroup generated by the elements of finite order.  Theorem~\ref{thm:stablecommutator} gives a general relation between an abstract group property, this time the commutator, and the dimension representation of the symbolic system.

The following lemma is the technical tool needed for the proof of Theorem~\ref{thm:stablecommutator}:
\begin{lemma}\label{lemma:swapcomm}
Let $(X_{A},\sigma_{A})$ be a shift of finite type and let $\tau$ be a simple graph symmetry of the graph $\Gamma_{A}$ which permutes two distinct edges $e$ and $f$ between the vertices $i$ and $j$.
Let $\tilde{\tau}$ denote the automorphism of $(X_{A},\sigma_{A})$ induced by $\tau$. Then $\tilde{\tau} \in [\aut(\sigma_{A}^{2}),\aut(\sigma_{A}^{2})]$.
\end{lemma}

\begin{proof}
We consider $(X_{A},\sigma_{A}^{2})$ as a shift on the alphabet $\begin{pmatrix} a_{0} \\ a_{1} \end{pmatrix}$ where $a_{0}a_{1}$ is an admissible word in $X_{A}$. Define the zero-block code $\phi_{0}$ in $\aut(\sigma_{A}^{2})$ by
\begin{equation*}
\phi_{0} \colon \begin{pmatrix} a_{0} \\ a_{1} \end{pmatrix} \mapsto  \begin{pmatrix} \tau(a_{0}) \\ a_{1} \end{pmatrix}
\end{equation*}

Note that since $\tilde{\tau}$ is a simple graph automorphism, it follows that $\phi_{0}$ is an
automorphism of $(X_{A},\sigma_{A}^{2})$. Then in $\aut(\sigma_{A}^{2})$, we have
\begin{equation}\label{eqn:simplecommeqn}
\tilde{\tau} = \phi_{0} \sigma_{A} \phi_{0}^{-1} \sigma_{A}^{-1}.\quad\qedhere
\end{equation}
\end{proof}

For a set $X$, let $\sym(X)$ denote the group of all permutations of the set $X$.
\begin{theorem}\label{theorem:simptocomm}
Let $(X_{A},\sigma_{A})$ be a shift of finite type and let $\phi \in \aut(\sigma_{A})$ be a simple automorphism. Then $\phi \in [\aut(\sigma_{A}^{2}),\aut(\sigma_{A}^{2})]$.
\end{theorem}
\begin{proof}
Since $\phi$ is simple, there exists some shift of finite type $(X_{B},\sigma_{B})$ and a conjugacy $h \colon (X_{A},\sigma_{A}) \to (X_{B},\sigma_{B})$ such that $h_{*}(\phi)$ is a simple graph automorphism. Set $\tilde{\tau} = h_{*}(\phi)$. Since $h$ also induces an isomorphism between $\aut(\sigma_{A}^{2})$ and $\aut(\sigma_{B}^{2})$, it suffices to show that $\tilde{\tau} \in [\aut(\sigma_{B}^{2}),\aut(\sigma_{B}^{2})]$.

Let $E_{i,j}$ denote the set of edges between vertices $i,j$ in the graph $\Gamma_{B}$. There exist permutations $\tau_{i,j} \in \sym(E_{i,j})$ such that $\tilde{\tau}$ is induced by the simple graph symmetry $\prod_{i,j}\tau_{i,j}$. For each pair $i,j$, the permutation $\tau_{i,j}$ is given by a product of transpositions in $\sym(E_{i,j})$. By Lemma~\ref{lemma:swapcomm}, the automorphism induced by each of these transpositions lies in $[\aut(\sigma_{B}^{2}),\aut(\sigma_{B}^{2})]$, so  $\tilde{\tau}$ lies in $[\aut(\sigma_{B}^{2}),\aut(\sigma_{B}^{2})]$ as well.
\end{proof}

We now use Theorem~\ref{theorem:simptocomm} to complete the proof of Theorem~\ref{thm:stablecommutator}:
\begin{proof}[Proof of Theorem~\ref{thm:stablecommutator}]
Theorem~\ref{theorem:simptocomm} implies that any simple automorphism lies in the commutator. By Theorem~\ref{thm:efog}, the group $\inert^{(\infty)}(\sigma_{A})$ is generated by simple automorphisms, proving the first part.

To check the second statement, when
$\aut^{(\infty)}(\mathcal{G}_{n})$ is abelian, the dimension representation
$$\pi^{(\infty)}_{n} \colon \autinf n \to \aut^{(\infty)}(\mathcal{G}_{n})$$
factors through the abelianization of $\autinf{A}$.
Thus \begin{equation*}[\autinf{A},\autinf{A}] \subseteq \inert^{(\infty)}(\sigma_{A}). \quad\qedhere
\end{equation*}
\end{proof}

As a second corollary of Theorem~\ref{theorem:simptocomm}, we can in some cases show that, in the non-stabilized automorphism group $\aut(\sigma_{A})$, a particular inert
automorphism can not lie in the subgroup generated by simple automorphisms. Such results can also be deduced from~\cite[Theorem 2]{Boyle88proc}, where the possible actions of simple automorphisms on finite subsystems of the shift were classified. Together with the powerful realization result in~\cite{KRWinertaction1, KRWinertaction2}, this provides a large class of examples where the answer to Wagoner's question~\ref{th:wagoner-answer} is no:
\begin{theorem}
\label{th:wagoner-answer}
Let $(X_{A},\sigma_{A})$ be a shift of finite type and suppose there exists odd $k \in \mathbb{N}$ such that $X_{A}$ has no $\sigma_{A}$-periodic points of least period $2k$ and further assume that there are at least three distinct orbits of least period $k$.  Then the group generated by
 simple automorphisms is a proper subgroup of $\inert(\sigma_{A})$.
\end{theorem}
\begin{proof}
By~\cite[Main Theorem]{KRWinertaction1, KRWinertaction2},
there exists $\phi \in \inert(\sigma_{A})$ such that the action of $\phi$ on the $\sigma_{A}$-orbits of length $k$ consists of a 3-cycle. We show that $\phi$ can not be written as a product of commutators in $\aut(\sigma_{A}^{2})$ of the form given in~\eqref{eqn:simplecommeqn}. By Theorem~\ref{theorem:simptocomm}, it follows that $\phi \not \in \simp(\sigma_{A})$.

Suppose $\gamma \in \aut(\sigma_{A}^{2})$. Since $k$ is odd, $\sigma_{A}^{2}$ maps length $k$ $\sigma_{A}$-orbits to themselves. Furthermore, since there are no $\sigma_{A}$-periodic points of least period $2k$, it follows that $\aut(\sigma_{A}^{2})$ induces a well-defined action on the set of $\sigma_{A}$-orbits of length $k$. Since $\sigma_{A}$ acts trivially on the set of $\sigma_{A}$-orbits of length $k$, the commutator $\gamma \sigma_{A} \gamma^{-1}\sigma_{A}^{-1}$ acts trivially on the set of length $k$ $\sigma_{A}$-orbits. Thus, since $\phi$ acts non-trivially on the $\sigma_{A}$-orbits of length $k$, $\phi$ can not be written as a product of such commutators.
\end{proof}

For a concrete example of the phenomena exhibited in this corollary, consider the primitive matrix
\begin{equation}
\label{eq:matrixA}
A = \begin{pmatrix}
1 & 1 & 1 & 0 \\
0 & 1 & 0 & 1 \\
0 & 1 & 1 & 0 \\
1 & 0 & 1 & 0 \\
\end{pmatrix}.
\end{equation}
Since the system $(X_{A},\sigma_{A})$ has 3 fixed points and no period 2 points, by Theorem~\ref{th:wagoner-answer}, $\inert(\sigma_{A}) \ne \simp(\sigma_{A})$.

\begin{remark}
Considering the matrix $A$ in~\eqref{eq:matrixA},
 it can be shown using~\cite[Theorem 1]{BF1991} that there exists a product of finite order inert automorphisms in $\aut(\sigma_{A})$ whose action on the set of fixed points in $(X_{A},\sigma_{A})$ is a 3-cycle.
 Letting  $\textnormal{Fin}(\sigma_{A})$ denote the subgroup of $\inert(\sigma_{A})$ generated by elements of finite order,  in light of Theorem~\ref{th:wagoner-answer}, for this matrix $A$ we have the following proper containments:
$$\simp(\sigma_{A}) \subsetneq \textnormal{Fin}(\sigma_{A}) \subsetneq \inert(\sigma_{A}).$$
\end{remark}

In general, we do not know if  $\simp(\sigma_{A})$ is always finite index in $\inert(\sigma_{A})$. Based on Theorems~\ref{theorem:simptocomm} and~\ref{th:wagoner-answer}, as a way to approach this question, we ask the following:
\begin{question}
\label{question:indexcomm}
Assume $(X_A, \sigma_A)$ is a shift of finite type.
Is
\begin{equation*}
\inert(\sigma_{A}) \cap [\aut(\sigma_{A}^{2}),\aut(\sigma_{A}^{2})]
\end{equation*}
finite index in $\inert(\sigma_{A})$?
\end{question}

\subsection{The abelianization of $\autinf{A}$ and Theorem~\ref{thm:main-distinguish}}
\label{sec:abelianization}
For a group $G$, we let $G_{\ab}$ denote its abelianization. With the previous results in hand, we can now show that the abelianization of $\autinf{A}$ for a general mixing shift of finite type $(X_{A},\sigma_{A})$ coincides with the abelianization of its dimension representation.

\begin{theorem}\label{thm:hopefullygeneral}
Suppose $(X_{A},\sigma_{A})$ is a mixing shift of finite type. Then
$$\autinf{A}_{\ab} \cong \aut^{(\infty)}(\mathcal{G}_{A})_{\ab}.$$
\end{theorem}
\begin{proof}
Consider the following diagram:
\begin{equation}
\xymatrix{
\autinf{A} \ar[d]_{{\rm Ab}_{\sigma_{A}}} \ar[r]^{\pi_{A}^{(\infty)}} & \aut^{(\infty)}(\mathcal{G}_{A}) \ar[d]^{{\rm Ab}_{\mathcal{G}_{A}}} \ar@{-->}[dl]_{f}\\
\left(\autinf{A}\right)_{\ab} & \left(\aut^{(\infty)}(\mathcal{G}_{A})\right)_{\ab} \ar@{-->}[l]^{g}\\
}
\end{equation}
By Theorem~\ref{thm:stablecommutator},
 $\inert^{(\infty)}(\sigma_{A}) \subset [\autinf A, \autinf A]$ and by Theorem~\ref{thm:stabledimrepsurj} the map $\pi_{A}^{(\infty)}$ is surjective, so the map $f$ is well-defined.
Since $f$  factors through the abelianization of $\aut^{(\infty)}(\mathcal{G}_{A})$, the map $g$ exists.
Moreover, since ${\rm Ab}_{\sigma_{A}}$ is surjective, $f$ is surjective, and hence $g$ is surjective.

We claim that the map $g$ is also injective. Suppose $a \in \ker g$. Since the
map ${\rm Ab}_{\mathcal{G}_{A}}$ is surjective, we can  find $b \in \aut^{(\infty)}(\mathcal{G}_{A})$ such that ${\rm Ab}_{\mathcal{G}_{A}}(b) = a$, and hence $f(b) = \Id$. By Theorem~\ref{thm:stabledimrepsurj}, $\pi_{A}^{(\infty)}$ is surjective, so there exists $c \in \autinf{A}$ such that $\pi_{A}^{(\infty)}(c) = b$. Then $c$ lies in the kernel of the map ${\rm Ab}_{\sigma_{A}}$,
which implies that $c$ is a commutator. Thus $\pi_{A}^{(\infty)}(c) = b$ is also  a commutator, and hence $a = {\rm Ab}_{\mathcal{G}_{A}}(b) = \Id$.
\end{proof}

\begin{corollary}\label{cor:rankprimediv}
If $n\geq 2$, then
$\autinf{n}_{\ab} \cong \mathbb{Z}^{\omega(n)}$.
\end{corollary}

\begin{proof}
This follows immediately from Theorem~\ref{thm:stablecommutator} and Proposition~\ref{prop:stableautodimgroup}.
\end{proof}

This allows us to complete the proof of Theorem~\ref{thm:main-distinguish}, via the following:
\begin{theorem}\label{thm:invofprims}
If $\autinf n$ and $\autinf m$ are isomorphic, then $\omega(n) = \omega(m)$.
\end{theorem}
\begin{proof}
If $\autinf n$ and $\autinf m$ are isomorphic, then their abelianizations are isomorphic. The result then follows from Corollary~\ref{cor:rankprimediv}.
\end{proof}
Towards a converse of Theorem~\ref{thm:invofprims}, observe that by Proposition~\ref{prop:smallprop}, if $m, n$ satisfy $m^{k} = n^{j}$ for some $k$ and $j$, then $\autinf m \cong \autinf n$.

 In general, we ask:
 \begin{question}
 \label{conj:stable}
 For integers $m,n\geq 2$, when are
 $\autinf m$ and $\autinf n$
  isomorphic?
 \end{question}

We end this section with an example showing how Theorem~\ref{thm:hopefullygeneral} can be used to compute the abelianization $\autinf{A}_{\ab}$ of the stabilized automorphism group. In the example, $\autinf{A}_{\ab}$ has nontrivial torsion, and it follows (by Corollary~\ref{cor:rankprimediv}) that $\autinf{A}$ is not isomorphic to $\autinf{n}$ for any $n\in\N$.

\begin{example}
Consider the matrix
$$A = \begin{pmatrix} 5 & 2 & 2 \\ 4 & 1 & 4 \\ 0 & 6 & 3 \end{pmatrix}$$
(this matrix appears in~\cite[Example 6.7]{BLR}). By Theorem~\ref{thm:hopefullygeneral}, in order to compute $\autinf{A}_{\ab}$,  it suffices to compute the abelianization of the stabilized automorphism group of the dimension group.

As shown in~\cite{BLR}, the matrix $A$ has eigenvalues $-3,3,9$ and can be conjugated over $\mathbb{Z}[\frac{1}{3}]$ to a diagonal matrix. For any $k$, $A^{2k}$ then has eigenvalues $9^{k},9^{k}$, and $81^{k}$.
It follows that $\aut(\mathcal{G}_{A^{2k}}) \cong \mathbb{Z} \oplus \GL_{2}(\mathbb{Z}[\frac{1}{3}])$, so
$\aut^{(\infty)}(\mathcal{G}_{A}) \cong \mathbb{Z} \oplus \GL_{2}(\mathbb{Z}[\frac{1}{3}])$ and $\aut^{(\infty)}(\mathcal{G}_{A})_{\ab}$ is isomorphic to $\mathbb{Z} \oplus \GL_{2}(\mathbb{Z}[\frac{1}{3}])_{\ab}$. By Theorem~\ref{thm:hopefullygeneral}, the dimension representation is surjective and coincides with the abelianization of $\autinf{A}$.

The remainder of this example is devoted to computing $\GL_{2}(\mathbb{Z}[\frac{1}{3}])_{\ab}$.  Consider the determinant map
$$\textnormal{det} \colon \GL_{2}(\mathbb{Z}[\frac{1}{3}]) \to \mathbb{Z}[\frac{1}{3}]^{\times}$$
where $\mathbb{Z}[\frac{1}{3}]^{\times}$ denotes the group of units. This map is a split surjection with kernel $\SL_{2}(\mathbb{Z}[\frac{1}{3}])$, with the splitting coming from embedding $\mathbb{Z}[\frac{1}{3}]^{\times} = \GL_{1}(\mathbb{Z}[\frac{1}{3}]) \hookrightarrow \GL_{2}(\mathbb{Z}[\frac{1}{3}])$.  Hence $\GL_{2}(\mathbb{Z}[\frac{1}{3}])$ is isomorphic to the semidirect product $\SL_{2}(\mathbb{Z}[\frac{1}{3}]) \rtimes\mathbb{Z}[\frac{1}{3}]^{\times}$.

In general, the abelianization of a semidirect product $H \rtimes G$ is given by $(H_{\ab})_{G} \times G_{\ab}$, where the subscript $G$ denotes the coinvariants of the $G$-action on $H_{\ab}$ (arising from the $G$-action on $H$). Since $\mathbb{Z}[\frac{1}{3}]^{\times}$ is abelian, the abelianization of the semidirect product $\SL_{2}(\mathbb{Z}[\frac{1}{3}]) \rtimes\mathbb{Z}[\frac{1}{3}]^{\times}$ has the form $$(\SL_{2}(\mathbb{Z}[\frac{1}{3}])_{\ab})_{\mathbb{Z}[\frac{1}{3}]^{\times}} \times \mathbb{Z}[\frac{1}{3}]^{\times}.$$
This leaves us with computing  $(\SL_{2}(\mathbb{Z}[\frac{1}{3}])_{\ab})_{\mathbb{Z}[\frac{1}{3}]^{\times}}$.

The abelianization of $\SL_{2}(\mathbb{Z}[\frac{1}{3}])$ is $\SL_{2}(\mathbb{Z}[\frac{1}{3}])_{\ab} \cong \mathbb{Z}/4$, as computed by Serre~\cite{serre} (see also~\cite{AdemNadim1998}). Thus we only need
to determine the coinvariants of the induced $\mathbb{Z}[\frac{1}{3}]^{\times}$-action on this copy of  $\mathbb{Z}/4$.

The ring map $\mathbb{Z}[\frac{1}{3}] \to \mathbb{Z}/4$ given by $\frac{a}{3^{k}} \mapsto a \mod  4$ induces a surjection mapping $\SL_{2}(\mathbb{Z}[\frac{1}{3}])$ to $\SL_{2}(\mathbb{Z}/4)$. The group
$\SL_{2}(\mathbb{Z}/4)$ has a normal subgroup $N$ of order 12 (this is its commutator subgroup) which is generated by the matrices $\begin{pmatrix} 2 & 3 \\ 3 & 1 \end{pmatrix}$ and $\begin{pmatrix} 3 & 1 \\ 3 & 0 \end{pmatrix}$. Thus $\SL_{2}(\mathbb{Z}/4)$ factors on to an abelian group $G$ of order $4$.
Let $\pi$ denote the composition of the two maps given by
$$\SL_{2}(\mathbb{Z}[\frac{1}{3}]) \to \SL_{2}(\mathbb{Z}/4) \to G.$$
One can check directly that the matrix $\begin{pmatrix} 1 & 1 \\ 0 & 1 \end{pmatrix}$ and its square do not lie in the normal subgroup $N$, and hence do not lie in the kernel of $\pi$. Thus $\pi(\begin{pmatrix} 1 & 1 \\ 0 & 1 \end{pmatrix})$ has order 4, and $\pi(\begin{pmatrix} 1 & 1 \\ 0 & 1 \end{pmatrix})$ is a generator for $G$, and hence also pushes down to a generator for the abelianization.

To compute the coinvariants, we are left with determining
the action of $\mathbb{Z}[\frac{1}{3}]^{\times}$ on the matrix $\begin{pmatrix} 1 & 1 \\ 0 & 1 \end{pmatrix}$ (since it pushes down to a generator of the abelianization). Note that $\mathbb{Z}[\frac{1}{3}]^{\times}$ is generated by $-1$ and $3$. The action of these units on $\begin{pmatrix} 1 & 1 \\ 0 & 1 \end{pmatrix}$ is given by (modulo commutators)
\begin{align*}
   -1 \colon & \begin{pmatrix} 1 & 1 \\ 0 & 1 \end{pmatrix} \mapsto \begin{pmatrix} 1 & 3 \\ 0 & 1 \end{pmatrix} \\
3 \colon & \begin{pmatrix} 1 & 1 \\ 0 & 1 \end{pmatrix} \mapsto \begin{pmatrix} 1 & 3 \\ 0 & 1 \end{pmatrix}.
 \end{align*}
It follows that the orbit of a generator for the abelianization under this action is a subgroup of order 2, and the coinvariants are
$$(\SL_{2}(\mathbb{Z}[\frac{1}{3}])_{\ab})_{\mathbb{Z}[\frac{1}{3}]^{\times}} \cong \mathbb{Z}/2.$$
Thus, we have that
$$\GL_{2}(\mathbb{Z}[\frac{1}{3}])_{\ab} \cong \mathbb{Z}[\frac{1}{3}]^{\times} \oplus \mathbb{Z}/2 \cong \mathbb{Z}/2 \oplus \mathbb{Z} \oplus \mathbb{Z}/2$$
and
$$\aut^{(\infty)}(\mathcal{G}_{A})_{\ab} \cong \mathbb{Z} \oplus \mathbb{Z}/2 \oplus \mathbb{Z} \oplus \mathbb{Z}/2.$$
\end{example}

\section{Stabilized Kim-Roush embedding}\label{sec:stabilizedembedding}
\subsection{Extending the embedding result}
The purpose of this section is to extend the following theorem of Kim and Roush to the stabilized setting:
\begin{theorem}[Kim-Roush Embedding~\cite{KR90}]
Let $(X_A,\sigma_A)$ be a mixing shift of finite type. Then for any $n\geq 2$,  the group $\aut(\sigma_n)$
embeds into the group $\aut(\sigma_A)$.
\label{prop:kim-roush}
\end{theorem}
Thus our goal is to prove the following.
\begin{theorem}
Let $(X_A,\sigma_A)$ be a mixing shift of finite type. Then for any $n\geq 2$,
the group $\autinf{n}$ embeds into
$\autinf{A}$.
\label{prop:stabilized-kim-roush}
\end{theorem}

The proof follows much of the original argument given in~\cite{KR90}, with a few modifications. Before beginning, we briefly indicate the idea.  We proceed by constructing a bijection $h$ from the given shift $X_A$ to some other space $K$. While $h$ is nothing more than  a bijection, the advantage in making use of the second space $K$ is that it admits a natural faithful  $\autinf{n}$-action. This action of $\autinf{n}$ leaves the image of $h$ invariant, and so  upon pulling back by $h$, we obtain an embedding of $\autinf{n}$ into the set of bijections from  $X_{A}$ to itself. We then show that this embedding actually lands in $\autinf{A}$.
The construction of the map $h$ uses markers, as used in~\cite{H,BLR} and we review this technique in the proof.

\begin{proof}[Proof of Theorem~\ref{prop:stabilized-kim-roush}]
Let $(X_A,\sigma_A)$ be a mixing shift of finite type.  The proof consists of multiple steps constructing the embedding.
\subsubsection*{Finding markers}
Assume that there exists
a word $M\in \CL (X_A)$ (a \textit{marker}) and a collection
${\mathcal{D}} \subset \CL (X_A)$ of $n^2$ words of some fixed length such that the word $M$ overlaps
$MDM$, for any $D\in {\mathcal{D}}$, only in the initial and finial segments (the \textit{data}). The existence of such pairs of marker and a data set of size $n^2$ is guaranteed for any $n\in\N$ since we assume that $X_A$ is a mixing shift of finite type.

Since there are $n^2$ words in  ${\mathcal{D}}$, we can view them as pairs of words,
from some collection of size $n$ of some other words. Namely,
we define an abstract set of $n$ words $\CW$ such that each $D \in {\mathcal{D}}$
is a pair of two words from $\CW$. Since there are $n$ words in  $\CW$,
we can view the full shift over these words as $(X_n,\sigma_n)$, and the stabilized automorphism group of this shift is the one we realize as a subgroup of $\autinf{A}$.

It is convenient to consider
the elements in $\mathcal{D}$ as vertical pairs, viewing them as
  \begin{align*}
    D &= \begin{pmatrix}
           W^u \\
           W^l
         \end{pmatrix}
  \end{align*}
where $W^u,W^l\in \CW$. For simplicity of the presentation we assume that all of the words $\mathcal{D}$ are words of length
1, which is possible after passing via a conjugacy, if needed, to a copy of $(X_A,\sigma_A)$.  Then for $x\in X_A$ and some index $j$, if $x_j=D$ we can write
\begin{align*}
    x_j &= \begin{pmatrix} x_j^u \\ x_j^l \end{pmatrix}.
\end{align*}

\subsubsection*{Coded stretches in the shift}
Fix some $R \in \mathbb{N}$.
An \textit{$(R, M, \mathcal{D})$-coded stretch} in $x\in X_A$ is an $R$-gapped (possibly finite) arithmetic progression  $C\subset \Z$ such that $x_j\in \mathcal{D}$ for all $j\in C$, and $C$ is  maximal with respect to these properties. That is, if $\max(C)$ exists then $x_{\max(C)+R} \notin \mathcal{D}$, and if $\min(C)$ exists, then $x_{\min(C)-R} \notin \mathcal{D}$.

Note that coded stretches may be finite, two sided infinite, or one sided infinite.
Since $X_A$ is mixing, there are points $x\in X_A$ with arbitrarily long coded stretches (including infinite ones). Moreover, each word in $\CL(X_n)$, whether finite or infinite, appears as a coded stretch of some $x\in X_A$.
For each $x\in X_A$,  let $\BL_x$ denote the union of all the coded stretches in $x$.

Fix some $x\in X_A$.
Recall that for $j\in \BL_x$, $x_j^u$ and $x_j^l$ are two words in $\CW$. Again, we consider elements in $\{u,l\}\times \BL_x$ as vertical pairs, so  if $p=\begin{pmatrix}\epsilon \\ j \end{pmatrix}\in \{u,l\}\times \BL_x$, we write
$x_p = x^{\epsilon}_j \in \CW$.

\subsubsection*{The function $\text{next}$}
We define an invertible map $\text{next}_{x}\colon \left\{ u,l\right\} \times\BL_{x}\to\left\{ u,l\right\} \times\BL_{x}$
by setting
$$\text{next}_{x}\left(_{j}^{u}\right)=\begin{cases}
\left(_{j+R}^{u}\right) & \text{if }j+R\in\BL_{x}\\
\left(_{j}^{l}\right) & \text{if}\,j+R\notin\BL_{x}
\end{cases}
$$
and
$$\text{next}_{x}\left(_{j}^{l}\right)=\begin{cases}
\left(_{j-R}^{l}\right) & \text{if }j-R\in\BL_{x}\\
\left(_{j}^{u}\right) & \text{if}\,j-R\notin\BL_{x}
\end{cases}.$$
Fix $\begin{pmatrix} \epsilon \\ j
\end{pmatrix}$ where $j \in \BL_x$.
Repeated application of the next function produces an element in $X_n$ when starting with an element in $X_{A}$, by reading the words  appearing in the current coded stretch when applying this function; for example,
\begin{align*}
\begin{matrix}
\dots &  * &  W_1^u & \rightarrow & W_2^u & \rightarrow & W_3^u & * & \dots \\
 &   &  \uparrow &  &  &  & \downarrow &  &  \\
\dots &  * &  W_1^l & \leftarrow & W_2^l & \leftarrow & W_3^l & * & \dots .
\end{matrix}
\end{align*}

Let $C$ be a finite or one-sided coded stretch, and let $j,j'\in C$.
Note that starting to read from $\begin{pmatrix} \epsilon \\ j
\end{pmatrix}$ or from $\begin{pmatrix} \epsilon' \\ j'
\end{pmatrix}$ yields the same element in $X_n$, up to a shift.
However, for a two sided stretch $C$, the element of $X_{n}$ read from the $u$ row has nothing
to do with the element read from the $l$ row.

\subsubsection*{The function $\text{read}$}
To maintain the group structure when embedding the group $\autinf{n}$, we are forced to keep track of which level an element belongs to (as $\phi\in\autk{k}{n}$ applies $k$ different block maps, depending on the index $\mbox{mod } k$).
For this, we define a {\em read} map which depends on the index, in such a way that the word read from  $\begin{pmatrix} \epsilon \\ j
\end{pmatrix}$ and from $\begin{pmatrix} \epsilon \\ j'
\end{pmatrix}$ would be identical (where identical means not just up
to a shift).
Formalizing this, define   $\text{read}_{x}\colon \BL_{x}\to X_{n}^{2}$ by setting  $\text{read}_{x}\left(i\right)=\left(y^{u},y^{l}\right)$ where
\begin{align*}
\begin{matrix}
y_{\left\lfloor \frac{i}{R}\right\rfloor + z}^{u}=x_{(\text{next}_{x})^{z}
\left(\begin{array}{c}
u\\ i
\end{array}\right)} &  \mbox{and} & y_{-\left\lfloor \frac{i}{R}\right\rfloor + z}^{l}=x_{(\text{next}_{x})^{z}
\left(\begin{array}{c}
l\\ i
\end{array}\right)}
\end{matrix}
\end{align*}
 for all $z\in\Z$.

We note that this complication does not arise in the original embedding of Kim and Roush~\cite{KR90} of  $\aut(\sigma_n)$ in $\aut(\sigma_{A})$, as one can
define the read map without the floor functions (similarly for the multidimensional version of Hochman~\cite{Hoch}).

Let  $Y=\CA\cup X_{n}^{2}$, where $\CA$ is the alphabet of $X_A$,
and consider the set $\bar{K}=\prod_{j\in\Z}Y_j$.

\subsubsection*{Definition of the map $h$}
Define a map   $h\colon X_A \to \bar{K}$ by  setting
\begin{align}
\label{eq:def-of-h}
h(x)_{j}=\begin{cases}
\text{read}_{x}(j) & \text{if }j\in\BL_{x}\\
x_{j} & \text{otherwise}
\end{cases}.
\end{align}
Thus $h$ assigns to every $x\in X_A$ a sequence in $\bar{K}$ in the following
way. If $x_j$ is not included in any coded stretch, $h$ copies the symbol $x_j$ to the $j$ coordinate of the new element in $\bar{K}$.
If $x_j$ is included in a coded stretch, there are two elements in $X_n$ that are read
from this stretch: the one associated with the upper row, and the one associated with the lower row, and this pair of elements is placed in the $j$ coordinate of the new element in $\bar{K}$.

Set $K = \im(h) \subset \bar{K}$.

\subsubsection*{The map $h$ is injective}
We claim that the map $h$ is injective.  To see this, we check the action of the inverse of $h$ on its image.   For any
coordinate of a given
point in $K$, there is either an element from $\CA$ or there
is a pair in $X^2_n$. In the first case, $h^{-1}$ copies the symbol. In the second case, we  (re)-form the pair
composed of one symbol from the first element and the other from the
second element from $X^2_n$. More precisely, in this case:
\begin{align*}
h^{-1}(k)_j=\begin{cases}
\begin{pmatrix}((k_j)_1)_{-\left\lfloor \frac{i}{R}\right\rfloor} \\ ((k_j)_2)_{\left\lfloor \frac{i}{R}\right\rfloor} \end{pmatrix}  & \text{if }k_j\in X^2_n\\
k_j & \text{if }k_j \in \CA
\end{cases}.
\end{align*}
This verifies the claim.

We now make use of the representation of the element $x$ as $h(x)$
by exploiting the natural associated  $\autinf{n}$
action. On $Y$, we have  a pointwise action of $\autinf{n}$
(and trivial action on the $\CA$ part), and this action naturally
extends to a diagonal action on $\bar{K}$. In other words, there is a group homomorphism $\autinf{n} \to \bij(\bar{K})$.

\subsubsection*{Stabilized automorphisms keep the set $K$ invariant}
Next we claim that every element in $\autinf{n}$ is a bijection that keeps the set $K$ invariant, and the restriction action of $\autinf{n}$ on $K$ is faithful.
To check this,  note
that each element of $\autinf{n}$ keeps $K$
invariant by the mixing
assumption. In fact, the same  holds for
any map $X_n \to X_n$.
As $K$ is invariant, we can consider the restriction of the $\autinf{n}$-action to $K$.
Since all words of $\CL(X_n)$
appear as coded stretches for  some $x\in X_A$, every word in $X_n$ appears in some coordinate of some element in $K$,
and as the action of $\autinf{n}$ on $X_n$ is faithful (by definition), we conclude that the action on $K$ is faithful as well.
Thus the claim follows.

In other words, this realizes $\autinf{n}$ as a subgroup of $\bij(K)$.  Furthermore, the bijection $h\colon X_A \to K$ induces a group isomorphism $h_*\colon \bij(K) \to \bij(X_A)$.

\subsubsection*{Stabilized automorphisms give rise to continuous maps commuting with some power of the shift}
By pushing $\autinf{n}$ through the injective map $h_*$, we realize $\autinf{n}$ as a subgroup of $\bij(X_A)$. To verify that the image lies in $\autinf{n}$, we are left with checking that every $\phi\in \autinf{n} \subseteq \bij(K)$ gives rise to a continuous $h_* \phi \in \Homeo(X_A)$ which commutes with some power of $\sigma_A$.

To do this, we make use of the block map description
of the stabilized automorphism group (Lemma~\ref{lem:stabilized-CLH}).
Fix some $\phi\in \aut^{(k)}(\sigma_{n})$ of radius $r$. That is, $\phi$ can be represented as
$k$ block maps of radius $r$, where $r$ is some number greater than $k$.
Now if $x$ and $x'$ are two points in $X_A$ which are close, then by definition they agree on a large number of coordinates around the $0$ coordinate. In particular, their coded stretches (if they exist) in this area coincide. So there exists large $s>0$ such that $\BL_x \cap [-s,s] = \BL_{x'} \cap [-s,s]$.
Since $\phi$ is of radius $r$, $h_*\phi(x)$ and $h_*\phi(x')$ agree on $[-s+r,s-r]$, and hence $h_{*}\phi$ is a continuous map.
Finally to check that $h_{*}\phi$ commutes with a power of the shift, using the fact that $\phi\in \aut^{(k)}(\sigma_{n})$
is induced by a $k$-tuple of block maps on $X_n$, it is easy to check that $h_* \phi$ can be modeled
by a $k\cdot R$-tuple of block maps on $X_A$.

This concludes the proof of Theorem~\ref{prop:stabilized-kim-roush}.
\end{proof}

\subsection{Residual Finiteness and subgroup properties}\label{sec:residual}
For a mixing shift of finite type $(X_{A},\sigma_{A})$, the classical automorphism group $\aut(\sigma_{A})$ is residually finite (see~\cite[Section 3]{BLR}). The stabilized Kim-Roush Embedding, together with simplicity of the stabilized inerts for the full shifts, implies that the stabilized group $\autinf{A}$ is never residually finite. We show below that, in addition, $\autinf{A}$ always contains a divisible group.
\begin{proposition}\label{prop:notresidfinite}
Let $(X_{A}, \sigma_{A})$ be a mixing shift of finite type. Then $\autinf{A}$ contains an infinite simple subgroup, and a divisible subgroup. In particular, the group $\autinf{A}$ is not residually finite.
\end{proposition}
\begin{proof}
Since any subgroup of a residually finite group is residually finite, and infinite simple groups are not residually finite, by Theorem~\ref{prop:stabilized-kim-roush} it suffices to prove the statements for full shifts. The group $\inertinf{n}$ is always infinite, and by Theorem~\ref{th:even-simple}, is simple, giving the first part. For the divisible subgroup, let $m \ge 2$. We show that $\autinf{2}$ contains the divisible group $\mathbb{Z}[\frac{1}{m}] / \mathbb{Z}$.
We claim that if $\phi_{0} \in \aut(\sigma_{2}^{k})$ is given by a 0-block code, then there exists $\phi_{1} \in \aut(\sigma_{2}^{mk})$ such that $\phi_{1}^{m} = \phi_{0}$. The result then follows by letting $\phi_{0}$ be any 0-block code of order $m$ in $\aut(\sigma_{2}^{j})$ for some $j, m$, and induction.

To prove the claim, suppose we have such $\phi_{0}$. We consider the alphabet for the shift $\sigma_{2}^{mk}$ as symbols $\begin{pmatrix} a_{0} \\ \vdots \\ a_{m-1} \end{pmatrix}$ where $a_{i} \in \{0,1\}^{k}$. Define 0-block codes in $\aut(\sigma_{2}^{mk})$ as follows:
$$ \alpha_{0}(\begin{pmatrix} a_{0} \\ \vdots \\ a_{m-1} \end{pmatrix}) =
\begin{pmatrix} \phi_{0}(a_{0}) \\ a_{1} \\ \vdots \\ a_{m-1} \end{pmatrix}, \hspace{.3in}  a_{i} \in \{0,1\}^{k}
$$
and
$$
c_{m}\begin{pmatrix} a_{0} \\ a_{1} \\ \vdots \\ a_{m-1} \end{pmatrix} = \begin{pmatrix} a_{1} \\ a_{2} \\ \vdots \\ a_{0} \end{pmatrix}.
$$
Then it is easy to  check that
$$\left(\alpha_{0}c_{m}\right)^{m} = \phi_{0},$$
as desired.
\end{proof}

This method can be used to produce other embeddings into $\autinf{n}$.
Given a prime $p\geq 2$, consider the direct limit $\SL_{\infty}^{\diag}(\mathbb{F}_{p})$ of the systems $(\SL_{2^{n}}(\mathbb{F}_{p}),i_{n})$ where $i_{n} \colon \SL_{2^{n}}(\mathbb{F}_{p}) \to \SL_{2^{n+1}}(\mathbb{F}_{p})$ is the map given by $A \mapsto A \oplus A$. A construction analogous to the one given in the proof of Proposition~\ref{prop:notresidfinite} can be used to produce an embedding of $\SL_{\infty}^{\diag}(\mathbb{F}_{p})$ into $\autinf{p}$.

We end this section with an example of how results in the stabilized setting can be used to study the classical automorphism group $\aut(\sigma_{A})$.

\begin{lemma}
For a full shift $(X_n, \sigma_n)$, the group $\aut(X_n)$ embeds into the group $\inert(X_n)$.
\end{lemma}

\begin{proof}
Let $f\colon \aut(\sigma_n) \to \aut(\sigma_n)$ be a Kim-Roush embedding.
Since we are considering a full shift,
for any $\phi\in\aut(\sigma_n)$, the action of $f_{*}(\phi)$ on the dimension group $\mathcal{G}_{n}$ is determined by its action on any $0$-ray $R$, since the equivalence class of any $0$-ray rationally generates $\mathcal{G}_{n}$.

For a symbol $a$, let $R_a$
denote the $0$-ray of points $x$ such that $x_{i} = a$ for all $i \le 0$.  By construction of the Kim-Roush embedding, there is some symbol $a$ such that
$f_{*}(\phi)(R_{a})$ is again a $0$-ray. Since all $0$-rays in $(X_{n},\sigma_{n})$ are equivalent, this implies that $f_{*}(\phi)$ acts trivially on the dimension group, i.e. $f_{*}(\phi) \in \inert(X_n)$.
\end{proof}

\begin{theorem}
Let $G$ be a finitely generated group which embeds into $\aut(\sigma_{n})$. Then $G$ embeds (using a possibly different embedding) into $[\aut(\sigma_{n}),\aut(\sigma_{n})]$.
\end{theorem}

\begin{proof}
Suppose $G$ embeds into $\aut(\sigma_{n})$. Composing this embedding with a Kim-Roush embedding $f$ gives an embedding of $G$ into $\inert(\sigma_{n})$ (by the previous lemma). In particular, $G$ embeds in $\inertinf{n}$, which, by Theorem~\ref{thm:stablecommutator}, is a subgroup of $[\autinf{n}, \autinf{n}]$.
Since $G$ is finitely generated, it follows that $G$ embeds inside $[\autk m n, \autk m n]$ for some $m\in\N$. We can then apply another  Kim-Roush embedding, this time to embed $\autk m n$ (which is isomorphic to $\aut(\sigma_{n^{m}})$) into $\aut(\sigma_{n})$.
The composition of these embeddings takes $G$ into $[\aut(\sigma_{n}),\aut(\sigma_{n})]$.
\end{proof}

While we focus mainly on  positive entropy mixing
shifts of finite type, the following proposition holds in greater generality. This is the
only obstruction, of which we are aware, for realization of a countable group in $\autinf{A}$.  The same proof as in Boyle, Lind, and Rudolph~\cite{BLR} immediately gives:
\begin{proposition}
    Let $(X,\sigma)$ be any subshift. Then any finitely generated subgroup of $\aut^{(\infty)}(\sigma)$ has a solvable word problem.
\end{proposition}

\section{Simplicity of the stabilized inerts for full shifts}\label{sec:simplicity}
\subsection{Simplicity}
For a mixing shift of finite type $(X_{A},\sigma_{A})$, the classical inert subgroup $\inert(\sigma_{A})$ has an abundance of normal subgroups. For example, given $\phi \in \inert(\sigma_{A})$ and $k \in \mathbb{N}$, $\phi$ leaves invariant the set $P_{k}(\sigma_{A})$ of $\sigma_{A}$-periodic points of period $k$, and there is a well-defined homomorphism from $\inert(\sigma_{A})$ to $\textnormal{Sym}(P_{k}(\sigma_{A}))$. Moreover, if $\id \ne \phi$, then there exists some $k$ such that $\phi$ acts nontrivially on $P_{k}(\sigma_{A})$, and it follows from this that the group $\inert(\sigma_{A})$ is in fact residually finite (see~\cite[Section 3]{BLR} for details).

In contrast, different behavior arises in the stabilized setting, where the inert subgroup has no nontrivial normal subgroups.  The remainder of this section is devoted to the proof of Theorem~\ref{th:even-simple},
which we restate for convenience:
\begin{theorem*}[Theorem~\ref{th:even-simple}]
For any $n \ge 2$, the group of stabilized inert automorphisms of the full shift $(X_{n},\sigma_{n})$ is simple.
\end{theorem*}

Simplicity of various groups defined via dynamical systems has been shown in other contexts (see for example~\cite{JM, Matui, Nek}). For many of these groups, an important and useful property is the existence of elements of the group which act by the identity on certain regions of the domain space. In contrast to such groups, the action of the group $\inertinf{n}$ on the shift space is of a very different nature; for example, for any mixing shift of finite type $(X_{A},\sigma_{A})$, and in particular any full shift, if $\id \ne \phi \in \inertinf{A}$, then for any open subset $U \subset X_{A}$, $\phi \ne \id$ on $U$ (in other words, $\inertinf{A}$ never contains non-trivial elements with small support).

\subsection{Stabilized simple automorphisms}
\label{sec:simple-setup}
Many of the ingredients in the proof of Theorem~\ref{th:even-simple} hold more generally, and
so we start with some preliminaries that hold for more than the full shift.

Assume $(X_A, \sigma_A)$ is a mixing shift of finite type defined by a $k \times k$ primitive $\mathbb{Z}_{+}$-matrix $A$ (note that the full shift on $n$ symbols corresponds to $A = (n))$.  Let $\Gamma_{A}$ denote a directed labeled graph associated to $A$ and let $\simp(\Gamma_{A})$ denote the subgroup of simple automorphisms in $\aut(\sigma_{A})$ induced by simple graph symmetries of $\Gamma_{A}$. Note that $\simp(\Gamma_A)$ is contained in $\simp(\sigma_{A})$, but the converse inclusion does not hold.

Recall that $E_{i,j}$
denotes the set of edges between vertices $i$ and $j$ in the graph $\Gamma_A$.
There is a natural isomorphism
\begin{equation}
    \label{eq:simp-is-sym}
\simp(\Gamma_{A}) \cong \prod_{i,j=1}^{k}\sym(E_{i,j}),
\end{equation}
where we adopt the convention that if
$E_{i,j} = \emptyset$ for some choice of  $i$ and $j$, we assume that $\sym(E_{i,j})$ is the trivial group with one element.

We define the subgroup of even simple graph automorphisms $\simp_{\ev}(\Gamma_{A})$ in $\simp(\Gamma_{A})$ by pulling back the associated product of alternating subgroups, meaning the subgroup $\prod_{i,j=1}^{k}\alt(E_{i,j})$,
via the isomorphism in~\eqref{eq:simp-is-sym}.

Let $\Gamma^{(m)}_{A}$ denote a graph which presents the shift $(X_{A},\sigma_{A}^{m})$; thus $\simp(\Gamma^{(m)}_{A}) \subset \aut(\sigma_{A}^{m})$. We note the graphs $\Gamma_{A}^{(m)}$ and $\Gamma_{A^{m}}$ differ only up to a choice of labeling. For any $k, m \ge 1$ we have an inclusion map
\begin{equation}
\label{eq:def-im}
    i_{m,k} \colon \simp(\Gamma^{(m)}_{A}) \hookrightarrow \simp(\Gamma^{(km)}_{A}),
    \end{equation}
and by making the natural identifications among the iterates, this homomorphism agrees with the restriction of the map
$$\aut(\sigma^{m}_{A}) \hookrightarrow \aut(\sigma_{A}^{km})$$
to $\simp(\Gamma^{(m)}_{A})$.
\begin{proposition}
For any $k,m \ge 1$, the map $i_{m,k}$ takes $\simp_{\ev}(\Gamma^{(m)}_{A})$ into $\simp_{\ev}(\Gamma^{(km)}_{A})$.
\end{proposition}
\begin{proof}
Fix vertices $I,J$ in $\Gamma_{A}^{(m)}$, and let $\tau \in \alt(E_{I,J})$. Letting $\tilde{\tau}$ denote the element of $\simp_{\ev}(\Gamma_{A}^{(m)})$ corresponding to $\tau$ under the isomorphism in~\eqref{eq:simp-is-sym}, it suffices to show that $i_{m,k}(\tilde{\tau})$ lies in $\simp_{\ev}(\Gamma_{A}^{(km)})$. We may write $\tilde{\tau}$ as a product of an even number of transpositions $\tilde{\tau} = \prod_{i=1}^{2l}\tilde{\tau}_{i}$, and for each $1 \le i \le 2l$, since $\tilde{\tau}_{i}$ is an involution, we may write $i_{m,k}(\tilde{\tau}_{i}) = \prod_{j=1}^{r_{i}}c_{j}$ where each $c_{j}$ is a 2-cycle. It suffices then to show that $r_{p} = r_{q}$ for any $1 \le p,q \le 2l$. Given some $1 \le p \le 2l$, suppose the involution $\tilde{\tau}_{p}$ corresponds (under the isomorphism~\eqref{eq:simp-is-sym}) to the transposition in $\alt(E_{I,J})$ which permutes a pair of edges $e_{p},f_{p}$ between vertices $I$ and $J$. Then the value $r_{p}$ is given by $\frac{1}{2} M_{p}$, where $M_{p}$ denotes the number of distinct words $w$ of length $k$, over the alphabet given by the edge set of $\Gamma_{A}^{(m)}$, where each word $w$ contains at least one $e_{p}$ or $f_{p}$. Since the number $M_{p}$ of such words is independent of what $e_{p},f_{p}$ are, it follows that $M_{p} = M_{q}$ for any other $1 \le q \le 2l$, as desired.
\end{proof}

We consider the corresponding stabilized groups, defining the subgroups
$$\simpinf{A} = \bigcup_{m=1}^{\infty} \simp(\Gamma^{(m)}_{A}) \subset \autinf{A}$$
and
$$\simpinfev{A} = \bigcup_{m=1}^{\infty}\simp_{\ev}(\Gamma^{(m)}_{A}) \subset \simpinf{A}.$$
Thus $\alpha \in \autinf{A}$ lies in $\simpinf{A}$ when  $\alpha$ is induced by a simple graph symmetry of $\Gamma^{(m)}_{A}$ for some $m\geq 1$, and $\alpha \in \simpinfev{A}$ if for some $m\geq 1$, $\alpha$ is induced by a simple graph symmetry of $\Gamma^{(m)}_{A}$ which consists of only even permutations on every edge set for $\Gamma^{(m)}_{A}$.
We note that it follows from the definitions that
$$\simpinf{A} \subset \inertinf{A}.$$

With this notation, Wagoner's Theorem (Theorem~\ref{thm:efog})
states that for a mixing shift of finite type $(X_{A},\sigma_{A})$, $\inertinf{A}$ is generated by the collection of subgroups
$\Psi_{*}^{-1}(\simpinf{B})$, where $\Psi \colon (X_{A}, \sigma_{A}^{m})\to (X_{B}, \sigma_{B}^{m})$ is any
conjugacy and $ m\geq 1$ is any integer.

The key lemma in the proof Theorem~\ref{th:even-simple} is:
\begin{lemma}\label{lemma:lemma1}
Let $n \ge 2$ and let $N$ be a nontrivial normal subgroup of $\inertinf{n}$. There exists $m \ge 0$ and $\id \ne \zeta \in \simpinf{n}$ such that $\sigma_{n}^{m} \zeta \sigma_{n}^{-m} \in N$.
\end{lemma}

The proof of Lemma~\ref{lemma:lemma1} is technical and long, and we postpone it until Section~\ref{sec:proof of lemma1}.
For now, we assume this result and proceed to develop the other tools needed in the proof of Theorem~\ref{th:even-simple}.

\begin{lemma}
\label{lemma:joel}
Assume $(X_A, \sigma_A)$ is a mixing shift of finite type defined by a primitive $\mathbb Z_+$-matrix $A$. Then the following hold:
\begin{enumerate}
\item\label{item:partone}
The commutator subgroup of $\simpinf{A}$ is $\simpinfev{A}$.
\item\label{item:parttwo}
The group $\simpinfev{A}$ is simple.
\item
\label{item:partthree}
If $A = (n)$ for some $n \ge 2$, then $\simpinf{n} = \simpinfev{n}$.
\end{enumerate}
\end{lemma}

\begin{proof}
For Part~\eqref{item:partone}, clearly $\simpinfev{A}$ is contained in $[\simpinf{A},\simpinf{A}]$. For the
other inclusion, consider a commutator $\alpha\beta\alpha^{-1}\beta^{-1} \in \simpinf{A}$, where $\alpha, \beta \in \simpinf{A}$. We may assume that both $\alpha, \beta \in \simp(\Gamma^{(m)}_{A})$ for some $m\geq 1$. Then for each vertex pair $i$ and $j$ in the graph $\Gamma_{A}^{(m)}$, the component of $\alpha \beta \alpha^{-1} \beta^{-1}$ in $\sym(E_{i,j})$ lies in $\alt(E_{i,j})$.  Thus $\alpha\beta\alpha^{-1}\beta^{-1} \in \simpinfev{A}$.

For Part~\eqref{item:parttwo}, let $\{e\} \ne N$ be a normal subgroup of $\simpinfev{A}$. For  $k\geq 1$ and a pair of vertices $i,j$ in the graph $\Gamma_{A}^{(k)}$, let $\alt^{(k)}_{i,j}$ denote the subgroup of $\simpinfev{A}$ obtained by pulling back the alternating subgroup contained in the $\sym(E_{i,j})$ component of $\simp_{\ev}^{(k)}(\Gamma_{A})$.

Let $e \ne \alpha \in N$ and choose $K\geq 1$ such that $\alpha \in \simp_{\ev}(\Gamma^{(K)}_{A})$. By passing to larger $K$ if necessary, since $A$ is primitive we may assume that all entries in
$A^{K}$ are greater than or equal to five.  We claim that for any $i,j\geq 1$ and for all $m$ sufficiently large, we have $N \cap \alt^{(Km)}_{i,j} \ne \{e\}$. Since $\alpha$ is nontrivial,
for some choice of $I,J$ we have that
$\alpha_{I,J}$, the component of $\alpha$ in $\alt^{(K)}_{I,J}$, is also nontrivial.
Choose a path $\gamma$ of length $m \ge 3$ in $\Gamma^{(K)}_{A}$ such that $\gamma$ begins at $i$, ends at $j$, and passes through an edge from $I$ to $J$ on which $\alpha_{I,J}$ acts nontrivially. Then $\gamma$ corresponds to an edge in $\Gamma^{(Km)}_{A}$ starting at vertex $i$ and ending at vertex $j$ on which $i_{K,m}(\alpha_{I,J})$ acts nontrivially. It follows that
\begin{equation}\label{eqn:smallpiece}
N \cap \alt_{i,j}^{(Km)}
\end{equation}
is nontrivial, proving the claim.

Since each entry of $A^{K}$ is at least 5, it follows that $\alt^{(Km)}_{i,j}$ is simple for all $i,j\geq 1$ and $m \geq 3$. Moreover $N$ is normal in $\simpinfev{A}$, and so $N \cap \alt_{i,j}^{(Km)}$ is normal in $\alt_{i,j}^{(Km)}$. Thus, since the
intersection in~\eqref{eqn:smallpiece} is nontrivial, it follows that
for all $i,j\geq 1$ and $m \geq 3$ we have that $\alt_{i,j}^{(Km)} \subset N$.
Therefore, $N$ contains the subgroup generated by the collection of subgroups
$$\Bigl\{\simp_{\ev}(\Gamma^{(Km)}_{A})\Bigr\}_{m = 3}^{\infty}.$$
Given any $r\geq 1$, there exists $M \ge 3$ such that $r$ divides $M$, so the subgroup $\simp(\Gamma^{(KM)}_{A})$ contains the subgroup $\simp(\Gamma^{(r)}_{A})$. It follows that $\simpinfev{A}$ is contained in the group generated by the collection
$$\Bigl\{\simp_{\ev}(\Gamma^{(Km)}_{A})\Bigr\}_{m = 3}^{\infty}$$
and hence
$$\simpinfev{A} \subset N,$$
proving Part~\eqref{item:parttwo}.

For Part~\eqref{item:partthree}, let $l \ge 1$
and suppose $\iota \in \simp(\Gamma^{(l)}_{n})$ is an order two automorphism induced
by the simple graph symmetry of $\Gamma^{(l)}_{n}$ which permutes two edges $e$ and $f$ and leaves all other edges fixed. We claim $i_{l,2}(\iota) \in \simp_{ev}({\Gamma^{(2l)}_{n}})$ (recall that the inclusion map $i_{l,2}$ is defined in~\eqref{eq:def-im}). To check this, observe that $i_{l,2}(\iota)$ is induced by the action of $\iota$ on paths of length two in $\Gamma^{(l)}_{n}$ of the form $ab$, where at least one of $a$ or $b$ is either $e$ or $f$. The action of $i_{l,2}(\iota)$ on such pairs of words is given by the composition of $2n-2$ transpositions, and it follows that $i_{l,2}(\iota) \in \simp_{ev}({\Gamma^{(2l)}_{n}})$, proving the claim.
Since such involutions generate all of $\simpinf{n}$, the equality in Part~\eqref{item:partthree} follows.
\end{proof}

It follows from Parts~\eqref{item:parttwo} and~\eqref{item:partthree} of Lemma~\ref{lemma:joel} that for a full shift $A = (n)$, $\simpinf{n}$ is a simple group.
\begin{lemma}\label{lemma:simplconjugates}
If $(X_{A},\sigma_{A})$ is a mixing shift of finite type, then:
\begin{enumerate}
\item
\label{item:oneone}
For any $\alpha \in \autinf{A}$, the group $\alpha \simpinfev{A} \alpha^{-1}$ is a simple subgroup of $\inertinf{A}$. Moreover, if $N$ is a normal subgroup in $\inertinf{A}$ such that
$$\alpha \simpinfev{A} \alpha^{-1} \cap N \ne \{e\},$$
then
$$\alpha \simpinfev{A} \alpha^{-1} \subset N.$$
\item
If for some $m_{1} \ge 0$
$$\sigma_{A}^{m_{1}}\simpinfev{A}\sigma_{A}^{-m_{1}} \subset N,$$
then for any $m \ge 0$
$$\sigma_{A}^{m}\simpinfev{A}\sigma_{}^{-m} \subset N.$$
\end{enumerate}
\end{lemma}
\begin{proof}
The first part follows immediately from Lemma~\ref{lemma:joel}. For the second part, since $\simp^{(\infty)}(\Gamma^{(l)}_{A}) = \simpinfev{A}$ for any $l \ge 1$ and $A$ is primitive, we may assume without loss of generality that for all $m \ge 1$, $A^{m}$ contains an entry which is strictly greater than $2$. Then
$$\sigma_{A}^{m_{1}}\simpinfev{A}\sigma_{A}^{-m_{1}} \cap \simpinfev{A} \ne \{e\},$$
since there exists some $\gamma \in \simp_{\textnormal{ev}}(\Gamma_{A})$ which commutes with $\sigma_{A}$, and hence with $\sigma_{A}^{m_{1}}$. Then since
$$\sigma_{A}^{m_{1}}\simpinfev{A}\sigma_{A}^{-m_{1}} \subset N,$$
it follows that
$$\simpinfev{A} \cap N \ne \{e\}.$$
Part~\eqref{item:oneone} now implies
$$\simpinfev{A} \subset N.$$
Given $m \ge 1$, since $A^{m}$ contains an entry strictly greater than $2$, the group $\simp_{\ev}(\Gamma^{(m)}_{A})$ is nontrivial.  Thus we have that
$$\sigma_{A}^{m} \simpinfev{A} \sigma_{A}^{-m} \cap \simpinfev{A} \ne \{e\}$$
and hence
$$\sigma_{A}^{m} \simpinfev{A} \sigma_{A}^{-m} \cap N \ne \{e\}.$$
Part~\eqref{item:oneone} then implies that
$$\sigma_{A}^{m} \simpinfev{A} \sigma_{A}^{-m} \subset N,$$
as desired.
\end{proof}

Finally, we use a lemma of Boyle, which is a stronger version of Wagoner's Theorem (Theorem~\ref{thm:efog}):
\begin{lemma}[Boyle~\cite{Boyle1988}]\label{lemma:mike0}
Let $(X_{A},\sigma_{A})$ be a mixing shift of finite type and suppose $\alpha \in \inertinf{A}$. There exists $m_{1},m_{2} \ge 1$ and $\psi_{1},\psi_{2} \in \simp(\Gamma_{A}^{(m_{1})})$ such that $\alpha = \psi_{1} \sigma_{n}^{m_{2}} \psi_{2} \sigma_{n}^{-m_{2}}$.
\end{lemma}

We have now assembled the ingredients to prove Theorem~\ref{th:even-simple}:
\begin{proof}[Proof of Theorem~\ref{th:even-simple}]
Suppose $N$ is a nontrivial normal subgroup of $\inertinf{n}$. By Lemma~\ref{lemma:lemma1}, there exists $m_{1} \ge 1$ such that
$$\sigma_{n}^{m_{1}}\simpinf{n}\sigma_{n}^{-m_{1}} \cap N \ne \{e\}.$$
Since $\simpinf{n} = \simpinfev{n}$ by Part~\eqref{item:partthree} of Lemma~\ref{lemma:joel}, we have that
$$\sigma_{n}^{m_{1}}\simpinfev{n}\sigma_{n}^{-m_{1}} \cap N \ne \{e\}.$$
Hence, by Lemma~\ref{lemma:simplconjugates},
$$\sigma_{n}^{m_{1}} \simpinfev{n} \sigma_{n}^{-m_{1}} \subset N.$$
Applying Lemma~\ref{lemma:simplconjugates} again, it follows that $N$ contains $\sigma_{n}^{-m}\simpinf{n}\sigma_{n}^{m}$ for all $m \ge 0$. By Lemma~\ref{lemma:mike0}, the collection of subgroups $\sigma_{n}^{-m}\simpinf{n}\sigma_{n}^{m}$, $m \ge 0$, generate $\inertinf{n}$, completing the proof.
\end{proof}

\subsection{Proof of Lemma~\ref{lemma:lemma1}}
\label{sec:proof of lemma1}
\subsubsection{Notation}
We start with some notation used in the proof of Lemma~\ref{lemma:lemma1} and we maintain this notation for the remainder of this section.

For $m \ge 1$ let $E^{(m)}(\Gamma_{n})$ denote the edge set of $\Gamma^{(m)}_{n}$.
Label the edges of $E^{(1)}(\Gamma_{n})$ by $\{1,2,\ldots,n\}$.
Note that we may label the edge sets $E^{(m)}(\Gamma_{n})$ such that for all $m \ge 2$,
$$E^{(m)}(\Gamma_{n}) = \prod_{i=1}^{m}E^{(1)}(\Gamma_{n}).$$

When working with $E^{(2)}(\Gamma_{n})$ for some $\Gamma_{n}$, we  denote points in $E^{(2)}(\Gamma_{n})$ by $\etwo{x_{1}}{y_{1}}$ where $x_{1},y_{1} \in E^{(1)}(\Gamma_{n})$. We refer to rows and columns of $E^{(2)}(\Gamma_{n})$, with the convention that row $i$ of $E^{(2)}(\Gamma_{n})$ refers to the set of points in $E^{(2)}(\Gamma_{n})$ of the form
$$\Bigl\{\etwo{i}{y}\colon y \in E^{(1)}(\Gamma_{n})\Bigr\},$$
while column $i$ refers to the set of points in $E^{(2)}(\Gamma_{n})$ of the form
$$\Bigl\{\etwo{x}{i}\colon x \in E^{(1)}(\Gamma_{n})\Bigr\}.$$

Assume $(X_{n},\sigma_{n})$ is a full shift and let $\mathcal{A}_{\sigma_{n}}$ denote the corresponding alphabet for the shift space.
By definition, $\mathcal{A}_{\sigma_{n}} = E^{(1)}(\Gamma_{n})$.
Thus, for $m \ge 1$, we identify the alphabet $\mathcal{A}_{\sigma_{n}^{m}}$ with the set of elements of the form $\begin{pmatrix} a_{0} \\ \vdots \\ a_{m-1} \end{pmatrix}$ where $a_{i} \in \mathcal{A}_{\sigma_{n}}$ for $i=1, \ldots, m-1$.

Given a point $x \in X$, as usual we write $x = (x_i)_{i\in\Z}$.
When we need to indicate where $x_0$ is located, we use a dot to indicate this; thus the point
$$x = \smalldots a \overset{\scriptscriptstyle \bullet}{b} c \smalldots$$
has $x_{0} = b$.

Given any $a \in \mathcal{A}_{\sigma_n}$, let $p_{a}$ denote the point $\smalldots aaa \smalldots $,
which is fixed by $\sigma_{n}$.

We let $P_{k}(\sigma_{n})$ denote the set of $k$-periodic points for $\sigma_{n}$, so $P_{k}(\sigma_{n})$ consists of all points $x$ for which $\sigma_{n}^{k}(x) = x$ (note that $P_{k}(\sigma_{n})$ in general contains, but is \emph{not} equal to, the set of points of \emph{least} period $k$). We can identify $P_{k}(\sigma_{n})$ with $E^{(k)}(\Gamma_{n})$, and similarly, given $m \ge 1$,  we can identify $P_{k}(\sigma^{m}_{n})$ with $E^{(k)}(\Gamma^{(m)}_{n})$.

To avoid overly cumbersome notation, we often suppress the $n$, writing $\Gamma$ and $\sigma$ instead of $\Gamma_n$ and $\sigma_n$, with the understanding that we are still working with a full shift on $n$ symbols.

Thus for the remainder of this section, we assume $(X_n, \sigma_n)$ is a full shift on $n\geq 2$ symbols,
and without loss of generality, we assume that $n\geq 7$.  This is not a restrictive assumption, as
in the stabilized setting, $\inert^{(\infty)}(\sigma_{n}) \cong \inert^{(\infty)}(\sigma_{n}^{m}) \cong \inert^{(\infty)}(\sigma_{n^{m}})$ for any $m \ge 1$.

Finally, for the remainder of this section, we fix a nontrivial normal subgroup $N$ of $\inert^{(\infty)}(\sigma_{n})$, and
our goal is to prove Lemma~\ref{lemma:lemma1},
showing that there exists $m \ge 0$ and $\id \ne \zeta \in \simpinf{n}$ such that $\sigma_{n}^{m} \zeta \sigma_{n}^{-m} \in N$.

\subsubsection{Existence of an inert with additional properties}
We start by recording a slightly stronger version of Lemma~\ref{lemma:mike0}:
\begin{lemma}[Boyle~\cite{Boyle1988}]\label{lemma:mike}
Suppose $\alpha \in \inertinf{n}$. There exists $M \ge 1$ such that for all $m \ge M$, there exist $\psi_{1}^{(m)},\psi_{2}^{(m)} \in \simp(\Gamma_{n}^{(2m)})$ such that $\alpha = \psi_{1}^{(m)}\sigma_{n}^{m}\psi_{2}^{(m)}\sigma_{n}^{-m}$.
\end{lemma}
\begin{proof}
This can be deduced from the proof of~\cite[Theorem, pg. 970]{Boyle1988} (in the notation used in the proof there, for $m$ large enough, we can choose $n = 2p-k+m \ge 0$, so that $t = k+m+n = 2p+2m = 2(p+m)$).
\end{proof}

From here on, we usually suppress the $n$ and just write $\sigma$ for $\sigma_{n}$.

Suppose $\alpha \in \inert(\sigma)$ and that $\alpha$ is induced by a block code $h_{\alpha}$ of range $r \ge 1$; thus $h_{\alpha} \colon \mathcal{A}_{\sigma}^{2r+1} \to \mathcal{A}_{\sigma}$.
We say that
\begin{equation}
\label{eqn:snowishere1}
\tag{*}
\alpha \text{ satisfies property~\eqref{eqn:snowishere1} if
there exist distinct } a,b,c \in \mathcal{A}_{\sigma} \text{ such that }
\end{equation}
\begin{enumerate}
\item
\label{condition:one}
$\alpha(p_{a}) = p_{a}$;
\item
\label{condition:two}
$h_{\alpha} \left(a^{r}a ba^{r-1} \right) \ne a \in \mathcal{A}_{\sigma}$;
\item
\label{condition:three}
For all $0 \le i \le r$, $h_{\alpha}(a^{r-i}b a^{r+i}) = a$ and $h_{\alpha}(a^{2r-i}c a^{i})=a$.
\end{enumerate}

\begin{lemma}\label{lemma:innandonp}
Suppose $\alpha \in \inert(\sigma)$ is induced by a block code $h_{\alpha}$ of range $r$
and satisfies~\eqref{eqn:snowishere1} for some $a,b,c \in \mathcal{A}_{\sigma}$. Then there exists $m \ge 1$ such that, upon vieweing $\alpha$
as an element of $\inert(\sigma^{2m})$, all of the following hold:
\begin{enumerate}
\item
\label{it:one1}
For some $\psi_{1}^{(m)},\psi_{2}^{(m)} \in \simp(\Gamma^{(2m)})$, we have $\alpha = \psi_{1}^{(m)}\sigma^{m}\psi_{2}^{(m)}\sigma^{-m}$;
\item
\label{it:two2}
$\alpha(p_{a}) = p_{a}$;
\item
\label{it:three3}
For $w=ba^{m-2}c$, the point $p_{aw} = \ldots a^{m}w \overset{\scriptscriptstyle \bullet}{a}a^{m-1}w \ldots$ is a point of least period two for $\sigma^{m}$,
and in particular, $\alpha(p_{aw}) \in P_{2}(\sigma^{m})$;
\item
\label{it:four4}
The point $\alpha(p_{aw})$ in $P_{2}(\sigma^{m})$ satisfies $\left(\alpha(p_{aw})\right)_{m-1} \ne a$ and satisfies
$\left(\alpha(p_{aw})\right)_{i} = a$ for all $m \le i \le 2m-1$.
\end{enumerate}
Furthermore, using the identification of
$P_{2}(\sigma^{m})$ and $E^{(2)}(\Gamma^{(m)})$, we have the following:
\begin{enumerate}[label={(\alph*})]
\item
\label{it:parta}
$\alpha \etwo{a^{m}}{a^{m}} = \etwo{a^{m}}{a^{m}}$;
\item
\label{it:partb}
$\alpha \etwo{a^{m}}{w} = \etwo{w^{\prime}}{a^{m}}$ for some word $w^{\prime}$ of length $m$ where $w^{\prime} \ne a^{m}$.
\end{enumerate}
\end{lemma}
\begin{proof}
By  Lemma~\ref{lemma:mike},
Part~\eqref{it:one1} holds for all sufficiently large $m$, so in particular for some $m \ge 2r + 2$.
Part~\eqref{it:two2} is obvious, and  since $a,b,c$ are distinct, Part~\eqref{it:three3} follows.
To prove Part~\eqref{it:four4}, note that since $\alpha(p_{a}) = p_{a}$, it follows that $h_{\alpha}(a^{2r+1}) = a$.
Since $m \ge 2r + 2$, we have that $m-r-1 \ge r + 1$, and it follows that
$$\sigma^{m-1}(p_{aw}) = \ldots w \underbrace{a \ldots a}_{m-r-1}\underbrace{a \ldots a}_{r}\overset{\scriptscriptstyle \bullet}{a}w\ldots$$
Thus $\left(\alpha (p_{aw})\right)_{m-1}=\left(\sigma^{m-1}\alpha (p_{aw})\right)_{0}=\left(\alpha \sigma^{m-1}(p_{aw})\right)_{0} = h_{\alpha}(a^{r}aba^{r-1}) \ne a$.
Using Condition~\eqref{condition:three} of~\eqref{eqn:snowishere1}, it follows that $\left(\alpha(p_{aw})\right)_{i} = a$ for all $m \le i \le 2m-1$.

Parts~\ref{it:parta} and~\ref{it:partb} follow immediately by translating the results via the identification.
\end{proof}

Given symbols $a,b \in \mathcal{A}_{\sigma}$, we use the shorthand $a \leftrightarrow b$ to denote the $0$-block code involution in $\aut(\sigma)$ which permutes the symbols $a$ and $b$ and leaves all other symbols fixed.
\begin{lemma}\label{lemma:alphainN}
There exists $\alpha \in N$ satisfying property~\eqref{eqn:snowishere1}.
\end{lemma}

\begin{proof}
Suppose $\id \ne \alpha \in N$ and $\alpha \in \inert(\sigma^{\ell})$ for some $\ell\geq 1$. By passing to a larger $\ell$ if necessary, we may assume that $\alpha$ acts nontrivially on $P_{1}(\sigma^{\ell})$. Since $\inert(\sigma^{\ell})$ can induce any permutation on $P_{1}(\sigma^{\ell})$, and since $N$ is normal, by replacing $\alpha$ with some other $\alpha^{\prime} \in N$ if needed, we can assume that $\alpha$ satisfies:
\begin{align*}
& \alpha(p_{A}) = p_{A} \text{ for some }p_{A} \in P_{1}(\sigma^{\ell}) \text{ with } A \in \mathcal{A}_{\sigma^{\ell}},\\
&\alpha(p_{D_{1}}) = p_{D_{2}} \text{ for some }p_{D_{1}},p_{D_{2}} \in P_{1}(\sigma^{\ell}) \text{ with } D_{1},D_{2} \in \mathcal{A}_{\sigma^{\ell}}, \\
&\alpha(p_{E_{1}}) = p_{E_{2}} \text{ for some }p_{E_{1}},p_{E_{2}} \in P_{1}(\sigma^{\ell}) \text{ with } E_{1},E_{2} \in \mathcal{A}_{\sigma^{\ell}},\\
\text{ and } & A,D_{1},D_{2},E_{1},E_{2} \text{ are all distinct.}
\end{align*}

Suppose $\alpha$ is induced by a block code $h_{\alpha}$ of range $r$.  Without loss of generality,
we may assume that $r \ge 1$ (if $r=0$, the conclusion of Lemma~\ref{lemma:lemma1} already holds).

Set $k = 2\ell r+1$. By considering $\alpha$ as an element of $\inert(\sigma^{k})$, we may assume that $\alpha$ is given by a block code $h^{(k)}_{\alpha}$ of range $1$.

Consider the words
$$v_d = \begin{pmatrix} A^{k} \\ D_{1}^{k} \\ A^{k} \end{pmatrix}  \begin{pmatrix} A^{k} \\ A^{k} \\ A^{k} \end{pmatrix} \begin{pmatrix} A^{k} \\ D_{1}^{k} \\ A^{k} \end{pmatrix} $$
and
$$
v_e =  \begin{pmatrix} A^{k} \\ E_{1}^{k} \\ A^{k} \end{pmatrix}  \begin{pmatrix} A^{k} \\ A^{k} \\ A^{k} \end{pmatrix} \begin{pmatrix} A^{k} \\ E_{1}^{k} \\ A^{k} \end{pmatrix}
$$
of length three over the alphabet $\mathcal{A}_{\sigma^{3k}}$.
Viewing $\alpha$ as an automorphism lying in $\inert(\sigma^{3k})$,
we have that $\alpha$ is induced by some block $h^{(3k)}_{\alpha}$ of radius one, and this block code satisfies
$$h^{(3k)}_{\alpha}(v_{d}) = \begin{pmatrix} A^{k} \\ A^{k} \\ A^{k} \end{pmatrix},  \quad h^{(3k)}_{\alpha}(v_{e}) = \begin{pmatrix} A^{k} \\ A^{k} \\ A^{k} \end{pmatrix} , $$
while
\begin{equation}\label{eqn:hD1}
h^{(3k)}_{\alpha} \left( \begin{pmatrix} A^{k} \\ A^{k} \\ A^{k} \end{pmatrix}
\begin{pmatrix} A^{k} \\ D_{1}^{k} \\ A^{k} \end{pmatrix}
\begin{pmatrix} A^{k} \\ A^{k} \\ A^{k} \end{pmatrix}
\right) = \begin{pmatrix} * \\ D_{2} \\ * \end{pmatrix} \ne \begin{pmatrix} A^{k} \\ D_{1}^{k} \\ A^{k} \end{pmatrix},
\end{equation}
\begin{equation}\label{eqn:hE1}
h^{(3k)}_{\alpha} \left( \begin{pmatrix} A^{k} \\ A^{k} \\ A^{k} \end{pmatrix}  \begin{pmatrix} A^{k} \\ E_{1}^{k} \\ A^{k} \end{pmatrix} \begin{pmatrix} A^{k} \\ A^{k} \\ A^{k}\end{pmatrix}  \right) = \begin{pmatrix} * \\ E_{2} \\ * \end{pmatrix} \ne \begin{pmatrix} A^{k} \\ E_{1}^{k} \\ A^{k} \end{pmatrix}
\end{equation}
(note that $h_{\alpha}(D_{1}^{r}) = D_{2} \ne D_{1}$ and $h_{\alpha}(E_{1}^{r}) = E_{2} \ne E_{1}$).

Define the words
$$w_{d} =    \begin{pmatrix} D_1^{k} \\ D_1^{k} \\ A^{k} \end{pmatrix} \begin{pmatrix} A^{k} \\ A^{k} \\ A^{k} \end{pmatrix}   \begin{pmatrix} A^{k} \\ D_1^{k} \\ D_1^{k} \end{pmatrix}, \quad
w_{e} =    \begin{pmatrix} E_1^{k} \\ E_1^{k} \\ A^{k} \end{pmatrix}
\begin{pmatrix} A^{k} \\ A^{k} \\ A^{k} \end{pmatrix}
  \begin{pmatrix} A^{k} \\ E_1^{k} \\ E_1^{k} \end{pmatrix}$$
and note that $h^{(3k)}_{\alpha}(w_{d}) = \begin{pmatrix} A^{k} \\ A^{k} \\ A^{k} \end{pmatrix}$ and $h^{(3k)}_{\alpha}(w_{e}) = \begin{pmatrix} A^{k} \\ A^{k} \\ A^{k} \end{pmatrix}$.

We set convenient notation for some letters in $\mathcal{A}_{\sigma^{3k}}$: given $X \in \mathcal{A}_{\sigma^{k}}$ we define
$$x = \begin{pmatrix} X^{k} \\ X^{k} \\ X^{k} \end{pmatrix}.$$
Thus for example
$$a = \begin{pmatrix} A^{k} \\ A^{k} \\ A^{k} \end{pmatrix}.$$
Choose  $b,c \in \mathcal{A}_{\sigma^{3k}}$ such that $a,b,c,d_{1},d_{2},e_{1},e_{2}$ are all distinct and such that $h^{(3k)}_{\alpha}(aac) \ne b$ (this is possible since, for example, $h^{(3k)}_{\alpha}(aac)$ contains letters from the original alphabet).

Define the automorphism $\beta_{1} \in \inert(\sigma^{9k})$ by
$$\beta_{1} = \sigma^{3k} \left( e_{1}e_{1}e_{1} \leftrightarrow aab \right) \sigma^{-3k}$$
(note that this is the conjugacy by $\sigma^{3k}$ of the involution $e_{1}e_{1}e_{1} \leftrightarrow aab$) and let $\alpha_{1} = \beta_{1}^{-1} \alpha \beta_{1}$. Then $\alpha_{1} \in N$, and can be induced by a block code of range $4$ on the alphabet $\mathcal{A}_{\sigma^{3k}}$.
Furthermore, we have
$$\smalldots a^{4}\overset{\scriptscriptstyle \bullet}{a}ba^{3} \smalldots
\stackrel{\beta_{1}}\longrightarrow \smalldots a^{3}e_{1}\overset{\scriptscriptstyle \bullet}{e_{1}}e_{1}a^{3} \smalldots \stackrel{\alpha}\longrightarrow \smalldots \overset{\scriptscriptstyle \bullet}{e_{2}} \smalldots \stackrel{\beta_{1}^{-1}}\longrightarrow \smalldots \overset{\scriptscriptstyle \bullet}{e_{2}} \smalldots$$
and $\beta_{1}(p_{a}) = p_{a}$, and so $\alpha_{1}$ satisfies conditions~\eqref{condition:one} and~\eqref{condition:two} of~\eqref{eqn:snowishere1} for the letters $a, b$.

Define the automorphism $\beta_{2} \in \inert(\sigma^{9k})$ by $\beta_{2} = \sigma^{3k} \beta_{2}^{\prime} \sigma^{-3k}$, where $\beta_{2}^{\prime}$ is the $0$-block code involution on the alphabet $\mathcal{A}_{\sigma^{3k}}$ which performs the following permutation on symbols
\begin{equation}
\beta_{2}^{\prime} \colon
\begin{cases}
aba \leftrightarrow v_{d}\\
baa \leftrightarrow w_{d}\\
aac \leftrightarrow w_{e}\\
aca \leftrightarrow v_{e}\\
\end{cases}
\end{equation}
and consider $\alpha_{2} = \beta_{2}^{-1}\alpha_{1}\beta_{2}$. Then $\alpha_{2} \in N$, and still satisfies conditions~\eqref{condition:one} and~\eqref{condition:two} of~\eqref{eqn:snowishere1}. To see that it satisfies condition~\eqref{condition:three} is a matter of checking case by case. For example,
$$\smalldots a^{3}a \overset{\scriptscriptstyle \bullet}{b} a a^{3} \smalldots \stackrel{\beta_{2}}\longrightarrow \smalldots a^{3}\overset{\scriptscriptstyle \bullet}{v_{d}}a^{3} \smalldots \stackrel{\alpha_{1}}\longrightarrow \smalldots *\overset{\scriptscriptstyle \bullet}{a} \smalldots \stackrel{\beta_{2}^{-1}}\longrightarrow \smalldots *\overset{\scriptscriptstyle \bullet}{a} \smalldots$$
since, by~\eqref{eqn:hD1}, $*$ is some word containing $D_{2}$'s. Next,
$$\smalldots a^{3}b \overset{\scriptscriptstyle \bullet}{a} a a^{3} \smalldots \stackrel{\beta_{2}}\longrightarrow \smalldots a^{3}\overset{\scriptscriptstyle \bullet}{w_{d}}a^{3} \smalldots \stackrel{\alpha_{1}}\longrightarrow \smalldots *\overset{\scriptscriptstyle \bullet}{a} \smalldots \stackrel{\beta_{2}^{-1}}\longrightarrow \smalldots *\overset{\scriptscriptstyle \bullet}{a} \smalldots$$
since $*$ also contains some $D_{2}$'s. Furthermore,
$$\smalldots a^{3}a \overset{\scriptscriptstyle \bullet}{c} a a^{3} \smalldots \stackrel{\beta_{2}}\longrightarrow \smalldots a^{3}\overset{\scriptscriptstyle \bullet}{v_{e}}a^{3} \smalldots \stackrel{\alpha_{1}}\longrightarrow \smalldots *\overset{\scriptscriptstyle \bullet}{a} \smalldots \stackrel{\beta_{2}^{-1}}\longrightarrow \smalldots *\overset{\scriptscriptstyle \bullet}{a} \smalldots$$
since, by~\eqref{eqn:hE1}, $*$ contains $E_{2}$'s, and
$$\smalldots a^{3}a \overset{\scriptscriptstyle \bullet}{a}c a^{3} \smalldots \stackrel{\beta_{2}}\longrightarrow \smalldots a^{3}\overset{\scriptscriptstyle \bullet}{w_{e}}a^{3} \smalldots \stackrel{\alpha_{1}}\longrightarrow \smalldots *\overset{\scriptscriptstyle \bullet}{a} \smalldots \stackrel{\beta_{2}^{-1}}\longrightarrow \smalldots *\overset{\scriptscriptstyle \bullet}{a} \smalldots$$
since $*$ contains some $E_{2}$'s.
\end{proof}

Combining Lemmas~\ref{lemma:innandonp} and~\ref{lemma:alphainN}, we obtain the existence of an automorphism, which for convenience we also denote by $\alpha$, with $\alpha \in N$, such that $\alpha$ satisfies the conditions in  Lemma~\ref{lemma:innandonp} for some  $m\geq 1$. The automorphism $\alpha$ constructed in~\ref{lemma:alphainN} also satisfies an additional property that we note for use in the sequel: there exists some word $z_{1}$ (for example, let $z_{1} = be_{1}^{m-1}$) such that, with the symbol $a$ given by Lemma~\ref{lemma:alphainN}, writing $\alpha \etwo{a^{m}}{z_{1}} = \etwo{x}{y}$, we have $x \ne a^{m}$ and $y \ne a^{m}$.

For ease of notation, for the remainder of the section we suppress the power $m$, and write $\sigma$ instead of $\sigma^{m}$, and  write $\psi_{1},\psi_{2}$ for the simple automorphisms
$\psi_{1}^{(m)}, \psi_{2}^{(m)}$ produced by Lemma~\ref{lemma:innandonp}.

It is convenient to recode the alphabet for our shift, and to do so we choose a bijection $\mathcal{A}_{\sigma} \leftrightarrow \{1,2,\ldots,n\}$ such that $1 \mapsto a^{m}$, and let $\{1,2,\ldots,n\}$ be the alphabet of our shift.
Summarizing, we have shown:
\begin{lemma}
\label{eqn:setupconditions}
There exists
$\alpha \in \inert(\sigma^{2})$ satisfying the following properties:
\begin{enumerate}
\item
\label{item:alpha1}
$\alpha \in N$;
\item
\label{item:alpha2}
$\alpha = \psi_{1}\sigma\psi_{2}\sigma^{-1}$, for some $\psi_{1},\psi_{2} \in \simp(\Gamma^{(2)})$;
\item
\label{item:alpha3}
$\alpha \etwo{1}{1} = \etwo{1}{1}$;
\item
\label{item:alpha4}
$\alpha \etwo{1}{u_{1}} = \etwo{u_{2}}{1}$ for some $1 \ne u_{1}$ and some $u_{2} \in \{1,2,\ldots,n\}$;
\item
\label{item:alpha5}
there exists $u_{3} \in \{1,2,\ldots,n\}$ such that neither component of $\alpha \etwo{1}{u_{3}}$ is 1.
\end{enumerate}
\end{lemma}

\subsubsection{Constructing a particular subgroup $\mathcal{K}$ of $\sym(E^{(2)}) \times \sym(E^{(2)})$}
Consider the set
\begin{equation}
\label{def:KN}
\mathcal{K}_{N} = \{(\phi_{1},\phi_{2}) \in  \simp(\Gamma^{(2)}) \times \simp(\Gamma^{(2)}) \colon \phi_{1}\sigma\phi_{2}^{-1}\sigma^{-1} \in N\}.
\end{equation}
\begin{lemma}
The set $\mathcal{K}_{N}$ defined in~\eqref{def:KN} is a subgroup of $\simp(\Gamma^{(2)}) \times \simp(\Gamma^{(2)})$.
\end{lemma}
\begin{proof}
Assume $\phi_{1}\sigma\phi_{2}^{-1}\sigma^{-1}, \phi_{3}\sigma\phi_{4}^{-1}\sigma^{-1} \in N$.  Then $\sigma\phi_{4}^{-1}\sigma^{-1}\phi_{3} \in N$, and hence
$$\sigma \phi_{4}^{-1} \sigma^{-1}\phi_{3} \phi_{1} \sigma \phi_{2}^{-1} \sigma^{-1} \in N$$
and
$$\phi_{3} \phi_{1} \sigma \phi_{2}^{-1}\phi_{4}^{-1} \sigma^{-1} = \phi_{3} \phi_{1} \sigma (\phi_{4}\phi_{2})^{-1} \sigma^{-1} \in N.$$
Lastly, if $\phi_{1} \sigma \phi_{2}^{-1} \sigma^{-1} \in N$, then $\phi_{1}^{-1}\sigma \phi_{2} \sigma^{-1} = \phi_{1}^{-1}\sigma \phi_{2} \sigma^{-1} \phi_{1}^{-1} \phi_{1} \in N$.
\end{proof}

To simplify notation, for the remainder of this section we write
$E^{(m)}$ instead of $E^{(m)}(\Gamma)$.
By definition, $E^{(2)}$ is the edge set of $\Gamma^{(2)}$, so there is an isomorphism
\begin{equation}
\label{def:H}
\mathcal{H} \colon \simp(\Gamma^{(2)}) \longrightarrow \sym(E^{(2)})
\end{equation}
and hence an isomorphism
$$\mathcal{H} \times \mathcal{H} \colon \simp(\Gamma^{(2)}) \times \simp(\Gamma^{(2)}) \longrightarrow \sym(E^{(2)}) \times \sym(E^{(2)}).$$
Define
\begin{equation}
\label{def:K}
\mathcal K = (\mathcal H\times \mathcal H)(\mathcal K_N),
\end{equation}
meaning that $\mathcal{K}$ is the image of $\mathcal{K}_{N}$ under this isomorphism. Thus we have
$$\mathcal{K} \subset \sym(E^{(2)}) \times \sym(E^{(2)}).$$

Letting $\alpha \in \inert(\sigma^{2})$ be the element of $N$ satisfying Lemma~\ref{eqn:setupconditions},
and maintaining the notation of that lemma, we have $\alpha = \psi_{1}\sigma\psi_{2}\sigma^{-1}$, for some $\psi_{1},\psi_{2} \in \simp(\Gamma^{(2)})$ so
$$(\psi_{1},\psi_{2}^{-1}) \in \mathcal{K}_{N}.$$
Defining
\begin{equation}
\label{eq:gammas}
\gamma_{1} = \mathcal{H}(\psi_{1}), \quad \gamma_{2} = \mathcal{H}(\psi_{2}),
\end{equation}
it follows that
\begin{equation}\label{eqn:apairinK}
(\gamma_{1},\gamma_{2}^{-1}) \in \mathcal{K}.
\end{equation}

Recall we have $E^{(2)} = E^{(1)} \times E^{(1)}$, and we write points in $E^{(2)}$ as $\begin{pmatrix}x \\ y \end{pmatrix}$ where $x,y \in E^{(1)}$. We embed $\sym(E^{(1)}) \times \sym(E^{(1)})$ into $\sym(E^{(2)})$ via the map
\begin{equation}
(\phi_{1},\phi_{2}) \mapsto \begin{pmatrix} \phi_{1} \\ \phi_{2} \end{pmatrix},
\end{equation}
where $\begin{pmatrix} \phi_{1} \\ \phi_{2} \end{pmatrix}\begin{pmatrix} x \\ y \end{pmatrix} = \begin{pmatrix} \phi_{1}(x) \\ \phi_{2}(y) \end{pmatrix}$.
Define
\begin{equation}
\label{def:P}
P \text{ to be the subgroup of $\sym(E^{(2)})$ that is the image of this embedding.}
\end{equation}

\begin{lemma}\label{lemma:PwithK}
For any $\begin{pmatrix} \phi_{1} \\ \phi_{2} \end{pmatrix} \in P$, we have $\left( \begin{pmatrix} \phi_{1} \\ \phi_{2} \end{pmatrix}, \begin{pmatrix} \phi_{2} \\ \phi_{1} \end{pmatrix} \right) \in \mathcal{K}$.
\end{lemma}
\begin{proof}
Let $\begin{pmatrix} \tilde{\phi}_{1} \\ \tilde{\phi}_{2} \end{pmatrix} \in \simp(\Gamma^{(2)})$ be the automorphism induced by the permutation $\begin{pmatrix} \phi_{1} \\ \phi_{2} \end{pmatrix}$ on the edge set $E^{(2)}(\Gamma)$.
Thus $\begin{pmatrix} \phi_{1} \\ \phi_{2} \end{pmatrix} = \left( \mathcal{H} \times \mathcal{H} \right) \begin{pmatrix} \tilde{\phi}_{1} \\ \tilde{\phi}_{2} \end{pmatrix}$. It is straightforward to check that
$$\begin{pmatrix} \tilde{\phi}_{1} \\ \tilde{\phi}_{2} \end{pmatrix} \sigma \begin{pmatrix} \tilde{\phi_{2}}^{-1} \\ \tilde{\phi}_{1}^{-1} \end{pmatrix} \sigma^{-1} = \begin{pmatrix} \tilde{\phi}_{1} \\ \tilde{\phi}_{2} \end{pmatrix} \begin{pmatrix} \tilde{\phi}_{1}^{-1} \\ \tilde{\phi}_{2}^{-1} \end{pmatrix} = \id \in N$$
so
\begin{equation*}\left( \begin{pmatrix} \tilde{\phi}_{1} \\ \tilde{\phi}_{2} \end{pmatrix} , \begin{pmatrix} \tilde{\phi}_{2} \\ \tilde{\phi}_{1} \end{pmatrix} \right) \in \mathcal{K}_{N}.\quad\qedhere\end{equation*}
\end{proof}

Define the swapping element $\mathfrak{s} \in \sym(E^{(2)})$ by
\begin{equation}
\label{def:swap}
\mathfrak{s}\begin{pmatrix} p \\ q \end{pmatrix} = \begin{pmatrix} q \\ p \end{pmatrix}.
\end{equation}

Recall we can identify period two points for $\sigma$ with the set $E^{(2)}$. Then $\sigma$ induces an action on $E^{(2)}$, and this action agrees with the action of $\mathfrak{s}$ on $E^{(2)}$.

\begin{lemma}\label{lemma:separate}
For the elements $\gamma_1, \gamma_2$ defined in Equation~\eqref{eq:gammas}, we have
$\gamma_{2}^{-1} \ne \mathfrak{s}^{-1} \gamma_{1} \mathfrak{s}$.
\end{lemma}
\begin{proof}
By Lemma~\eqref{eqn:setupconditions}, $\alpha = \psi_{1}\sigma\psi_{2}\sigma^{-1}$ for some $\psi_{1},\psi_{2} \in \simp(\Gamma^{(2)})$.
If $\gamma_{2}^{-1} = \mathfrak{s}^{-1} \gamma_{1} \mathfrak{s}$, then $\alpha$ acts on $E^{(2)}$ by the permutation
$$\gamma_{1} \mathfrak{s} \mathfrak{s}^{-1} \gamma_{1}^{-1} \mathfrak{s} \mathfrak{s}^{-1} = \gamma_{1}\gamma_{1}^{-1} = \id.$$ But this contradicts Lemma~\ref{eqn:setupconditions}, as $\alpha$ acts nontrivially on $E^{(2)}$.
\end{proof}

\subsubsection{Completion of the proof of Lemma~\ref{lemma:lemma1}}

To translate properties of $\mathcal K$ to subgroups of $\sym(E^{(2)})$, we make use of the following result:
\begin{lemma}[Goursat's Lemma (see~\cite{lang})]
Let $G_{1},G_{2}$ be groups and let $H$ be a subgroup of $G_{1} \times G_{2}$. Then there exist subgroups $H_{1} \subset G_{1}, H_{2} \subset G_{2}$, normal subgroups $N_{1} \unlhd H_{1}, N_{2} \unlhd H_{2}$, and an isomorphism $\Psi \colon H_{1} / N_{1} \to H_{2} / N_{2}$ such that
$$H = \{(x,y) \in H_{1} \times H_{2} \colon \Psi([x]) = [y]\}.$$
\end{lemma}

Applying Goursat's Lemma to the group $\mathcal{K}$, we obtain:
\begin{corollary}
\label{cor:goursat}
Let $\mathcal K$ be the subgroup defined in~\eqref{def:K}.
There exist $H_{1},H_{2} \subset \sym(E^{(2)})$, normal subgroups $N_{1} \unlhd H_{1}, N_{2} \unlhd H_{2}$, and an isomorphism $\Psi \colon H_{1} / N_{1} \to H_{2} / N_{2}$ such that
$$\mathcal{K} = \{(\phi_{1},\phi_{2}) \in H_{1} \times H_{2} \colon \Psi([\phi_{1}]) = [\phi_{2}]\}.$$
\end{corollary}

We turn our attention then to studying the subgroups $H_{1}, H_{2}, N_{1}, N_{2}$. The key lemma regarding their structure is the following:
\begin{lemma}\label{lemma:obtainment}
Assume both subgroups $N_1$ and $N_2$ of Corollary~\ref{cor:goursat} are trivial.  Then at least one of the following holds:
\begin{enumerate}
\item
$H_{1} = \sym(E^{(2)})$ and $H_{2} = \sym(E^{(2)})$.
\item
$H_{1} = \alt(E^{(2)})$ and $H_{2} = \alt(E^{(2)})$.
\end{enumerate}
\end{lemma}
As the proof of this lemma is lengthy and involves checking multiple cases, we defer its proof to Section~\ref{sec:lemma-obtainment}.

For use in the proof of Lemma~\ref{lemma:lemma1}, we recall the following classical theorem:
\begin{theorem}\label{thm:amazingfact}
Suppose $|X| > 6$, $G$ is either $\sym(X)$ or $\alt(X)$, and $\Psi \colon G \to G$ is an automorphism. Then there exists $g \in \sym(X)$ such that $\Psi(h) = g^{-1}hg$ for all $h \in G$.
\end{theorem}

We have now assembled the tools to prove Lemma~\ref{lemma:lemma1} (modulo the deferral of the technical statement in Lemma~\ref{lemma:obtainment}):
\begin{proof}[Proof of Lemma~\ref{lemma:lemma1}]
Let $N_1, N_2$ be the subgroups produced in Corollary~\ref{cor:goursat} and let $\Psi \colon H_{1} / N_{1} \to H_{2} / N_{2}$ be the isomorphism in the same result.

Assume first that $N_{1} \ne \{\id\}$, so there is some  $\phi_1\neq \id$ with $\phi_{1} \in N_{1}$. Then $\Psi([\phi_{1}]) = \Psi([\id]) = [\id] \in H_{2}/N_{2}$, so $(\phi_{1},\id) \in \mathcal{K}$.
This implies that
$$\mathcal{H}^{-1}(\phi_{1})\sigma \sigma^{-1} = \mathcal{H}^{-1}(\phi_{1}) \in N.$$
But since $\mathcal{H}^{-1}(\phi_{1}) \in \simp(\Gamma)$, the statement of Lemma~\ref{lemma:lemma1} follows.
Likewise, if $N_{2} \ne \{\id\}$, then $(\id,\phi_{2}) \in \mathcal{K}$ for some $ \phi_{2} \in N_{2}$, and again the result follows.  Thus we are left with showing that  either $N_{1} \ne \{\id\}$ or $N_{2} \ne \{\id\}$.

We proceed by contradiction
and  suppose that both $N_{1} = \{\id\}$ and $N_{2} = \{\id\}$.
Combining Corollary~\ref{cor:goursat}
and Lemma~\ref{lemma:obtainment}, we have that the isomorphism $\Psi$ is either
$$\Psi \colon \sym(E^{(2)}) \to \sym(E^{(2)})$$
or
$$\Psi \colon \alt(E^{(2)}) \to \alt(E^{(2)}).$$
By Theorem~\ref{thm:amazingfact}, we have that $\Psi$ is given by $\Psi(h) = g^{-1}hg$ for some $g \in \sym(E^{(2)})$.

We claim that $g$ is the swap map
$\mathfrak{s}$, defined in~\eqref{def:swap}.

To check this claim,
note that for any $\begin{pmatrix} \phi_{1} \\ \phi_{2} \end{pmatrix} \in P$, where $P$ is defined in~\eqref{def:P}, it follows from Lemma~\ref{lemma:PwithK} that $\left( \begin{pmatrix} \phi_{1} \\ \phi_{2} \end{pmatrix}, \begin{pmatrix} \phi_{2} \\ \phi_{1} \end{pmatrix} \right) \in K$.  Thus
$$g^{-1} \begin{pmatrix} \phi_{1} \\ \phi_{2} \end{pmatrix} g = \begin{pmatrix} \phi_{2} \\ \phi_{1} \end{pmatrix}$$
and hence
\begin{equation}\label{eqn:relforgiss}
g^{-1} \mathfrak{s}\begin{pmatrix} \phi_{2} \\ \phi_{1} \end{pmatrix} \mathfrak{s}^{-1} g = \begin{pmatrix} \phi_{2} \\ \phi_{1} \end{pmatrix}
\end{equation}
for all $\phi_{1},\phi_{2} \in \sym(E^{(1)})$.
We now check that this implies that $\mathfrak{s}^{-1}g = \id$. If not,  there exists $\begin{pmatrix} x_{1} \\ y_{1} \end{pmatrix}, \begin{pmatrix} x_{2} \\ y_{2} \end{pmatrix} \in E^{(2)}$ such that $\mathfrak{s}^{-1}g\begin{pmatrix} x_{1} \\ y_{1} \end{pmatrix} = \begin{pmatrix} x_{2} \\ y_{2} \end{pmatrix}$ and $\begin{pmatrix} x_{1} \\ y_{1} \end{pmatrix} \ne \begin{pmatrix} x_{2} \\ y_{2} \end{pmatrix}$.
Either $x_{1} \ne x_{2}$ or $y_{1} \ne y_{2}$; assume $x_{1} \ne x_{2}$ (the other case is similar).
Choose  $z \in E^{(1)}$ such that $z \ne x_{1},x_{2}$ and define $\phi_{3} \in \sym(E^{(1)})$ to be the transposition swapping $x_{2}$ and $z$. Then
$$\begin{pmatrix} \phi_{3} \\ \id \end{pmatrix} \mathfrak{s}^{-1}g \begin{pmatrix} \phi_{3} \\ \id \end{pmatrix}^{-1} \begin{pmatrix} x_{1} \\ y_{1} \end{pmatrix} = \begin{pmatrix} \phi_{3} \\ \id \end{pmatrix} \begin{pmatrix} x_{2} \\ y_{2} \end{pmatrix} = \begin{pmatrix} z \\ y_{2} \end{pmatrix} \ne \begin{pmatrix} x_{2} \\ y_{2} \end{pmatrix} = \mathfrak{s}^{-1}g\begin{pmatrix} x_{1} \\ y_{1} \end{pmatrix}$$
contradicting~\eqref{eqn:relforgiss}, thus proving the claim.

Since  $(\gamma_{1},\gamma_{2}^{-1}) \in \mathcal{K}$ (see~\eqref{eqn:apairinK})we have $\gamma_{2}^{-1} = \Psi(\gamma_{1})$.  It then follows from the claim that
$\gamma_{2}^{-1} = \mathfrak{s}^{-1}\gamma_{1}\mathfrak{s}$.
But this contradicts Lemma~\ref{lemma:separate}, completing the proof.
\end{proof}

\subsection{The proof of Lemma~\ref{lemma:obtainment}}
\label{sec:lemma-obtainment}
\subsubsection{Preliminary reductions}
We are left with showing Lemma~\ref{lemma:obtainment}.  Recall that Corollary~\ref{cor:goursat}
gives us the existence  of  subgroups $H_{1},H_{2} \subset \sym(E^{(2)})$, normal subgroups $N_{1} \unlhd H_{1}, N_{2} \unlhd H_{2}$, and an isomorphism $\Psi \colon H_{1} / N_{1} \to H_{2} / N_{2}$ such that
$$\mathcal{K} = \{(\phi_{1},\phi_{2}) \in H_{1} \times H_{2} \colon \Psi([\phi_{1}]) = [\phi_{2}]\}.$$
The statement of Lemma~\ref{lemma:obtainment}
is that when both subgroups $N_1$ and $N_2$ are trivial,
at least one of the following holds:
\begin{enumerate}
\item
$H_{1} = \sym(E^{(2)})$ and $H_{2} = \sym(E^{(2)})$.
\item
$H_{1} = \alt(E^{(2)})$ and $H_{2} = \alt(E^{(2)})$.
\end{enumerate}
We start with some terminology used to study these subgroups.

For a finite set $X$, recall  that $\sym(X)$ denotes the group of permutations of the set $X$.
If $K \subset \sym(X)$ is a subgroup, a nonempty subset $A \subset X$ is called a $K$-\emph{block} if for all $g \in K$ either $g(A) = A$ or $g(A) \cap A = \emptyset$.
A subgroup $K \subset \sym(X)$ is called \emph{primitive} if the only $K$-blocks are singletons and $X$. We say the subgroup $K \subset \sym(X)$ \emph{contains a $p$-cycle} if it contains some element $\tau \in K$ such that $\tau$ consists of a single $p$-cycle.

\begin{theorem}[Jordan ({see~\cite[Theorem 13.9]{Wielandt}})]
Suppose $K \subset \sym(X)$ is primitive and contains a $p$-cycle for some prime $p < |X|-2$. Then $K = \alt(X)$ or $K = \sym(X)$.
\end{theorem}

Thus  to prove Lemma~\ref{lemma:obtainment}, by Jordan's Theorem,
since $H_{1}, H_{2} \subset \sym(E^{(2)})$, it suffices to show that at least one of $H_{1}, H_{2}$ is primitive and also contains a $p$-cycle for some prime $p<|E^{(2)}|-2$.

We start with some technical results on subgroups of  $\sym(E^{(2)})$, then prove primitivity, and
then show how to generate a $p$-cycle for some prime $p < |E^{(2)}|-2$.

\subsubsection{Subgroups of  $\sym(E^{(2)})$}
To denote the first and second components of an element $\etwo{x}{y} \in E^{(2)}$, we write
$$\etwo{x}{y}_{1} = x, \quad \etwo{x}{y}_{2} = y.$$
We say that an element $\tau \in \sym(E^{(2)})$ is
\begin{enumerate}
\item
\emph{row-preserving} if $\tau \etwo{x_{1}}{y_{1}}_{1} = \tau \etwo{x_{1}}{y_{2}}_{1}$ for all $y_{1},y_{2} \in E^{(1)}$.
\item
\emph{column-preserving} if $\tau \etwo{x_{1}}{y_{1}}_{2} = \tau \etwo{x_{2}}{y_{1}}_{2}$ for all $x_{1}, x_{2} \in E^{(1)}$.
\item
\emph{free} if $\tau$ is neither row-preserving nor column-preserving.
\end{enumerate}

For any element $\tau \in \sym(E^{(2)})$, there exists a pair of functions $\tau_{1}, \tau_{2} \colon E^{(2)} \to E^{(1)}$ such that
$$\tau \etwo{x_{1}}{y_{1}}= \begin{pmatrix} \tau_{1}\etwo{x_{1}}{y_{1}} \\ \tau_{2}\etwo{x_{1}}{y_{1}} \end{pmatrix}.$$
It follows quickly from the definitions that:
\begin{enumerate}
\item
$\tau$ is row-preserving if and only if $\tau_{1}\etwo{x}{y}$ is independent of $y$,
\item
$\tau$ is column-preserving if and only if $\tau_{2}\etwo{x}{y}$ is independent of $x$.
\end{enumerate}
It is also easy to check that:
\begin{enumerate}
\item
The collection of  $\tau \in \sym(E^{(2)})$ that are row-preserving
forms a subgroup.
\item
The collection of  $\tau \in \sym(E^{(2)})$ that are column-preserving
forms a subgroup.
\item
Any $\tau \in P$, where $P$ is the subgroup defined in~\eqref{def:P}, is both row-preserving and column-preserving.
\end{enumerate}

In Lemma~\ref{eqn:setupconditions}, we showed the existence of $\alpha\in N$
of the form $\alpha = \psi_{1}\sigma\psi_{2}\sigma^{-1}$  for some $\psi_{1},\psi_{2} \in \simp(\Gamma^{(2)})$.
The automorphism $\alpha$ acts on $P_{2}(\sigma)$, and upon identifying $P_{2}(\sigma)$ with $E^{(2)}$, there is a corresponding permutation of $E^{(2)}$ induced by $\alpha$, which we denote by $\overline{\alpha} \in \sym(E^{(2)})$.
(Recall that we are identifying $E^{(2)}$ with $E^{(1)} \times E^{(1)}$, and that $E^{(1)} = \{1,2,\ldots,n\}$.)

Recall that $\gamma_1, \gamma_2$ are defined in~\eqref{eq:gammas} and the swap map $\mathfrak{s}$
is defined in~\eqref{def:swap}.
By Part~\eqref{item:alpha3}  of Lemma~\ref{eqn:setupconditions},
we have that $\overline{\alpha}(\etwo{1}{1}) = \etwo{1}{1}$.
Since the subgroup $P$ (see~\eqref{def:P})
acts transitively on $E^{(2)}$, there exists some $\phi \in P$ such that $\gamma_{1}\phi \etwo{1}{1} = \etwo{1}{1}$. Letting $\tilde{\phi}$ denote the automorphism in $\simp(\Gamma^{(2)})$ corresponding to $\phi \in \sym(E^{(2)})$, we have that
$$\alpha = \psi_{1}\sigma\psi_{2}\sigma^{-1} = \psi_{1}\tilde{\phi} \tilde{\phi}^{-1} \sigma \psi_{2} \sigma^{-1} = \psi_{1} \tilde{\phi} \sigma \sigma^{-1}\tilde{\phi}^{-1}\sigma \psi_{2} \sigma^{-1} \in N.$$
Since $\phi \in P$, it is straightforward to check that $\sigma^{-1} \tilde{\phi}^{-1}\sigma \in \simp(\Gamma^{(2)})$,
and hence $(\psi_{1}\tilde{\phi},\psi_{2}^{-1}\sigma^{-1}\tilde{\phi}\sigma) \in \mathcal{K}_{N}$.
Furthermore (recall that the isomorphism $\mathcal{H}$ is defined in~\eqref{def:H}),
$$\mathcal{H}(\sigma^{-1}\tilde{\phi}\sigma) = \mathfrak{s}^{-1}\phi\mathfrak{s},$$
and it follows that $(\gamma_{1} \phi,\gamma_{2}^{-1}\mathfrak{s}^{-1}\phi \mathfrak{s}) \in \mathcal{K}$.
Abusing notation, we replace $\gamma_{1}$ and $\gamma_{2}^{-1}$ by $\gamma_{1}\phi$ and $\gamma_{2}^{-1}\mathfrak{s}^{-1}\phi\mathfrak{s}$, respectively. Then $\gamma_{1} \etwo{1}{1} = \etwo{1}{1}$. Since $\overline{\alpha} \etwo{1}{1} = \etwo{1}{1}$ and $\sigma \etwo{1}{1} = \etwo{1}{1}$,
it follows that $\gamma_{2} \etwo{1}{1} = \etwo{1}{1}$ as well.

By Part~\ref{it:partb} of Lemma~\ref{lemma:innandonp}, $\overline{\alpha} \etwo{1}{u_{1}} = \etwo{u_{2}}{1}$ for some $u_{1} \ne 1 ,u_{2} \ne 1$. Since $\alpha \in \aut(\sigma)$, it follows that $\overline{\alpha} \etwo{u_{1}}{1} = \etwo{1}{u_{2}}$ as well. Finally, recall in our notation the action $\overline{\alpha}$ of $\alpha$ on $E^{(2)}$ is given by
$$\overline{\alpha} = \gamma_{1}\mathfrak{s}\gamma_{2}\mathfrak{s}^{-1}.$$
\begin{lemma}\label{lemma:freeness}
Either $\gamma_{1}$ is free or $\gamma_{2}$ is free.
\end{lemma}
\begin{proof}
Suppose $\gamma_{2}$ is row-preserving. Then $\gamma_{2} \etwo{1}{u_{1}} = \etwo{1}{v_{1}}$ for some $v_{1} \in E^{(1)}$, $v_{1} \ne 1$, since $\gamma_{2}$ fixes $\etwo{1}{1}$. Then:
$$\etwo{1}{u_{2}} = \overline{\alpha} \etwo{u_{1}}{1} = \gamma_{1} \mathfrak{s} \gamma_{2} \mathfrak{s}^{-1} \etwo{u_{1}}{1}= \gamma_{1} \mathfrak{s} \gamma_{2} \etwo{1}{u_{1}} = \gamma_{1} \mathfrak{s} \etwo{1}{v_{1}}= \gamma_{1} \etwo{v_{1}}{1}.$$
Thus $\gamma_{1} \etwo{v_{1}}{1} = \etwo{1}{u_{2}}$. Since $\gamma_{1}$ fixes $\etwo{1}{1}$, it follows that $\gamma_{1}$ is free.

Suppose instead that $\gamma_{2}$ is column-preserving. Then, likewise, we have $\gamma_{2} \etwo{u_{1}}{1} = \etwo{v_{2}}{1}$ for some $v_{2} \in E^{(1)}$, $v_{2} \ne 1$, since $\gamma_{2}$ fixes $\etwo{1}{1}$.
Thus, as in the first case, we then have:
$$\etwo{u_{2}}{1} = \overline{\alpha} \etwo{1}{u_{1}} = \gamma_{1} \mathfrak{s} \gamma_{2} \mathfrak{s}^{-1} \etwo{1}{u_{1}} = \gamma_{1} \mathfrak{s} \gamma_{2} \etwo{u_{1}}{1} = \gamma_{1} \mathfrak{s} \etwo{v_{2}}{1}= \gamma_{1} \etwo{1}{v_{2}}.$$
Since $\gamma_{1}$ fixes $\etwo{1}{1}$, it again follows that $\gamma_{1}$ is free.
\end{proof}

For a subgroup $H \subset \sym(E^{(2)})$, we say $H$ \emph{contains the arrangement}
\begin{equation}\label{eqn:arrangementdef}
\begin{cases}
      \etwo{x_{1}}{y_{1}} \mapsto \etwo{x^{\prime}_{1}}{y^{\prime}_{1}} \\
      \vdots \\
      \etwo{x_{n}}{y_{n}} \mapsto \etwo{x^{\prime}_{n}}{y^{\prime}_{n}}
   \end{cases}
\end{equation}
if $H$ contains an element $\phi$ such that $\phi$ maps points as in~\eqref{eqn:arrangementdef}.
Note that not all points of $E^{(2)}$ may be listed, and
 if a point is not listed it means we make no claim how $\phi$ acts on that point.
 Instead of writing $\etwo{x_{1}}{y_{1}} \mapsto \etwo{x_{1}}{y_{1}}$,
 we  simply write $\etwoid{x_{1}}{y_{1}}$.

\begin{lemma}\label{lemma:diagmove}
Suppose $H$ is a subgroup of $\sym(E^{(2)})$ and $P \subset H$, where $P$ is the subgroup defined in~\eqref{def:P}.
\begin{enumerate}
\item
\label{it:row-tau}
Suppose there exists $\tau \in H$ such that $\tau$ is not row-preserving. Then at least one of the following holds:
\begin{enumerate}
\item
$H$ contains the arrangement
\begin{equation}\label{eqn:arr1}
\begin{cases}
      \etwoid{1}{1}\\
      \etwo{1}{2} \mapsto \etwo{2}{2}
   \end{cases}
\end{equation}
\item
$H$ contains the arrangement
\begin{equation}\label{eqn:arr2}
\begin{cases}
    \etwoid{1}{1}\\
    \etwo{2}{1} \mapsto \etwo{1}{2}\
\end{cases}
\end{equation}
\end{enumerate}
\item
\label{it:column-tau}
Suppose there exists $\tau \in H$ such that $\tau$ is not column-preserving. Then at least one of the following holds:
\begin{enumerate}
\item
$H$ contains the arrangement
\begin{equation}\label{eqn:arr1}
\begin{cases}
      \etwoid{1}{1}\\
      \etwo{2}{1} \mapsto \etwo{2}{2}
   \end{cases}
\end{equation}
\item
$H$ contains the arrangement
\begin{equation}\label{eqn:arr2s}
\begin{cases}
    \etwoid{1}{1}\\
    \etwo{2}{1} \mapsto \etwo{1}{2}
\end{cases}
\end{equation}
\end{enumerate}
\item If $H$ contains some $\tau$ where $\tau$ is free, then $H$ contains the arrangement
\label{it:free-tau}
\begin{equation}\label{eqn:arr2}
\begin{cases}
    \etwoid{1}{1}\\
    \etwo{1}{2} \mapsto \etwo{2}{1}
\end{cases}
\end{equation}
\end{enumerate}
\end{lemma}
\begin{proof}
The proofs of Parts~\eqref{it:row-tau} and~\eqref{it:column-tau} are similar, so we only prove case~\eqref{it:column-tau}, assuming that $\tau$ is not column-preserving.

Since $\tau$ is not column-preserving, there exist
$a_{1},a_{2},b_{1} \in E^{(1)}$ such that $\left(\tau\etwo{a_{1}}{b_{1}}\right)_{2} \ne \left(\tau\etwo{a_{2}}{b_{1}}\right)_{2}$.
The group $P$ acts transitively on $E^{(2)}$, so there exists $\phi_{1} \in P$ such that $\phi_{1}\tau \etwo{a_{1}}{b_{1}} = \etwo{a_{1}}{b_{1}}$.
It follows that $\phi_{1}\tau\etwo{a_{2}}{b_{1}}_{2} \ne b_{1}$.
Choose $\phi_{2} \in P$ such that $\phi_{2}\etwo{1}{1} = \etwo{a_{1}}{b_{1}}$, let $\phi_{3} = \phi_{2}^{-1}\phi_{1}\tau\phi_{2}$, and let $\etwo{a_{3}}{1} = \phi_{2}^{-1}\etwo{a_{2}}{b_{1}}$. Note that $a_{3} \ne 1$. We have $\phi_{3}\etwo{1}{1} = \etwo{1}{1}$, and setting $k = \phi_{3}\etwo{a_{3}}{1}_{2}$, we have $k \ne 1$ (since $\phi_{1}\tau\etwo{a_{2}}{b_{1}}_{2} \ne b_{1}$).
Letting $\phi_{4} = \etwo{\id}{k \leftrightarrow 2}\phi_{3}$,
it follows that $\phi_{4}\etwo{a_{3}}{1}_{2} = 2$. Finally, let $\phi_{5} = \phi_{4}\etwo{2 \leftrightarrow a_{3}}{\id}$, so that $\phi_{5}\etwo{2}{1} = \etwo{t}{2}$ for some $t$. Note that we still have $\phi_{5}\etwo{1}{1} = \etwo{1}{1}$. If $t=1$, then $\phi_{5}$ gives arrangement~\eqref{eqn:arr2s}. If $t>1$ then letting $\phi_{6} = \etwo{t \leftrightarrow 2}{\id}\phi_{5}$, $\phi_{6}$ gives arrangement~\eqref{eqn:arr1}.

Turning to Part~\eqref{it:free-tau}, suppose $\tau \in H$ and $\tau$ is free. By  Parts~\eqref{it:row-tau} and~\eqref{it:column-tau}, either $H$ contains the arrangement
\begin{equation}\label{eqn:arr2}
\begin{cases}
    \etwoid{1}{1}\\
    \etwo{2}{1} \mapsto \etwo{1}{2}
\end{cases}
\end{equation}
in which case (upon taking an inverse) we are done, or $H$ contains both arrangements
\begin{equation}\label{eqn:arr2}
\begin{cases}
    \etwoid{1}{1}\\
    \etwo{2}{1} \mapsto \etwo{2}{2}
\end{cases}
\text{ and } \quad
\begin{cases}
    \etwoid{1}{1}\\
    \etwo{1}{2} \mapsto \etwo{2}{2}.
\end{cases}
\end{equation}
In the latter case, if $\phi_{1}, \phi_{2}$ implement these arrangements, then $\phi_{1}^{-1}\phi_{2}$ implements the arrangement
\begin{equation}\label{eqn:arr2}
\begin{cases}
    \etwoid{1}{1}\\
    \etwo{1}{2} \mapsto \etwo{2}{1}
    .
\end{cases}
\end{equation}

\end{proof}

\subsubsection{Structures in the subgroups $H_1, H_2$}
We  use pictures to depict the action of elements of $\sym(E^{(2)})$. Since $E^{(2)} = E^{(1)} \times E^{(1)}$, we consider $E^{(2)}$ as a grid of points. When we say $\phi \in \sym(E^{(2)})$ \emph{acts by}
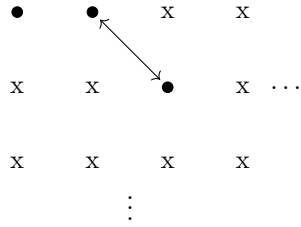
\begin{center}
\begin{tikzpicture}
  \filldraw [black] (0,3) circle (2pt)
                   (1,3) circle (2pt)
                   (2,3) node{x}
                   (3,3) node{x}
                                      (0,1) node{x}
                                                   (0,2) node{x}
                                                         (1,1) node{x}
                                                                      (1,2) node{x}
                   (2,2) circle (2pt)
                                        (3,2) node{x}
                                                             (2,1) node{x}
                  (3,1) node{x}
               (1.5,.5) node{\vdots}
               (3.6,2) node{\dots};
 \draw[<->]      (1.1,2.9)   -- (1.9,2.1);

\end{tikzpicture}
\end{center}
we mean that $\phi$ acts on $E^{(2)}$ as drawn in the picture, with the following conventions:
\begin{enumerate}
\item
A dot associated with no arrow represents a point fixed by $\phi$.
\item
An \textbf{x}
means the point could be mapped anywhere, i.e. we make no assumption on how that point is mapped by $\phi$.
\item
Ellipses indicate the type of action continues in that direction, and so the use of ellipses following $\textbf{x}$'s means that we make no assumption
on how $\phi$ acts on points in that direction.
\item
When no ellipses are present, $\phi$ acts by the identity on any unrepresented points (i.e. points in $E^{(2)}$ which do not appear in the picture).
\end{enumerate}

\begin{definition}\label{def:substantialsubgroup}
We say a subgroup $H \subset \sym(E^{(2)})$ is \emph{substantial} if it contains both of the following:
\begin{enumerate}
\item
\label{item:startingone}
A free element.
\item
\label{item:startingtwo}
An involution implementing at least one of the following arrangements:
\vskip.2in
\begin{minipage}{0.4\textwidth}
\begin{center}
\begin{tikzpicture}
  \filldraw [black] (0,3) circle (2pt)
  (0,2) circle (2pt)
         (1,1) circle (2pt)
                   (1,3) node{x}
                   (2,3) node{x}
                   (2,0) node{x}
                                      (0,1) node{x}
(1,2) node{x}

                                                                       (2,2) node{x}
                                        (1,0) node{x}
                                                             (2,1) node{x}
                  (0,0) node{x}
               (1.5,-.5) node{\vdots}
               (2.6,1.5) node{\dots};
 \draw[<->]      (0.1,1.9)   -- (0.9,1.1);
\end{tikzpicture}
\end{center}
\end{minipage}
\begin{minipage}{0.4\textwidth}
\begin{center}
\begin{tikzpicture}
  \filldraw [black] (0,3) circle (2pt)
                   (1,3) circle (2pt)
                   (2,3) node{x}
                   (3,3) node{x}
                                      (0,1) node{x}
                                                   (0,2) node{x}
                                                         (1,1) node{x}
                                                                      (1,2) node{x}
                   (2,2) circle (2pt)
                                        (3,2) node{x}
                                                             (2,1) node{x}
                  (3,1) node{x}
               (1.5,.5) node{\vdots}
                 (3.6,2) node{\dots};
 \draw[<->]      (1.1,2.9)   -- (1.9,2.1);
\end{tikzpicture}
\end{center}
\end{minipage}

\begin{minipage}{0.4\textwidth}
\begin{center}
\captionsetup{font=footnotesize}
\captionof*{figure}{Arrangement (a)}
\end{center}
\end{minipage}
\begin{minipage}{0.4\textwidth}
\begin{center}
\captionsetup{font=footnotesize}
\captionof*{figure}{Arrangement (b)}
\end{center}
\end{minipage}
\end{enumerate}
\end{definition}

\begin{lemma}\label{lemma:startingarrangement}
At least one of the subgroups $H_{1}, H_{2}$ in $\sym(E^{(2)})$ is substantial.
\end{lemma}

Before the proof, we introduce some notation. Define
\begin{equation}
\label{def:CR}
{\rm CR} = \left\{\etwo{x}{y} \in E^{(2)} \colon \textnormal{ either } x=1 \textnormal{ or } y=1\right\}
\end{equation}
(thus ${\rm CR}$ is the union of row one and column one in $E^{(2)}$), and define
\begin{equation}
\label{def:IS}
{\rm IS} = E^{(2)} \setminus {\rm CR}.
\end{equation}

\begin{proof}
First suppose $H$ is a subgroup of $\sym(E^{(2)})$ with $P \subset H$ and suppose $\phi \in H$ satisfies both of the following:
\begin{enumerate}
\item
\label{claim:one}
$\phi \etwo{1}{1} = \etwo{1}{1}$.
\item
\label{claim:two}
$\phi \etwo{x_{1}}{y_{1}} \in {\rm IS}$ for some $\etwo{x_{1}}{y_{1}} \in {\rm CR}$, $\etwo{x_{1}}{y_{1}} \ne \etwo{1}{1}$.
\end{enumerate}
We prove that $H$ contains an involution implementing at least one of the arrangements in Part~\eqref{item:startingtwo} of Definition~\ref{def:substantialsubgroup}. Thus suppose that we have such a $\phi$ and some $\etwo{1}{1} \ne \etwo{x_{1}}{y_{1}} \in {\rm CR}$ with $\phi \etwo{x_{1}}{y_{1}} = \etwo{x_{2}}{y_{2}} \in {\rm IS}$. Suppose first that $\etwo{x_{1}}{y_{1}}$ is in column one (so $y_{1} = 1$ and $x_{1} > 1$, since $\etwo{x_{1}}{y_{1}} \ne \etwo{1}{1}$). By replacing $\phi$ by
$\etwo{2 \leftrightarrow x_{1}}{id}^{-1} \phi \etwo{2 \leftrightarrow x_{1}}{id}$, we may assume
that $x_{1} = 2$. Set
$$L_{1} = \{\etwo{x}{y} \in E^{(2)} \colon x > 2 \textnormal{ and } y > 1\}.
$$ Since $n$ is large, there exists some
$\etwo{x_{3}}{y_{3}} \in L_{1}$ such that $\phi \etwo{x_{3}}{y_{3}} \in {\rm IS}$. Choose some involution $\tau_{1} \in P$ such that $\tau_{1} \etwo{1}{1} = \etwo{1}{1}$ and $\tau_{1} \phi \etwo{x_{3}}{y_{3}} = \etwo{x_{2}}{y_{2}}$. Then
$$\phi^{-1} \tau_{1} \phi \etwo{x_{3}}{y_{3}} = \phi^{-1} \etwo{x_{2}}{y_{2}} = \etwo{x_{1}}{y_{1}} = \etwo{2}{1}.$$
Thus $\phi^{-1} \tau_{1} \phi$ is an involution in $H$ fixing $\etwo{1}{1}$ which satisfies $\phi^{-1} \tau_{1} \phi \etwo{2}{1} = \etwo{x_{3}}{y_{3}} \in L_{1}$, and we may choose
another involution $\tau_{2} \in P$ such that $\tau_{2}$ fixes $\etwo{1}{1}$, and $\tau_{2} \etwo{x_{3}}{y_{3}} = \etwo{3}{2}$. Now the involution $\tau_{2}^{-1} \phi^{-1} \tau_{1} \phi \tau_{2}$ is in $H$, and implements the first arrangement.
The case that $\etwo{x_{1}}{y_{1}}$ is in row one is similar, and produces an involution in $H$ implementing the second arrangement.  This completes the proof of the claim.

Recall we have $\gamma_{1} \in H_{1}, \gamma_{2}^{-1} \in H_{2}$ (see~\eqref{eq:gammas}) and both $\gamma_{1}$ and $\gamma_{2}$ fix $\etwo{1}{1}$.
By Lemmma~\ref{lemma:freeness}, either $\gamma_{1}$ is free or $\gamma_{2}$ is free.
Suppose then that $\gamma_{1}$ is free.
If $\gamma_{1}$ maps any point (necessarily not $\etwo{1}{1}$) in ${\rm CR}$ into ${\rm IS}$, then $H_{1}$ satisfies both parts of Definition~\ref{def:substantialsubgroup} by the claim above. Suppose then that $\gamma_{1}$ leaves ${\rm CR}$ invariant. Then $\gamma_{1} \mathfrak{s}$ leaves ${\rm CR}$ invariant, and fixes $\etwo{1}{1}$. By condition~\eqref{item:alpha5} of Lemma~\ref{eqn:setupconditions}, $\overline{\alpha} = \gamma_{1} \mathfrak{s} \gamma_{2} \mathfrak{s}^{-1}$ maps the points $\etwo{1}{u_{3}}$ and $\etwo{u_{3}}{1}$ into ${\rm IS}$. Since $\mathfrak{s}$ leaves ${\rm CR}$ and hence ${\rm IS}$ invariant, this means $\gamma_{2}$ maps both $\etwo{1}{u_{3}}$ and $\etwo{u_{3}}{1}$ into ${\rm IS}$.
Since $\gamma_{2}$ fixes $\etwo{1}{1}$, this implies $\gamma_{2}$ is neither row-preserving nor column-preserving,
and so is free.  Furthermore, $\gamma_{2}$ maps a point in ${\rm CR}$ (specifically, $\etwo{1}{u_{3}}$) into ${\rm IS}$. By the claim,
 this implies $H_{2}$ satisfies both conditions~\eqref{item:startingone} and~\eqref{item:startingtwo} of Definition~\ref{def:substantialsubgroup}.

A similar argument shows that if $\gamma_{2}$ is free and preserves ${\rm CR}$, then $H_{1}$ satisfies both conditions~\eqref{item:startingone} and~\eqref{item:startingtwo} of Definition~\ref{def:substantialsubgroup}, finishing the proof.
\end{proof}

\subsubsection{Primitivity}
Our goal now is to show that any substantial subgroup of $\sym(E^{(2)})$ which contains $P$ is primitive.

We  make use of the following lemma from~\cite{GenSymGroups}.
\begin{lemma}[See~{\cite[page 735]{GenSymGroups}}]
\label{lemma:primitive-group}
Suppose $X$ is a finite set, $K \subset \sym(X)$ is transitive, and $x \in X$. Then $K$ is primitive if the only blocks which contain $x$ are $\{x\}$ and $X$.
\end{lemma}
\begin{lemma}
\label{lemma:primitive}
Suppose $H \subset \sym(E^{(2)})$ is a subgroup which contains $P$ and is substantial. Then $H$ is primitive.
\end{lemma}
\begin{proof}
Since the subgroup $P$ (see~\eqref{def:P}) acts transitively on $E^{(2)}$ and $P \subset H$, the subgroup
$H$ also acts transitively on $E^{(2)}$.
By Lemma~\ref{lemma:primitive-group}, it suffices to show that if $A$ is any $H$-block containing $\etwo{1}{1}$ and at least one other point, then $A$ must be all of $E^{(2)}$.

Let $A$ be an $H$-block containing $\etwo{1}{1}$ and some other point $\etwo{x_{1}}{y_{1}}$.
We claim that if $A$ contains a point in ${\rm IS}$ (recall that the set ${\rm IS}$ is defined in~\eqref{def:IS}), then $A = E^{(2)}$.
To check this, suppose $A$ contains $\etwo{u_{1}}{v_{1}} \in {\rm IS}$. If $\etwo{u_{2}}{v_{2}}$ is any other point in ${\rm IS}$, then there exists $\phi \in P$ such that
\begin{equation*}
\phi \colon
\begin{cases}
\etwoid{1}{1}\\
\etwo{u_{1}}{v_{1}} \mapsto \etwo{u_{2}}{v_{2}}
\end{cases}
\end{equation*}
It follows that $A$ contains ${\rm IS}$. Now $\etwo{\id}{1 \leftrightarrow 2}A \cap A \ne \emptyset$ and $\etwo{\id}{1 \leftrightarrow 2}A$ contains all of column $1$ except $\etwo{1}{1}$, so $A$ contains all of column $1$ (since $A$ already contained $\etwo{1}{1}$). Likewise, $\etwo{1 \leftrightarrow 2}{\id}A \cap A \ne \emptyset$ so $A$ must contain all of row $1$. Thus $A$ must contain all of $E^{(2)}$, proving the claim.

To finish the proof of the lemma, it suffices then to show that $A$ contains some point in
${\rm IS}$.
By assumption, $A$ contains some point $\etwo{x_{1}}{y_{1}} \ne \etwo{1}{1}$. The only remaining cases then are that either $\etwo{x_{1}}{y_{1}}$ lies in row 1 or $\etwo{x_{1}}{y_{1}}$ lies in column 1. We prove the first case; the second case is analogous.

Assume $x_{1}=1$. Then for any $1 \ne z \in E^{(1)}$, $\etwo{\id}{z \leftrightarrow y_{1}}A \cap A$ contains $\etwo{1}{1}$, so $A$ contains $\etwo{1}{z}$ for all such $z$, and $A$ contains row 1. Let $\rho \in \sym(E^{(1)})$ denote the 3-cycle mapping $3 \mapsto 2, 2 \mapsto 1, 1 \mapsto 3$. Since $H$ is substantial, it contains a free element. Thus by part $(iii)$ of Lemma~\ref{lemma:diagmove}
there is some $\tilde{\gamma} \in H$ such that
\begin{equation*}
\tilde{\gamma} \colon
\begin{cases}
\etwoid{1}{1}\\
\etwo{1}{2} \mapsto \etwo{2}{1}.
\end{cases}
\end{equation*}
Then $\etwo{\id}{\rho^{-1}}\tilde{\gamma}\etwo{\id}{\rho} \in H$ and
\begin{equation*}
\etwo{\id}{\rho^{-1}}\tilde{\gamma}\etwo{\id}{\rho} \colon
\begin{cases}
\etwoid{1}{2}\\
\etwo{1}{3} \mapsto \etwo{2}{2}.
\end{cases}
\end{equation*}
Since $A$ contains row 1, it contains $\etwo{1}{2}$, so this implies that $A$ contains $\etwo{2}{2}$, completing the proof.
\end{proof}

\subsubsection{Obtaining a $p$-cycle}
The main goal of this subsection is to prove the following lemma.
\begin{lemma}\label{lemma:substantialpcycle}
Let $H \subset \sym(E^{(2)})$ be a subgroup which contains $P$ and is substantial. Then $H$ contains a $p$-cycle for some prime $p < |E^{(2)}|-2$.

\end{lemma}

We start with some notation to aid in describing the arrangements.
Define
\begin{equation}
\label{def:Rij}
R_{i.j} =  \Bigl\{\etwo{i}{y}\colon y \in E^{(1)}(\Gamma)\Bigr\} \cup  \Bigl\{\etwo{j}{y}\colon y \in E^{(1)}(\Gamma)\Bigr\},
\end{equation}
and
\begin{equation}
\label{def:Cij}
C_{i,j} = \Bigl\{\etwo{x}{i}\colon x \in E^{(1)}(\Gamma)\Bigr\} \cup
\Bigl\{\etwo{x}{j}\colon x \in E^{(1)}(\Gamma)\Bigr\}.
\end{equation}
Thus $R_{i,j}$ denotes the set of points in $E^{(2)}$ which belong to either row $i$ or $j$, and $C_{i,j}$ denotes the set of points in $E^{(2)}$ which belong to either column $i$ or $j$.

Given $1 \le i,j \le n$,
let $\phi^{C}_{i,j}$ denote the involution in $P$ swapping columns $i$ and $j$ and let $\phi^{R}_{i,j}$ denote the involution in $P$ swapping rows $i$ and $j$. Given any $\phi_{1}, \phi_{2} \in H_{1}$ we let $\phi_{2}^{\phi_{1}} = \phi_{1}^{-1}\phi_{2}\phi_{1}$, and for $\tau,\phi \in H_{1}$, define
$$\tau \star \phi = \left( \tau^{\phi} \right)^{-1} \tau = \phi^{-1} \tau^{-1} \phi \tau.$$

(While $\tau \star \phi$ is usually denoted by $[\phi,\tau]$, we find the $\star$ notation to be more readable.)

We  frequently use the following observation: if $c$ is a cycle whose support does not intersect $C_{i,j}$ (respectively, $R_{i,j}$), then $c \star \phi^{C}_{i,j} = \id$ ($c \star \phi^{R}_{i,j} = \id$, respectively).

Let us briefly outline the proof of Lemma~\ref{lemma:substantialpcycle}. Suppose $H$ is a substantial subgroup of $\sym(E^{(2)})$ which contains $P$. To show $H$ contains a $p$-cycle, we begin by letting $\gamma_{3}$ denote some element of $H$ which acts by one of the arrangements in Definition~\ref{def:substantialsubgroup}; say Arrangement $(a)$. Letting $\gamma_{4} = \gamma_{3} \star \phi^{R}_{1,2}$, by passing from $\gamma_{3}$ to this $\gamma_{4}$, any $2$-cycles in $\gamma_{3}$ whose support were disjoint from rows one and two vanish. Moreover, the element $\gamma_{4}$ has a distinguished $3$-cycle whose support consists of the points $\etwo{1}{1}, \etwo{2}{1}, \etwo{3}{2}$, and we use this distinguished cycle to reduce to a collection of cases, which we then handle.  The proof of this occupies the remainder of this section.

\begin{lemma}\label{lemma:somecycles}
Suppose $H$ is a subgroup of $\sym(E^{(2)})$ which contains $P$ and any of the following arrangements:
\vskip.2in

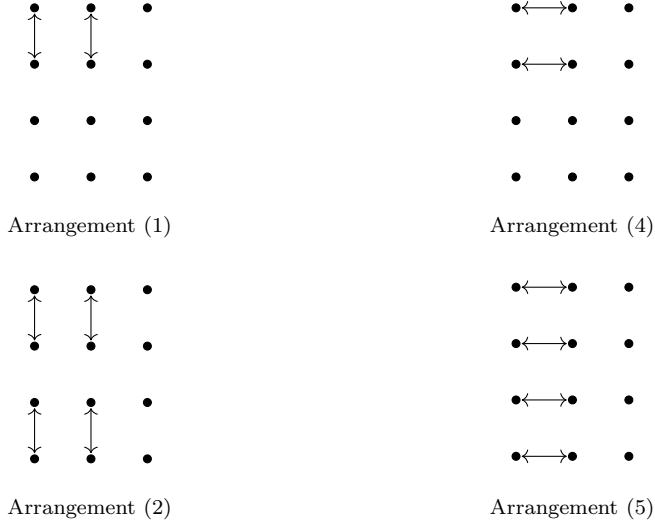
\begin{minipage}[b]{0.5\textwidth}
\centering
\begin{tikzpicture}[scale=0.75]
\filldraw [black] (0,3) circle (2pt)
                   (1,3) circle (2pt)
                   (2,3) circle (2pt)
                   (0,2) circle (2pt)
                   (1,2) circle (2pt)
                   (2,2) circle (2pt)
                   (0,1) circle (2pt)
                   (1,1) circle (2pt)
                   (2,1) circle (2pt)
                   (0,0) circle (2pt)
                   (1,0) circle (2pt)
                   (2,0) circle (2pt)
                  ;
      \draw[<->]      (0,2.9)   -- (0,2.1);
            \draw[<->]      (1,2.9)   -- (1,2.1);

\end{tikzpicture}
\captionsetup{font=footnotesize}
\captionof*{figure}{Arrangement (1)}
\end{minipage}
\begin{minipage}[b]{0.5\textwidth}
\centering
\begin{tikzpicture}[scale=0.75]
  \filldraw [black] (0,3) circle (2pt)
                   (1,3) circle (2pt)
                   (2,3) circle (2pt)
                   (0,2) circle (2pt)
                   (1,2) circle (2pt)
                   (2,2) circle (2pt)
                   (0,1) circle (2pt)
                   (1,1) circle (2pt)
                   (2,1) circle (2pt)
                   (0,0) circle (2pt)
                   (1,0) circle (2pt)
                   (2,0) circle (2pt)
                  ;
      \draw[<->]      (0.1,3)   -- (0.9,3);
            \draw[<->]      (0.1,2)   -- (0.9,2);
\end{tikzpicture}
\captionsetup{font=footnotesize}
\captionof*{figure}{Arrangement (4)}
\end{minipage}

\vskip.2in
\begin{minipage}[b]{0.5\textwidth}
\centering
\begin{tikzpicture}[scale=0.75]
  \filldraw [black] (0,3) circle (2pt)
                   (1,3) circle (2pt)
                   (2,3) circle (2pt)
                   (0,2) circle (2pt)
                   (1,2) circle (2pt)
                   (2,2) circle (2pt)
                   (0,1) circle (2pt)
                   (1,1) circle (2pt)
                   (2,1) circle (2pt)
                   (0,0) circle (2pt)
                   (1,0) circle (2pt)
                   (2,0) circle (2pt)
                  ;
       \draw[<->]      (0,2.9)   -- (0,2.1);
            \draw[<->]      (1,2.9)   -- (1,2.1);

         \draw[<->]      (0,0.9)   -- (0,0.1);
            \draw[<->]      (1,0.9)   -- (1,0.1);

\end{tikzpicture}
\captionsetup{font=footnotesize}
\captionof*{figure}{Arrangement (2)}
\end{minipage}
\begin{minipage}[b]{0.5\textwidth}
\centering
\begin{tikzpicture}[scale=0.75]
  \filldraw [black] (0,3) circle (2pt)
                   (1,3) circle (2pt)
                   (2,3) circle (2pt)
                   (0,2) circle (2pt)
                   (1,2) circle (2pt)
                   (2,2) circle (2pt)
                   (0,1) circle (2pt)
                   (1,1) circle (2pt)
                   (2,1) circle (2pt)
                   (0,0) circle (2pt)
                   (1,0) circle (2pt)
                   (2,0) circle (2pt)
                  ;
       \draw[<->]      (0.1,3)   -- (0.9,3);
            \draw[<->]      (0.1,2)   -- (0.9,2);

         \draw[<->]      (0.1,1)   -- (0.9,1);
            \draw[<->]      (0.1,0)   -- (0.9,0);

\end{tikzpicture}
\captionsetup{font=footnotesize}
\captionof*{figure}{Arrangement (5)}
\end{minipage}

\vskip.2in
\begin{minipage}[b]{0.5\textwidth}
\centering
\begin{tikzpicture}[scale=0.75]
  \filldraw [black] (0,3) circle (2pt)
                   (1,3) circle (2pt)
                   (2,3) circle (2pt)
                   (0,2) circle (2pt)
                   (1,2) circle (2pt)
                   (2,2) circle (2pt)
                   (0,1) circle (2pt)
                   (1,1) circle (2pt)
                   (2,1) circle (2pt)
                   (0,0) circle (2pt)
                   (1,0) circle (2pt)
                   (2,0) circle (2pt)
                  ;
   \draw[->]      (0,2.9)   -- (0,2.1);
            \draw[->]      (1,2.9)   -- (1,2.1);

         \draw[->]      (0,1.9)   -- (0,1.1);
            \draw[->]      (1,1.9)   -- (1,1.1);
            \draw[->] (-.1,1.1) .. controls (-.5,2) .. (-.1, 2.9);
                        \draw[->] (.9,1.1) .. controls (.5,2) .. (.9, 2.9);
\end{tikzpicture}
\captionsetup{font=footnotesize}
\captionof*{figure}{Arrangement (3)}
\end{minipage}
\begin{minipage}[b]{0.5\textwidth}
\centering
\begin{tikzpicture}[scale=0.75]
  \filldraw [black] (0,3) circle (2pt)
                   (1,3) circle (2pt)
                   (2,3) circle (2pt)
                   (0,2) circle (2pt)
                   (1,2) circle (2pt)
                   (2,2) circle (2pt)
                   (0,1) circle (2pt)
                   (1,1) circle (2pt)
                   (2,1) circle (2pt)
                   (0,0) circle (2pt)
                   (1,0) circle (2pt)
                   (2,0) circle (2pt)
                  ;
             \draw[->]      (0.1,3)   -- (0.9,3);
            \draw[->]      (1.1,3)   -- (1.9,3);

             \draw[->]      (1.9,3.1)  .. controls (1.5,3.4) and (0.5,3.4)  .. (0.1,3.1);

            \draw[->]      (0.1,2)   -- (0.9,2);
            \draw[->]      (1.1,2)   -- (1.9,2);
            \draw[->]      (1.9,2.1)  .. controls (1.5,2.4) and (0.5,2.4)  .. (0.1,2.1);
\end{tikzpicture}
\captionsetup{font=footnotesize}
\captionof*{figure}{Arrangement (6)}
\end{minipage}
Then $H$ contains a $3$-cycle.
\end{lemma}

\begin{proof}
We prove the lemma for arrangements $(1), (2), (3)$; the proofs for arrangements $(4), (5), (6)$ are similar.

Suppose the arrangement $(1)$ is implemented by the involution $\gamma_{3}$. Then
$$\gamma_{4} = \left( \gamma_{3}^{\phi^{R}_{2,3}} \right) \gamma_{3}$$ acts by the arrangement
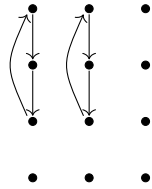
\begin{center}
\begin{tikzpicture}[scale=0.75]
  \filldraw [black] (0,3) circle (2pt)
                   (1,3) circle (2pt)
                   (2,3) circle (2pt)
                   (0,2) circle (2pt)
                   (1,2) circle (2pt)
                   (2,2) circle (2pt)
                   (0,1) circle (2pt)
                   (1,1) circle (2pt)
                   (2,1) circle (2pt)
                   (0,0) circle (2pt)
                   (1,0) circle (2pt)
                   (2,0) circle (2pt)
                  ;
   \draw[->]      (0,2.9)   -- (0,2.1);
            \draw[->]      (1,2.9)   -- (1,2.1);

         \draw[->]      (0,1.9)   -- (0,1.1);
            \draw[->]      (1,1.9)   -- (1,1.1);
            \draw[->] (-.1,1.1) .. controls (-.5,2) .. (-.1, 2.9);
                        \draw[->] (.9,1.1) .. controls (.5,2) .. (.9, 2.9);
\end{tikzpicture}
\end{center}
and $\gamma_{3} \gamma_{4}^{\phi^{C}_{1,3}}$ acts by the arrangement
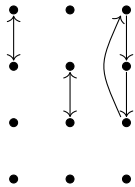
\begin{center}
\begin{tikzpicture}[scale=0.75]
  \filldraw [black] (0,3) circle (2pt)
                   (1,3) circle (2pt)
                   (2,3) circle (2pt)
                   (0,2) circle (2pt)
                   (1,2) circle (2pt)
                   (2,2) circle (2pt)
                   (0,1) circle (2pt)
                   (1,1) circle (2pt)
                   (2,1) circle (2pt)
                   (0,0) circle (2pt)
                   (1,0) circle (2pt)
                   (2,0) circle (2pt)
                  ;
   \draw[<->]      (0,2.9)   -- (0,2.1);

         \draw[<->]      (1,1.9)   -- (1,1.1);
          \draw[->]      (2,2.9)   -- (2,2.1);
           \draw[->]      (2,1.9)   -- (2,1.1);
                        \draw[->] (1.9,1.1) .. controls (1.5,2) .. (1.9, 2.9);
\end{tikzpicture}
\end{center}
Squaring now produces a $3$-cycle.

Suppose now the arrangement $(2)$ is implemented by some $\gamma_{3}$. Then $\gamma_{4} = \gamma_{3}^{\phi^{R}_{3,5}}$ acts by the arrangement
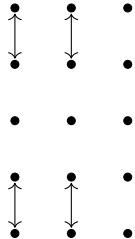
\begin{center}
\begin{tikzpicture}[scale=0.75]
  \filldraw [black] (0,3) circle (2pt)
                   (1,3) circle (2pt)
                   (2,3) circle (2pt)
                   (0,2) circle (2pt)
                   (1,2) circle (2pt)
                   (2,2) circle (2pt)
                   (0,1) circle (2pt)
                   (1,1) circle (2pt)
                   (2,1) circle (2pt)
                   (0,0) circle (2pt)
                   (1,0) circle (2pt)
                   (2,0) circle (2pt)
                   (0,-1) circle (2pt)
                   (1,-1) circle (2pt)
                   (2,-1) circle (2pt)
                  ;
      \draw[<->]      (0,2.9)   -- (0,2.1);
            \draw[<->]      (1,2.9)   -- (1,2.1);
             \draw[<->]      (0,-0.1)   -- (0,-0.9);
            \draw[<->]      (1,-0.1)   -- (1,-0.9);

\end{tikzpicture}
\end{center}
Setting $\gamma_{5} = \left( \gamma_{4}^{\phi^{R}_{2,3}} \right) \gamma_{4}$, the element $\left( \gamma_{4} \gamma_{5}^{\phi^{C}_{1,3}} \right)^{2}$ consists of a single $3$-cycle.

Suppose now the arrangement $(3)$ is implemented by some $\gamma_{3}$. Then $\gamma_{4} = \left(\gamma_{3}^{\phi^{R}_{3,4}}\right)\gamma_{3}$ acts by the arrangement
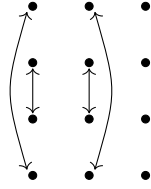
\begin{center}
\begin{tikzpicture}[scale=0.75]
  \filldraw [black] (0,3) circle (2pt)
                   (1,3) circle (2pt)
                   (2,3) circle (2pt)
                   (0,2) circle (2pt)
                   (1,2) circle (2pt)
                   (2,2) circle (2pt)
                   (0,1) circle (2pt)
                   (1,1) circle (2pt)
                   (2,1) circle (2pt)
                   (0,0) circle (2pt)
                   (1,0) circle (2pt)
                   (2,0) circle (2pt)
                  ;
                   \draw[<->] (-0.1,2.9) .. controls (-0.5,1.5) .. (-0.1,0.1);
                        \draw[<->] (1.1,2.9) .. controls (1.5,1.5) .. (1.1,0.1);
            \draw[<->]      (0,1.9)   -- (0,1.1);
            \draw[<->]      (1,1.9)   -- (1,1.1);

\end{tikzpicture}
\end{center}
and $\gamma_{5} = \gamma_{4}^{\phi^{R}_{2,4}}$ acts by the arrangement
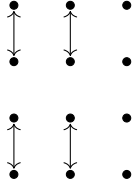
\begin{center}
\begin{tikzpicture}[scale=0.75]
  \filldraw [black] (0,3) circle (2pt)
                   (1,3) circle (2pt)
                   (2,3) circle (2pt)
                   (0,2) circle (2pt)
                   (1,2) circle (2pt)
                   (2,2) circle (2pt)
                   (0,1) circle (2pt)
                   (1,1) circle (2pt)
                   (2,1) circle (2pt)
                   (0,0) circle (2pt)
                   (1,0) circle (2pt)
                   (2,0) circle (2pt)
                  ;
       \draw[<->]      (0,2.9)   -- (0,2.1);
            \draw[<->]      (1,2.9)   -- (1,2.1);

         \draw[<->]      (0,0.9)   -- (0,0.1);
            \draw[<->]      (1,0.9)   -- (1,0.1);

\end{tikzpicture}
\end{center}
But this is exactly the arrangement in case $(2)$, so the result follows for the same reason.
\end{proof}

\begin{lemma}\label{lemma:r1234}
Suppose $H$ is a subgroup of $\sym(E^{(2)})$ which contains $P$, and suppose $H$ contains an involution $\tau_{1}$ which satisfies the following:
\begin{enumerate}
\item
$\tau_{1}$ is supported in rows $1,2,3,4$, and consists of an even number of $2$-cycles $d_{i}$, $i=1,\ldots,2q$ for some $q\geq 1$.
\item
Each $2$-cycle in $\tau_{1}$ has support containing a point in $R_{1,2}$ and a point in $R_{3,4}$.
\item
Each $2$-cycle in $\tau_{1}$ has a companion $2$-cycle, meaning that for each $2$-cycle $d_{i}$, we have $d_{i + q \textnormal{ mod } 2q} = d_{i}^{\phi^{R}_{1,2}\phi^{R}_{3,4}}$.
\item
$\tau_{1}$ has a pair of $2$-cycles $d_{1}, d_{q+1}$ such that $d_{1} = \left( \etwo{1}{1}, \etwo{4}{2} \right)$ and $d_{q+1} = \left( \etwo{2}{1}, \etwo{3}{2} \right)$.
\end{enumerate}
Then $H$ contains a $p$-cycle for some prime $p < |E^{(2)}|-2$.
\end{lemma}
\begin{proof}
We proceed by cases (recall that $C_{1,2}$ is defined in~\eqref{def:Cij}):
\subsubsection*{Case 1 } Suppose $\tau_{1}$ leaves $C_{1,2}$ invariant and acts nontrivially on $C_{1} \cap R_{3,4}$ (and hence, given the setup, also nontrivially on $C_{2} \cap R_{1,2}$).
Then one of the following two cases occurs:
\subsubsection*{Case 1a}
Suppose $\tau_{1}$ acts by the arrangement
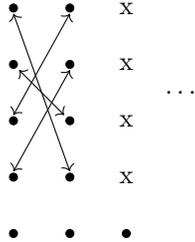
\begin{center}
\begin{tikzpicture}[scale=0.75]
  \filldraw [black] (0,3) circle (2pt)
                   (1,3) circle (2pt)
                   (2,3) node{x}
                   (0,2) circle (2pt)
                   (1,2) circle (2pt)
                   (2,2) node{x}
                   (0,1) circle (2pt)
                   (1,1) circle (2pt)
                   (2,1) node{x}
                   (0,0) circle (2pt)
                   (1,0) circle (2pt)
                   (2,0) node{x}
                   (0,-1) circle (2pt)
                   (1,-1) circle (2pt)
                   (2,-1) circle (2pt)
                   (3,1.5) node{\dots}
                  ;
   \draw[<->]      (0,2.9)   -- (1,0.1);
            \draw[<->]      (1,2.9)   -- (0,1.1);
            \draw[<->]      (0.1,1.9)   -- (0.9,1.1);
            \draw[<->]      (1,1.9)   -- (0,0.1);
\end{tikzpicture}
\end{center}
on $C_{1,2}$. Then $\tau_{1} \star \phi^{C}_{1,2}$
acts by arrangement $(2)$ of Lemma~\ref{lemma:somecycles},  and the result follows.

\subsubsection*{Case 1b:}
Suppose $\tau_{1}$ acts by the arrangement
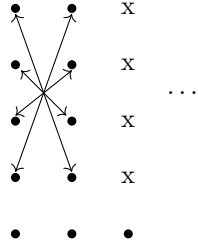
\begin{center}
\begin{tikzpicture}[scale=0.75]
  \filldraw [black] (0,3) circle (2pt)
                   (1,3) circle (2pt)
                   (2,3) node{x}
                   (0,2) circle (2pt)
                   (1,2) circle (2pt)
                   (2,2) node{x}
                   (0,1) circle (2pt)
                   (1,1) circle (2pt)
                   (2,1) node{x}
                   (0,0) circle (2pt)
                   (1,0) circle (2pt)
                   (2,0) node{x}
                   (0,-1) circle (2pt)
                   (1,-1) circle (2pt)
                   (2,-1) circle (2pt)
                   (3,1.5) node{\dots}
                  ;
   \draw[<->]      (0,2.9)   -- (1,0.1);
            \draw[<->]      (1,2.9)   -- (0,0.1);
            \draw[<->]      (0.1,1.9)   -- (0.9,1.1);
            \draw[<->]      (1,1.9)   -- (0,1.1);
\end{tikzpicture}
\end{center}
on $C_{1,2}$.  We then split this into two further subcases.

\subsubsection*{Subcase 1b.1}
Suppose that $\tau_{1}$ leaves some column $j$ invariant. If $\tau_{1}$ acts by the identity on column $j$,
we set $\tau_{2} = \tau_{1}^{\phi^{R}_{2,4}}$ and $\tau_{3} = \tau_{2} \star \phi^{C}_{2,j}$.
Then, setting $\tau_{4} = \tau_{3} \star \phi^{R}_{2,5}$,
$\tau_{4}$ consists of one $3$-cycle and one $5$-cycle.  Thus $\tau_{4}^{3}$ consists of a single $5$-cycle.

Suppose instead $\tau_{1}$ acts nontrivially on column $j$. Let $\tau_{2} = \tau_{1} \star \phi^{C}_{2,j}$. Then $\tau_{2}$ acts by one of the following:

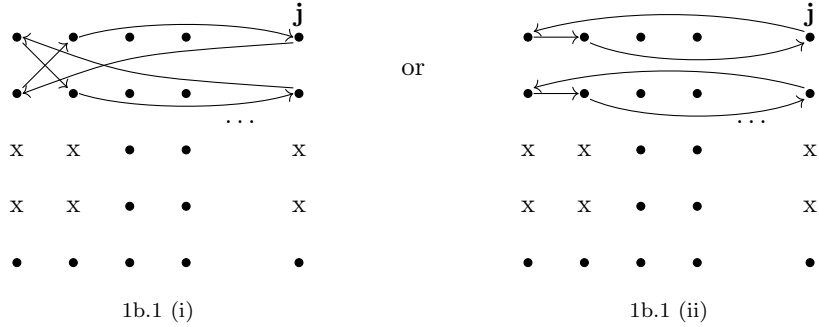
\begin{minipage}[b]{0.5\textwidth}
\centering
\begin{tikzpicture}[scale=0.75]
  \filldraw [black] (0,3) circle (2pt)
                   (1,3) circle (2pt)
                   (2,3) circle (2pt)
                   (0,2) circle (2pt)
                   (1,2) circle (2pt)
                   (2,2) circle (2pt)
                   (0,1) node{x}
                   (1,1) node{x}
                   (2,1) circle (2pt)
                   (0,0) node{x}
                   (1,0) node{x}
                   (2,0) circle (2pt)
                   (0,-1) circle (2pt)
                   (1,-1) circle (2pt)
                   (2,-1) circle (2pt)

                   (3,-1) circle (2pt)
                   (3,0) circle (2pt)
                   (3,1) circle (2pt)
                   (3,2) circle (2pt)
                   (3,3) circle (2pt)

                   (4,1.5) node{\dots}

                   (5,-1) circle (2pt)
                   (5,0) node{x}
                   (5,1) node{x}
                   (5,2) circle (2pt)
                   (5,3) circle (2pt)
                   (5,3.4) node{\textbf{j}}
                  ;
   \draw[->]      (4.9,2.1)  .. controls (2,2.3) .. (0.1,3);
      \draw[->]      (4.9,2.9) .. controls (2,2.7) .. (0.1,2);
    \draw[->]      (0.1,2.9)   -- (0.9,2.1);
     \draw[->]      (0.1,2.1)   -- (0.9,2.9);
\draw[->]      (1.1,3) .. controls (2,3.3) and (4,3.3) .. (4.9,3);
\draw[->]      (1.1,2) .. controls (2,1.7) and (4,1.7) .. (4.9,2);

\end{tikzpicture}
\captionsetup{font=footnotesize}
\captionof*{figure}{1b.1 (i)}
\end{minipage}
\raisebox{1.31in}{or}
\begin{minipage}[b]{0.5\textwidth}
\centering
\begin{tikzpicture}[scale=0.75]
  \filldraw [black] (0,3) circle (2pt)
                   (1,3) circle (2pt)
                   (2,3) circle (2pt)
                   (0,2) circle (2pt)
                   (1,2) circle (2pt)
                   (2,2) circle (2pt)
                   (0,1) node{x}
                   (1,1) node{x}
                   (2,1) circle (2pt)
                   (0,0) node{x}
                   (1,0) node{x}
                   (2,0) circle (2pt)
                   (0,-1) circle (2pt)
                   (1,-1) circle (2pt)
                   (2,-1) circle (2pt)

                   (3,-1) circle (2pt)
                   (3,0) circle (2pt)
                   (3,1) circle (2pt)
                   (3,2) circle (2pt)
                   (3,3) circle (2pt)

                   (4,1.5) node{\dots}

                   (5,-1) circle (2pt)
                   (5,0) node{x}
                   (5,1) node{x}
                   (5,2) circle (2pt)
                   (5,3) circle (2pt)
                   (5,3.4) node{\textbf{j}}
                  ;
   \draw[->]      (0.1,3)   -- (0.9,3);
      \draw[->]      (0.1,2)   -- (0.9,2);
    \draw[->]      (4.9,3.1)   .. controls (3.5,3.5) and (1.5,3.5) .. (0.1,3.1);
     \draw[->]      (4.9,2.1)   .. controls (3.5,2.5) and (1.5,2.5) .. (0.1,2.1);
\draw[->]      (1.1,1.9)   .. controls (2,1.5) and (4,1.5) .. (4.9,1.9);
\draw[->]      (1.1,2.9)   .. controls (2,2.5) and (4,2.5) .. (4.9,2.9);
\end{tikzpicture}
\captionsetup{font=footnotesize}
\captionof*{figure}{1b.1 (ii)}
\end{minipage}

In the first case, setting $\tau_{3} = \tau_{2} \star \phi^{R}_{2,5}$, we have that
$\tau_{3}^{3}$ consists of a single $5$-cycle and the result follows. In the second case, first let $\tau_{3} = \tau_{2} \star \phi^{R}_{2,5}$, then define $\tau_{4} = \tau_{3}^{\phi^{C}_{j,3}}$, and $\tau_{5} =  \tau_{4}^{\phi^{C}_{3,4}}\tau_{4}$.
Finally, letting $\tau_{6} = \tau_{5} \star \phi^{R}_{2,3}$, $\tau_{7} = \tau_{6}^{\phi^{C}_{2,4}}$, and $\tau_{8} = \tau_{7}^{\phi^{R}_{1,3}}$, then $\tau_{8}$ acts by arrangement $(5)$ in Lemma~\ref{lemma:somecycles} and the result follows.

\subsubsection*{Subcase 1b.2}
Suppose $\tau_{1}$ leaves no column invariant. Then we may assume that
$\tau_{1}$ maps points in column $3$ into some columns $j_{1},j_{2}$. We may assume at least one of $j_{1}, j_{2}$ is not equal to $3$, since if not, we are in subcase 1b.1. Thus without loss of generality, we can suppose
that $j_{1} \ne 3$.

 Suppose first that $j_{2} \ne 3$.
Then letting $\tau_{2} = \tau_{1} \star \phi^{C}_{2,3}$, $\tau_{2}$ acts by one of the following arrangements:

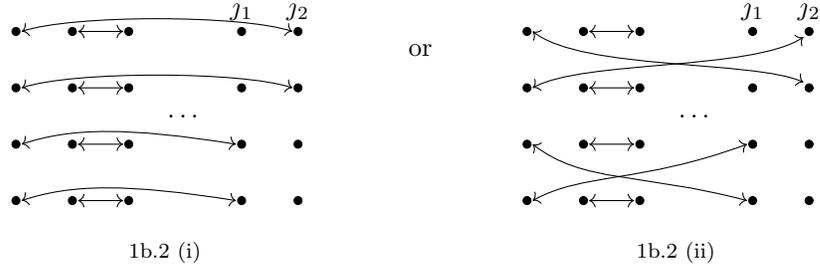
\begin{minipage}[b]{0.5\textwidth}
\centering
\begin{tikzpicture}[scale=0.75]
  \filldraw [black] (0,3) circle (2pt)
                   (1,3) circle (2pt)
                   (2,3) circle (2pt)
                   (4,3) circle (2pt)
                   (5,3) circle (2pt)
                   (0,2) circle (2pt)
                   (1,2) circle (2pt)
                   (2,2) circle (2pt)
                   (4,2) circle (2pt)
                   (5,2) circle (2pt)
                   (0,1) circle (2pt)
                   (1,1) circle (2pt)
                   (2,1) circle (2pt)
                   (4,1) circle (2pt)
                   (5,1) circle (2pt)
                   (0,0) circle (2pt)
                   (1,0) circle (2pt)
                   (2,0) circle (2pt)
                   (4,0) circle (2pt)
                   (5,0) circle (2pt)
                   (3,1.5) node{\dots}
                   (4,3.4) node{${j_{1}}$}
                    (5,3.4) node{${j_{2}}$}
                  ;
   \draw[<->]      (0.1,3)   .. controls (1,3.3) and (4,3.3) .. (4.9,3);
     \draw[<->]      (0.1,2)   .. controls (1,2.3) and (4,2.3) .. (4.9,2);
       \draw[<->]      (0.1,1)   .. controls (1,1.3) and (2,1.3) .. (3.9,1);
         \draw[<->]      (0.1,0)   .. controls (1,0.3) and (2,0.3) .. (3.9,0);

         \draw[<->]      (1.1,3) -- (1.9,3);
         \draw[<->]      (1.1,2) -- (1.9,2);
         \draw[<->]      (1.1,1) -- (1.9,1);
         \draw[<->]      (1.1,0) -- (1.9,0);
\end{tikzpicture}
\captionsetup{font=footnotesize}
\captionof*{figure}{1b.2 (i)}
\end{minipage}
\raisebox{1.1in}{or}
\begin{minipage}[b]{0.5\textwidth}
\centering
\begin{tikzpicture}[scale=0.75]
  \filldraw [black] (0,3) circle (2pt)
                   (1,3) circle (2pt)
                   (2,3) circle (2pt)
                   (4,3) circle (2pt)
                   (5,3) circle (2pt)
                   (0,2) circle (2pt)
                   (1,2) circle (2pt)
                   (2,2) circle (2pt)
                   (4,2) circle (2pt)
                   (5,2) circle (2pt)
                   (0,1) circle (2pt)
                   (1,1) circle (2pt)
                   (2,1) circle (2pt)
                   (4,1) circle (2pt)
                   (5,1) circle (2pt)
                   (0,0) circle (2pt)
                   (1,0) circle (2pt)
                   (2,0) circle (2pt)
                   (4,0) circle (2pt)
                   (5,0) circle (2pt)
                   (3,1.5) node{\dots}
                   (4,3.4) node{${j_{1}}$}
                    (5,3.4) node{${j_{2}}$}
                  ;
   \draw[<->]      (0.1,3)   .. controls (1,2.3) and (4,2.5) .. (4.9,2.1);
     \draw[<->]      (0.1,2)   .. controls (1,2.5) and (4,2.3) .. (4.9,2.9);
       \draw[<->]      (0.1,1)   .. controls (1,0.3) and (2,0.5) .. (3.9,0);
         \draw[<->]      (0.1,0)   .. controls (1,0.5) and (2,0.3) .. (3.9,1);

         \draw[<->]      (1.1,3) -- (1.9,3);
         \draw[<->]      (1.1,2) -- (1.9,2);
         \draw[<->]      (1.1,1) -- (1.9,1);
         \draw[<->]      (1.1,0) -- (1.9,0);

\end{tikzpicture}
\captionsetup{font=footnotesize}
\captionof*{figure}{1b.2 (ii)}
\end{minipage}
Note that while we have drawn these arrangements as
if $j_{1} \ne j_{2}$, we could also have $j_{1} = j_{2}$ and the proof is the same.
Thus for arrangement 1b.2 (i), we set $\tau_{3} = \tau_{2} \star \phi^{R}_{1,5}$, and then $\tau_{4} = \tau_{3}^{\phi^{R}_{2,5}}$, $\tau_{5} = \tau_{4}^{\phi^{C}_{4,j_{2}}}$, $\tau_{6} = \tau_{5}^{\phi^{C}_{1,3}}$, $\tau_{7} =  \tau_{6} \star \phi^{C}_{2,5}$, and $\tau_{8} = \tau_{7}^{\phi^{C}_{3,5}}$, $\tau_{8}$ acts by arrangement $(6)$ of Lemma~\ref{lemma:somecycles}.
For the arrangement 1b.2 (ii), set $\tau_{3} = \tau_{2} \star \phi^{R}_{4,5}$ and then $\tau_{4} = \tau_{3}^{3}$,  $\tau_{5} = \tau_{4}^{\phi^{C}_{1,3}}$, and $\tau_{6} = \tau_{5}^{\phi^{R}_{1,4}\phi^{R}_{2,5}}$. Then $\tau_{6}$ acts by the arrangement $(4)$ in Lemma~\ref{lemma:somecycles}.

Suppose instead that $j_{2} = 3$. Then $\tau_{1}$ fixes two points in column $3$.
Set $\tau_{2} = \tau_{1} \star \phi^{C}_{2,3}$ and $\tau_{3} = \tau_{2}^{3}$. Then setting $\tau_{4} = \tau_{3} \star \left( \phi^{R}_{1,3}\phi^{R}_{2,4} \right)$, we have that $\tau_{4}$ acts by one of the two arrangements 1b.2 (i) or 1b.2 (ii) above, and we proceed as when $j_2 \neq 3$.

\subsubsection*{Case 2}
Suppose $C_{1,2}$ is invariant under $\tau_{1}$ and $\tau_{1}$ acts by the identity on $\etwo{1}{2},\etwo{2}{2},\etwo{3}{1},\etwo{4}{1}$. Then $\tau_{1}$ acts by the arrangement

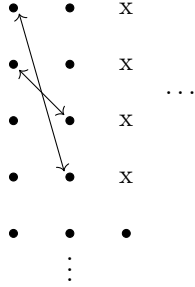
{\centering
\begin{tikzpicture}[scale=0.75]
  \filldraw [black] (0,4) circle (2pt)
                   (0,3) circle (2pt)
                   (0,2) circle (2pt)
                   (0,1) circle (2pt)
                   (0,0) circle (2pt)
                 (1,4) circle (2pt)
                   (1,3) circle (2pt)
                   (1,2) circle (2pt)
                   (1,1) circle (2pt)
                   (1,0) circle (2pt)
                   (2,4) node{x}
                   (2,3) node{x}
                   (2,2) node{x}
                   (2,1) node{x}
                   (2,0) circle (2pt)
                   (1,-0.5) node{\vdots}
                   (3,2.5) node{\dots}
                  ;
                   \draw[<->] (0.1,3.9) -- (0.9,1.1);
                   \draw[<->] (0.1,2.9) -- (0.9,2.1);
\end{tikzpicture}
\captionsetup{font=footnotesize}
\captionof*{figure}{}}
\noindent
Setting $\tau_{2} = \tau_{1} \star \phi^{C}_{1,2}$, we have reduced to Case 1b, and the result follows.

\subsubsection*{Case 3 }
Suppose $C_{1,2}$ is not invariant under $\tau_{1}$.  Again we split the analysis into cases.

\subsubsection*{Subcase 3a}
Suppose $\tau_{1}$ acts nontrivially on $\etwo{4}{1}$, and hence also on $\etwo{3}{1}$. Then $\tau_{1}$ maps $\etwo{4}{1}$ into some column $j$, and by assumption, we must have $j \ne 1,2$. It follows from the setup that $\tau_{1}$ also maps $\etwo{3}{1}$ into column $j$.  We split the analysis into two subcases.

\subsubsection*{Subcase 3a.1 }
Suppose $\tau_{1}$ fixes both $\etwo{1}{2}$ and $\etwo{2}{2}$, so $\tau_{1}$ acts by one of the following arrangements:

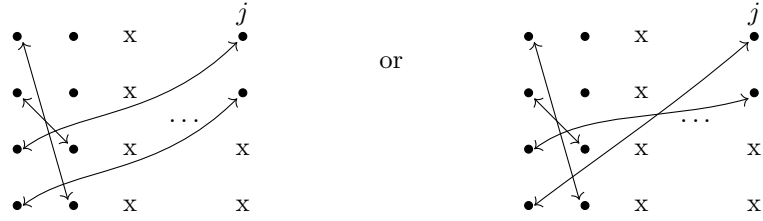
\begin{minipage}[b]{0.5\textwidth}
\centering
\begin{tikzpicture}[scale=0.75]
  \filldraw [black] (0,4) circle (2pt)
                   (0,3) circle (2pt)
                   (0,2) circle (2pt)
                   (0,1) circle (2pt)
                 (1,4) circle (2pt)
                   (1,3) circle (2pt)
                   (1,2) circle (2pt)
                   (1,1) circle (2pt)
                   (2,4) node{x}
                   (2,3) node{x}
                   (2,2) node{x}
                   (2,1) node{x}
                   (3,2.5) node{\dots}
                   (4,4) circle (2pt)
                   (4,3) circle (2pt)
                   (4,2) node{x}
                   (4,1) node{x}
                   (4,4.4) node{${j}$}
                  ;
                   \draw[<->] (0.1,3.9) -- (0.9,1.1);
                   \draw[<->] (0.1,2.9) -- (0.9,2.1);
                   \draw[<->] (0.1,2) .. controls (1,2.7) and (2.5,2.5) .. (3.9,3.9);
                   \draw[<->] (0.1,1) .. controls (1,1.7) and (2.5,1.5) .. (3.9,2.9);

\end{tikzpicture}
\captionsetup{font=footnotesize}
\captionof*{figure}{}
\end{minipage}
\raisebox{1.1in}{or}
\begin{minipage}[b]{0.5\textwidth}
\centering
\begin{tikzpicture}[scale=0.75]
  \filldraw [black] (0,4) circle (2pt)
                   (0,3) circle (2pt)
                   (0,2) circle (2pt)
                   (0,1) circle (2pt)
                 (1,4) circle (2pt)
                   (1,3) circle (2pt)
                   (1,2) circle (2pt)
                   (1,1) circle (2pt)
                   (2,4) node{x}
                   (2,3) node{x}
                   (2,2) node{x}
                   (2,1) node{x}
                   (3,2.5) node{\dots}
                   (4,4) circle (2pt)
                   (4,3) circle (2pt)
                   (4,2) node{x}
                   (4,1) node{x}
                   (4,4.4) node{${j}$}
                  ;
                   \draw[<->] (0.1,3.9) -- (0.9,1.1);
                   \draw[<->] (0.1,2.9) -- (0.9,2.1);
                   \draw[<->] (0.1,2) .. controls (1,2.7) and (2.5,2.5) .. (3.9,2.9);
                   \draw[<->] (0.1,1) .. controls (1,1.7) and (2.5,2.7) .. (3.9,3.9);

\end{tikzpicture}
\captionsetup{font=footnotesize}
\captionof*{figure}{}
\end{minipage}
In either case, setting $\tau_{2} = \tau_{1} \star \phi^{C}_{1,2}$ and $\tau_{3} = \tau_{2} \star \phi^{R}_{4,5}$, we have that  $\tau_{3}$ consists of a $7$-cycle.

\subsubsection*{Subcase 3a2:} Suppose $\tau_{1}$ maps $\etwo{1}{2}$ into column $i$ where $i \ne 1,2$ (it follows from the setup that $\tau_{1}$ also maps $\etwo{2}{2}$ into column $i$). Let $\tau_{2} = \tau_{1} \star \phi^{C}_{1,2}$. Then $\tau_{2}$ acts by one of the following arrangements:

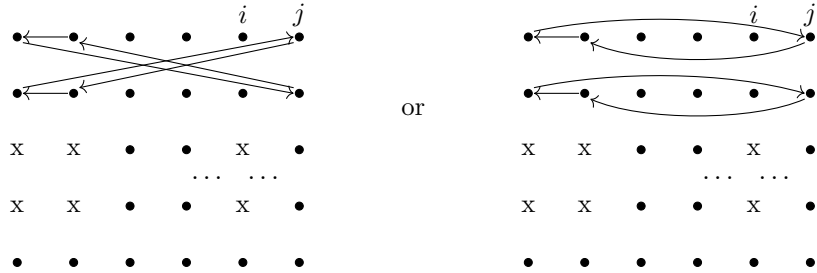
\begin{minipage}[b]{0.5\textwidth}
\centering
\begin{tikzpicture}[scale=0.75]
  \filldraw [black] (0,3) circle (2pt)
                   (1,3) circle (2pt)
                   (2,3) circle (2pt)
                   (0,2) circle (2pt)
                   (1,2) circle (2pt)
                   (2,2) circle (2pt)
                   (0,1) node{x}
                   (1,1) node{x}
                   (2,1) circle (2pt)
                   (0,0) node{x}
                   (1,0) node{x}
                   (2,0) circle (2pt)
                   (0,-1) circle (2pt)
                   (1,-1) circle (2pt)
                   (2,-1) circle (2pt)

                   (3,-1) circle (2pt)
                   (3,0) circle (2pt)
                   (3,1) circle (2pt)
                   (3,2) circle (2pt)
                   (3,3) circle (2pt)

                   (3.4,0.5) node{\dots}
                   (4.4,0.5) node{\dots}

                   (4,3) circle (2pt)
                   (4,2) circle (2pt)
                   (4,1) node{x}
                   (4,0) node{x}
                   (4,-1) circle (2pt)

                   (5,-1) circle (2pt)
                   (5,0) circle (2pt)
                   (5,1) circle (2pt)
                   (5,2) circle (2pt)
                   (5,3) circle (2pt)
                   (5,3.4) node{{$j$}}
                   (4,3.4) node{{$i$}}
                  ;
   \draw[->]      (0.9,3)   -- (0.1,3);
      \draw[->]      (0.9,2)   -- (0.1,2);
    \draw[->]      (0.1,2.9)   -- (4.9,2);
     \draw[->]      (0.1,2.1)   -- (4.9,3);
\draw[->]      (4.9,2.1)   -- (1.1,2.9);
\draw[->]      (4.9,2.9)   -- (1.1,2.1);
\end{tikzpicture}
\captionsetup{font=footnotesize}
\captionof*{figure}{}
\end{minipage}
\raisebox{1.1in}{or}
\begin{minipage}[b]{0.5\textwidth}
\centering

\begin{tikzpicture}[scale=0.75]
  \filldraw [black] (0,3) circle (2pt)
                   (1,3) circle (2pt)
                   (2,3) circle (2pt)
                   (0,2) circle (2pt)
                   (1,2) circle (2pt)
                   (2,2) circle (2pt)
                   (0,1) node{x}
                   (1,1) node{x}
                   (2,1) circle (2pt)
                   (0,0) node{x}
                   (1,0) node{x}
                   (2,0) circle (2pt)
                   (0,-1) circle (2pt)
                   (1,-1) circle (2pt)
                   (2,-1) circle (2pt)

                   (3,-1) circle (2pt)
                   (3,0) circle (2pt)
                   (3,1) circle (2pt)
                   (3,2) circle (2pt)
                   (3,3) circle (2pt)

                   (3.4,0.5) node{\dots}
                   (4.4,0.5) node{\dots}

                   (4,3) circle (2pt)
                   (4,2) circle (2pt)
                   (4,1) node{x}
                   (4,0) node{x}
                   (4,-1) circle (2pt)

                   (5,-1) circle (2pt)
                   (5,0) circle (2pt)
                   (5,1) circle (2pt)
                   (5,2) circle (2pt)
                   (5,3) circle (2pt)
                   (5,3.4) node{{$j$}}
                   (4,3.4) node{{$i$}}
                  ;
   \draw[->]      (0.9,3)   -- (0.1,3);
      \draw[->]      (0.9,2)   -- (0.1,2);
    \draw[->]      (0.1,3.1) .. controls (1.5,3.4) and (3.5,3.4) .. (4.9,3);
     \draw[->]      (0.1,2.1)   .. controls (1.5,2.4) and (3.5,2.4) .. (4.9,2);
\draw[->]      (4.9,1.9)   .. controls (4,1.5) and (2,1.5) .. (1.1,1.9);
\draw[->]      (4.9,2.9)   .. controls (4,2.5) and (2,2.5) .. (1.1,2.9);
\end{tikzpicture}
\captionsetup{font=footnotesize}
\captionof*{figure}{}
\end{minipage}

For the first case, set $\tau_{3} = \tau_{2} \star \phi^{R}_{2,5}$. Then $\tau_{4} = \tau_{3}^{3}$ consists of a single 5-cycle. The second case proceeds analogous to Subcase 1b.1, as illustrated in Figure 1b.1 (ii).

\subsubsection*{Subcase 3b}
Suppose $\tau_{1}$ fixes both $\etwo{4}{1}$ and $\etwo{3}{1}$. Then by assumption, $\tau_{1}$ maps
$\etwo{1}{2}$ and $\etwo{2}{2}$ into some column $j \ne 1,2$. Letting $\tau_{2} = \tau_{1}^{\phi^{C}_{1,2}}$ and $\tau_{3} = \tau_{2}^{\phi^{R}_{1,3}\phi^{R}_{2,4}}$, we are back in Subcase 3a.1.
\end{proof}

We now prove Lemma \ref{lemma:substantialpcycle}
\begin{proof}[Proof of Lemma \ref{lemma:substantialpcycle}]
Since $H$ is substantial, it satisfies both conditions~\eqref{item:startingone} and~\eqref{item:startingtwo} of Definition~\ref{def:substantialsubgroup}. Thus $H$ contains an involution implementing either arrangement $(a)$ or $(b)$ of Definition~\ref{def:substantialsubgroup}. First we note that the subgroup $H$ contains a $p$-cycle for some prime $p < |E^{(2)}|-2$ if and only if the subgroup $\mathfrak{s}^{-1}H\mathfrak{s}$ does. Moreover, $H$ contains an involution implementing arrangement $(b)$ if and only if $\mathfrak{s}^{-1}H\mathfrak{s}$ contains an involution implementing arrangement $(a)$. It follows that it suffices to consider the case that there is
 an involution $\gamma_{3} \in H$ implementing arrangement $(a)$, and we call this arrangement $\mathcal{IC}$:

\begin{center}
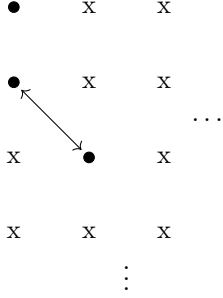

\begin{tikzpicture}
  \filldraw [black] (0,3) circle (2pt)
  (0,2) circle (2pt)
         (1,1) circle (2pt)
                   (1,3) node{x}
                   (2,3) node{x}
                   (2,0) node{x}
                                      (0,1) node{x}
(1,2) node{x}

                                                                       (2,2) node{x}
                                        (1,0) node{x}
                                                             (2,1) node{x}
                  (0,0) node{x}
               (1.5,-.5) node{\vdots}
               (2.6,1.5) node{\dots};
 \draw[<->]      (0.1,1.9)   -- (0.9,1.1);

\end{tikzpicture}
\captionsetup{font=footnotesize}
\captionof{figure}{Arrangement  $\mathcal{IC}$}
\label{def:IC}
\end{center}

Set $\gamma_{4} = \gamma_{3} \star \phi^{R}_{1,2}$.  Then $\gamma_{4}$ acts by the arrangement

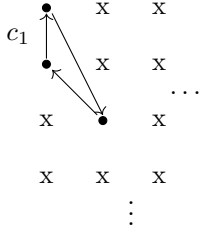
\begin{center}
\begin{tikzpicture}[scale=0.75]
  \filldraw [black] (0,3) circle (2pt)
                   (0,2) circle (2pt)
                   (0,1) node{x}
                   (0,0) node{x}
                   (1,3) node{x}
                   (1,2) node{x}
                   (1,1) circle (2pt)
                   (1,0) node{x}
                   (2,3) node{x}
                   (2,2) node{x}
                   (2,1) node{x}
                   (2,0) node{x}
 (-.5,2.5) node{$c_1$}
                   (1.5,-0.5) node{\vdots}
                   (2.5,1.5) node{\dots}
                  ;
   \draw[->]      (0.9,1.1)   -- (0.1,1.9);
\draw[->]      (0,2.1)   -- (0,2.9);
\draw[->]      (0.1,2.9)   -- (1,1.1);
\end{tikzpicture}
\end{center}
and we label the distinguished $3$-cycle as $c_{1}$.

We claim that any cycle in $\gamma_{4}$ whose support does not intersect $R_{1,2}$ (see~\eqref{def:Rij}) must be a $2$-cycle. To see this, note that $\gamma_{4}^{\phi^{R}_{1,2}} = \gamma_{3}\gamma_{3}^{\phi^{R}_{1,2}} = \gamma_{4}^{-1}$. If $c$ is a cycle in $\gamma_{4}$ whose support does not intersect $R_{1,2}$, then $c^{\phi^{R}_{1,2}} = c$, and it follows that $c$ is equal to its inverse, and hence order two, proving the claim.

Thus, we may choose a large $m_{1} \in \mathbb{N}$ which is relatively prime to $3$ such that $\gamma_{5} = \gamma_{4}^{m_{1}}$ consists of cycles $c_{i}$, $i=1,\ldots,L$, each cycle of length $3^{k_{i}}$ for some $k_{i} \ge 1$, and such that each of these $c_{i}$ has support which intersects $R_{1,2}$. Note that $L \ge 1$ since $\gamma_{5}$ still contains the cycle $c_{1}$ (or its inverse). Define
$$\overline{I} = \{i \in \{1,\ldots,L\} \colon \textnormal{ the support of } c_{i} \textnormal{ is not contained in } R_{1,2}\}.$$

We adopt the following notation: if $c$ is a cycle whose support intersects $E^{(2)} \setminus R_{1,2}$ in exactly one point, we denote this point by $\omega(c)$.

Observe that for each $i \in \overline{I}$, $c_{i}$ has support with at most one point not in $R_{1,2}$ (since each $c_{i}$ satisfies $c_{i}^{\phi^{R}_{1,2}} = c_{i}^{-1}$ and each cycle $c_{i}$ is of odd length). Thus for $i \in \overline{I}$, $\omega(c_{i})$ is well-defined. We also note that
\begin{equation}\label{eqn:wicount}
\sum_{i \in \overline{I}} \left( |c_{i}|-1 \right) \le 2n,
\end{equation}
where $|c_{i}|$ denotes the length of a cycle $c_{i}$. In particular, in the case that all the $c_{i}$'s are $3$-cycles, we have $|\overline{I}| \le n$. We also have $1 \le |\overline{I}|$ since $1 \in \overline{I}$ (the cycle $c_{1}$ has support not contained in $R_{1,2}$).

 We now analyze the cases that arise:

\subsubsection*{Case 1}
Suppose $k_{i} = 1$ for all $i \in \{1,\ldots,L\}$ (recall this means each cycle $c_{i}$ has length $3^{k_{i}}$). Thus $\gamma_{5}$ consists of a collection of $3$-cycles,
and since we have the cycle $c_1$ in the arrangement, it follows that $1 \le |\overline{I}| \le n$. We split into two subcases.

\subsubsection*{Subcase 1a}
Suppose there exists $j \ge 3$ such that $\gamma_{5}$ fixes $\etwo{j}{2}$. Set $\gamma_{6} = \left( \gamma_{5}^{\phi^{R}_{3,j}} \right) \gamma_{5}$. Then $\gamma_{6}$ consists of cycles determined by the following:
\begin{enumerate}
\item
Let $i \in \overline{I}$ be an index such that, writing $\omega(c_{i}) = \etwo{x_{i}}{y_{i}}$, either of the following occur:
\begin{enumerate}
\item
$x_{i} = 3$ and $\etwo{j}{y_{i}} = \omega(c_{l})$ for some $l \in \overline{I}$.
\item
$x_{i} = j$ and $\etwo{3}{y_{i}} = \omega(c_{l})$ for some $l \in \overline{I}$.
\end{enumerate}
Then $\gamma_{6}$ contains a pair of $3$-cycles supported in the union of the supports of $c_{i}$ and $c_{l}$.
\item
Let $i \in \overline{I}$ be an index such that, writing $\omega(c_{i}) = \etwo{x_{i}}{y_{i}}$, either of the following occur:
\begin{enumerate}
\item
$x_{i} = 3$ and $\gamma_{5}$ fixes $\etwo{j}{y_{i}}$.
\item
$x_{i} = j$ and $\gamma_{5}$ fixes $\etwo{3}{y_{i}}$.
\end{enumerate}
Then $\gamma_{6}$ contains a pair of $2$-cycles whose support is contained in the set $\left(c_{i} \cap R_{1,2}\right) \cup \left\{\etwo{3}{y_{i}}, \etwo{j}{y_{i}}\right\}$.
\end{enumerate}

Note that the index $1 \in \overline{I}$ falls into the second case.
Set $\gamma_{7} = \gamma_{6}^{3}$ and set $\gamma_{8} = \gamma_{7}^{\phi^{R}_{4,j}}$. Then either $\gamma_{8}$ or $\gamma_{8}^{\phi^{R}_{1,2}}$ satisfies the hypotheses of Lemma~\ref{lemma:r1234}, completing this case.

\subsubsection*{Subcase 1b} Suppose there is no $j \ge 3$ such that $\gamma_{5}$ fixes $\etwo{j}{2}$. This means
that for all $j \ge 3$, there exists some $i(j) \in \overline{I}$ such that the cycle $c_{i(j)}$ intersects column two,
meaning that $\omega(c_{i(j)})$ lies in column two. Since $|\overline{I}| \le n$, there exist at most two other cycles, call them  $c_{\ell_{1}}, c_{\ell_{2}}$, such that $\omega(c_{\ell_{1}})$ lies in some column
$L_{1}$ and $\omega(c_{\ell_{2}})$ lies in some column $L_{2}$, with $L_{1} \ne 2$ and $L_{2} \ne 2$.
The analysis of this splits into three subcases.

\subsubsection*{Subcase 1b.1 } Suppose the support of $c_{\ell_{1}}$ is not contained entirely in column $L_{1}$,
so the support of $c_{\ell_{1}}$ also intersects some column $L_{3} \ne L_{1}$.
By assumption, $\omega(c_{\ell_{1}})$ lies in column $L_{1}$, so we may write $\omega(c_{\ell_{1}}) = \etwo{x_{1}}{L_{1}}$. Furthermore,
it also follows from our assumptions that there must exist some $j \ge 3$ such that $\gamma_{5}$ fixes $\etwo{j}{L_{1}}$. Setting $\gamma_{6} = \gamma_{5}^{\phi^{C}_{L_{3},1}}$, $\gamma_{7} = \gamma_{6}^{\phi^{R}_{x_{1},3}\phi^{C}_{L_{1},2}}$, it follows that $\gamma_{7}$ acts by the arrangement
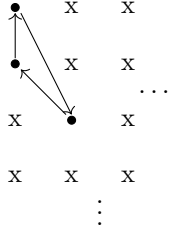
\begin{center}
\begin{tikzpicture}[scale=0.75]
  \filldraw [black] (0,3) circle (2pt)
                   (0,2) circle (2pt)
                   (0,1) node{x}
                   (0,0) node{x}
                   (1,3) node{x}
                   (1,2) node{x}
                   (1,1) circle (2pt)
                   (1,0) node{x}
                   (2,3) node{x}
                   (2,2) node{x}
                   (2,1) node{x}
                   (2,0) node{x}
                   (1.5,-0.5) node{\vdots}
                   (2.5,1.5) node{\dots}
                  ;
   \draw[->]      (0.9,1.1)   -- (0.1,1.9);
\draw[->]      (0,2.1)   -- (0,2.9);
\draw[->]      (0.1,2.9)   -- (1,1.1);
\end{tikzpicture}
\end{center}
so $\gamma_{7}$ again consists of $3$-cycles all of whose supports intersect $R_{1,2}$. Moreover, $\gamma_{7}$ has a distinguished $3$-cycle which matches $c_{1}$ (or its inverse), and also acts by the identity on some $\etwo{j}{2}$ for some $j \ge 3$, so we can apply Subcase 1a.

\subsubsection*{Subcase 1b.2:} If the support of $c_{\ell_{2}}$ is not entirely contained in column $L_{2}$, the argument proceeds exactly as in  Subcase 1b.1.

\subsubsection*{Subcase 1b.3:} The remaining case is that the support of $c_{\ell_{1}}$ is entirely contained in $L_{1}$ and the support of $c_{\ell_{2}}$ is entirely contained in $L_{2}$ (note that if neither $c_{\ell_{1}}$ nor $c_{\ell_{2}}$ exist, their supports are viewed as empty, and so this scenario is covered by this Subcase).
There exists some cycle $c_{m}$ such that $\omega(c_{m}) = \etwo{J_{1}}{2}$ and the support of $c_{m}$ intersects some column $J_{2} \ne 2$. Set $\gamma_{6} = \gamma_{5} \star \phi^{R}_{3,J_{1}}$. Then, after conjugating by $\phi^{R}_{1,2}$ if necessary, $\gamma_{6}$ acts by one of the following:

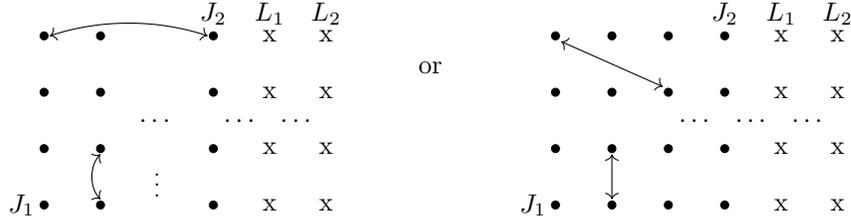
\begin{minipage}[b]{0.5\textwidth}
\centering
\begin{tikzpicture}[scale=0.75]
  \filldraw [black] (0,3) circle (2pt)
                   (1,3) circle (2pt)
                   (0,2) circle (2pt)
                   (1,2) circle (2pt)
                   (0,1) circle (2pt)
                   (1,1) circle (2pt)
                   (0,0) circle (2pt)
                   (1,0) circle (2pt)
                   (3,0) circle (2pt)
                   (3,1) circle (2pt)
                   (3,2) circle (2pt)
                   (3,3) circle (2pt)

                   (3.5,1.5) node{\dots}
                   (4.5,1.5) node{\dots}
                   (2,1.5) node{\dots}
                   (2,0.5) node{\Small{\vdots}}

                   (4,3) node{x}
                   (4,2) node{x}
                   (4,1) node{x}
                   (4,0) node{x}
                   (5,0) node{x}
                   (5,1) node{x}
                   (5,2) node{x}
                   (5,3) node{x}
                   (5,3.4) node{${L_{2}}$}
                   (4,3.4) node{${L_{1}}$}
                   (3,3.4) node{${J_{2}}$}
                   (-0.4,0) node{${J_{1}}$}
                  ;
   \draw[<->]      (0.1,3)  .. controls (1,3.3) and (2,3.3) .. (2.9,3);
      \draw[<->]      (1,0.1) .. controls (0.8,0.3) and (0.8,0.7) .. (1,0.9);
\end{tikzpicture}
\captionsetup{font=footnotesize}
\captionof*{figure}{}
\end{minipage}
\raisebox{1.1in}{or}
\begin{minipage}[b]{0.5\textwidth}
\centering
\begin{tikzpicture}[scale=0.75]
  \filldraw [black] (0,3) circle (2pt)
                   (1,3) circle (2pt)
                   (2,3) circle (2pt)
                   (0,2) circle (2pt)
                   (1,2) circle (2pt)
                   (2,2) circle (2pt)
                   (0,1) circle (2pt)
                   (1,1) circle (2pt)
                   (2,1) circle (2pt)
                   (0,0) circle (2pt)
                   (1,0) circle (2pt)
                   (2,0) circle (2pt)
                   (3,0) circle (2pt)
                   (3,1) circle (2pt)
                   (3,2) circle (2pt)
                   (3,3) circle (2pt)

                   (3.5,1.5) node{\dots}
                   (4.5,1.5) node{\dots}
                   (2.5,1.5) node{\dots}

                   (4,3) node{x}
                   (4,2) node{x}
                   (4,1) node{x}
                   (4,0) node{x}
                   (5,0) node{x}
                   (5,1) node{x}
                   (5,2) node{x}
                   (5,3) node{x}
                   (5,3.4) node{${L_{2}}$}
                   (4,3.4) node{${L_{1}}$}
                   (3,3.4) node{${J_{2}}$}
                   (-0.4,0) node{${J_{1}}$}
                  ;
   \draw[<->]      (0.1,2.9) -- (1.9,2.1);
      \draw[<->]      (1,0.1)   -- (1,0.9);
\end{tikzpicture}
\captionsetup{font=footnotesize}
\captionof*{figure}{}
\end{minipage}
In either case, there exists some column $J_{3}$ on which $\gamma_{6}$ acts by the identity, and setting $\gamma_{7} = \gamma_{6} \star \phi^{C}_{2,J_{3}}$, $\gamma_{7}$ acts by

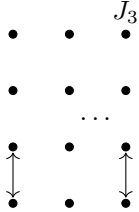
\begin{center}
\begin{tikzpicture}[scale=0.75]
  \filldraw [black] (0,3) circle (2pt)
                   (1,3) circle (2pt)
                   (2,3) circle (2pt)
                   (0,2) circle (2pt)
                   (1,2) circle (2pt)
                   (2,2) circle (2pt)
                   (0,1) circle (2pt)
                   (1,1) circle (2pt)
                   (2,1) circle (2pt)
                   (0,0) circle (2pt)
                   (1,0) circle (2pt)
                   (2,0) circle (2pt)
                   (1.5,1.5) node{\dots}
                   (2,3.4) node{${J_{3}}$}
                  ;
       \draw[<->]      (0,0.1)   -- (0,0.9);
            \draw[<->]      (2,0.1)   -- (2,0.9);
\end{tikzpicture}
\end{center}
We may then conjugate $\gamma_{7}$ to move this pair of $2$-cycles into case $(1)$ of Lemma~\ref{lemma:somecycles}.

\subsubsection*{Case 2}
Suppose there exists a cycle $c_{i}$ with $k_{i} \ge 2$ (recall this means the cycle $c_{i}$ has length $3^{k_{i}}$
and note that this $i$ may not be in $\overline{I}$). Let $k^{\prime} = \max_{i}k_{i}$, let $I_{1} \subset \{1,\ldots,L\}$ be the set of indices for which $k_{i} = k^{\prime}$, and set $\gamma_{6} = \gamma_{5}^{3^{k^{\prime}-1}}$. Then $\gamma_{6}$ is order $3$ and contains $3^{k^{\prime}-1}|I_{1}|$ 3-cycles $d_{i}$, each of whose support intersect $R_{1,2}$.  We proceed by analyzing two subcases.

\subsubsection*{Subcase 2a}
Suppose every $d_{i}$ has support entirely contained in $R_{1,2}$. Note that we still have $\gamma_{6}^{\phi^{R}_{1,2}} = \gamma_{6}^{-1}$. As a result, any cycle $d_{i} = (z_{1},z_{2},z_{3})$ in $\gamma_{6}$ has a companion cycle $d_{i^{\prime}} = (z_{3} + 1 \textnormal{ mod } 2, z_{2} + 1 \textnormal{ mod } 2,z_{1} + 1 \textnormal{ mod } 2)$ in $\gamma_{6}$. Moreover, for each cycle $d_{i} = (z_{1},z_{2},z_{3})$ in $\gamma_{6}$, we must have
$z_{1},z_{2}$, and $z_{3}$ lying in distinct columns. We further note $\gamma_{6}$ acts by the identity on $\etwo{1}{1}, \etwo{2}{1}$. Among all the cycles $d_{i}$, there are two companion cycles,
call them $d_{j}$ and $d_{j^{\prime}}$, whose supports intersect a column, say column $J$, which is furthest to the left.
Thus we have that  $J < J^{\prime}$ for any other column $J^{\prime}$ hit by cycles in the list $d_{i}$.
Consider $\gamma_{7} = \gamma_{6}^{\phi^{C}_{1,J}}\gamma_{6}$. Then $\gamma_{7}$ consists of four $2$-cycles, two of which intersect column one; call these $e_{1}, e_{2}$.
Due to the structure of the companion cycles $d_{j}, d_{j^{\prime}}$,
it follows that $e_{1}$ and $e_{2}$ also intersect some distinct columns $J_{1} < J_{2}$.
Choose a column $J_{3} \ne 1,J_{1},J_{2}$, and set $\gamma_{8} = \gamma_{7} \star \phi^{C}_{1,J_{3}}$. Then $\gamma_{8}$ consists of two $3$-cycles, $e_{1}^{\prime}, e_{2}^{\prime}$, whose support columns consist of $1,J_{1},J_{3}$ and $1,J_{2},J_{3}$, respectively. Since $n$ is large ($n \ge 7$), we may find yet another column $J_{4} \ne 1,J_{1},J_{2},J_{3}$, and let $\gamma_{9} = \gamma_{8}^{\phi^{C}_{J_{1},J_{4}}}\gamma_{8}$. Then $\gamma_{9}$ consists of two $2$-cycles, whose supports intersect four distinct columns. Choosing again a new column $J_{5}$, we have $\gamma_{9}^{\phi^{C}_{J_{4},J_{5}}}\gamma_{9}$ consists of only one $3$-cycle, and we are done.

\subsubsection*{Subcase 2b} Suppose there exists a cycle $d_{i}$ whose support is not contained in $R_{1,2}$. Then $I_{1} \cap \overline{I} \ne \emptyset$ and we can consider the nonempty set of indices
$$\overline{J} = \{j \colon \textnormal{ the support of } d_{j} \textnormal{ is not contained in } R_{1,2}\}.$$
Recall $\gamma_{6} = \gamma_{5}^{3^{k^{\prime}-1}}$ and that $2 \leq k^{\prime}  = \max_{i}k_{i}$.
Since each cycle in $\gamma_{5}$ has at most one point not in $R_{1,2}$, each cycle of length $3^{k^{\prime}}$ in $\gamma_{5}$ contributes one cycle of length $3$ in $\gamma_{6}$ whose support is not contained in $R_{1,2}$.
Thus it follows that $|\overline{J}| = |I_{1} \cap \overline{I}|$, and that, since the support of the cycles of length $3^{k^{\prime}}$ in $\gamma_{5}$ have at least 8 points in $R_{1,2}$, we must have
\begin{equation}\label{eqn:overlineJbound}
|\overline{J}| \le \frac{n}{4}.
\end{equation}
Thus the collection $\{\omega(d_{j}) \colon  j \in \overline{J}\}$ has at most $\frac{n}{4}$ points, and we may choose some $k \in \overline{J}$ such that, upon writing $\omega(d_{k}) = \etwo{x_{k}}{y_{k}}$, there exists some $3 \le \ell \le n$ such that $\gamma_{6}$ fixes the point $\etwo{\ell}{y_{k}}$. Consider
$$\gamma_{7} = \gamma_{6}^{\phi^{R}_{x_{k},\ell}}\gamma_{6}.$$
Then $\gamma_{7}$ contains cycles determined by the following:
\begin{enumerate}
\item
\label{cyclce:one}
A pair of $3$-cycles corresponding to each (un-ordered) pair of indices $j_{1},j_{2} \in \overline{J}$ such that $\omega(d_{j_{1}}) \in R_{x_{k}}, \omega(d_{j_{2}}) \in R_{\ell}$, and $\omega(d_{j_{1}}), \omega(d_{j_{2}})$ lie in the same column.
\item
\label{cyclce:two}
A pair of $2$-cycles corresponding to each index $j \in \overline{J}$ such that either $\etwo{x_{k}}{y_{j}} = \omega(d_{j}) \in R_{x_{k}}$ and $\gamma_{6}$ fixes $\etwo{\ell}{y_{j}}$, or $\etwo{\ell}{y_{j}} = \omega_{d_{j}} \in R_{\ell}$ and $\gamma_{6}$ fixes $\etwo{x_{k}}{y_{j}}$.
\end{enumerate}
The $2$-cycles which arise in case~\eqref{cyclce:two} have support intersecting rows $1,2, x_{k},\ell$. Moreover since $k \in \overline{J}$ satisfies case~\eqref{cyclce:two}, we have at least one pair of $2$-cycles; suppose this pair has support contained in columns $y_{k}, y_{k}^{\prime}$ (note we could have $y_{k} = y_{k}^{\prime}$). Setting $\gamma_{8} = \gamma_{7}^{3}$, we have that
$\gamma_{8}$ consists of only pairs of $2$-cycles corresponding to each $j \in \overline{J}$ satisfying
case~\eqref{cyclce:two}. Setting $\gamma_{9} = \gamma_{8}^{\phi^{R}_{x_{k},3}\phi^{R}_{\ell,4}}$, $\gamma_{9}$ is an involution satisfying the first three  conditions of Lemma~\ref{lemma:r1234}.

Now by~\eqref{eqn:overlineJbound}, there exists a column $F_{1}$ such that $\gamma_{9}$ acts by the identity on the column $F_{1}$. Suppose $y_{k} = y_{k}^{\prime}$. Then $\left( \gamma_{9}^{\phi^{C}_{y_{k},F_{1}}} \right)^{-1} \gamma_{9}$ consists of two pairs of $2$-cycles, supported in rows $1,2,3,4$,; upon conjugating and moving these cycles if necessary, we can apply Lemma~\ref{lemma:somecycles}. If $y_{k} \ne y_{k}^{\prime}$, then setting $\gamma_{10} = \gamma_{9}^{\phi^{C}_{1,y_{k}}\phi^{C}_{2,y_{k}^{\prime}}}$, and if necessary replacing $\gamma_{10}$ with $\gamma_{11} = \gamma_{10}^{\phi^{R}_{1,2}}$, $\gamma_{10}$ is an involution satisfying all four conditions of Lemma~\ref{lemma:r1234}, and the result follows.
\end{proof}

We have now assembled all the ingredients to complete the proof of the technical lemma:
\begin{proof}[Proof of Lemma~\ref{lemma:obtainment}]
Our goal is to show that at least one of
\begin{enumerate}
\item
$H_{1} = \sym(E^{(2)})$ and $H_{2} = \sym(E^{(2)})$.
\item
$H_{1} = \alt(E^{(2)})$ and $H_{2} = \alt(E^{(2)})$
\end{enumerate}
holds. Since both $N_{1}$ and $N_{2}$ are trivial by assumption, and $H_{1}$ and $H_{2}$ are isomorphic by assumption, it suffices to show that at least one of $H_{1}$ or $H_{2}$ is either $\sym(E^{(2)})$ or $\alt(E^{(2)})$. By Jordan's Theorem, it then suffices to show that at least one of $H_{1}, H_{2}$ is primitive and also contains a $p$-cycle for some prime $p<|E^{(2)}|-2$. By Lemma~\ref{lemma:startingarrangement}, at least one of $H_{1}$ or $H_{2}$ is substantial. Since both $H_{1}$ and $H_{2}$ contain $P$, combining Lemma~\ref{lemma:primitive} and Lemma~\ref{lemma:substantialpcycle} gives that at least one of $H_{1}$ or $H_{2}$ satisfies the hypotheses of Jordan's Theorem, and hence is either $\sym(E^{(2)})$ or $\alt(E^{(2)})$, as desired.
\end{proof}

\end{document}